\documentclass[bj,authoryear,noshowframe]{imsart}

\providecommand{\xishape}{\itshape}

\makeatletter
\def\journal@name{}
\makeatother

\usepackage{mathtools}
\usepackage{bbm}
\usepackage{booktabs}
\usepackage{caption}
\usepackage{subcaption}
\usepackage{float}
\usepackage{hyperref}
\hypersetup{hidelinks}

\startlocaldefs
\DeclareMathOperator{\Var}{Var}
\DeclareMathOperator{\Cov}{Cov}
\DeclareMathOperator{\Corr}{Corr}
\DeclareMathOperator{\tr}{tr}
\DeclareMathOperator{\MSE}{MSE}

\theoremstyle{plain}
\newtheorem{theorem}{Theorem}[section]
\newtheorem{lemma}[theorem]{Lemma}
\newtheorem{proposition}[theorem]{Proposition}
\newtheorem{corollary}[theorem]{Corollary}

\theoremstyle{definition}
\newtheorem{definition}[theorem]{Definition}
\newtheorem{assumption}[theorem]{Assumption}
\newtheorem{remark}[theorem]{Remark}

\endlocaldefs

\begin{document}

\begin{frontmatter}

\title{Finite-Sample Bounds for Expected Signature Estimation under Weak Dependence}
\runtitle{Finite-sample bounds for expected signatures}
\runauthor{B.\ Schenck}

\begin{aug}
\author{\inits{B.}\fnms{Bryson}~\snm{Schenck}\ead[label=e1]{brysondale@gmail.com}}
\address{Independent Researcher\\
New York, NY, USA\\
\printead{e1}}
\end{aug}

\begin{abstract}
The expected signature uniquely determines the law of a random rough path under a moment-growth condition, yet finite-sample bounds for estimating its truncations from a single long dependent trajectory remain unavailable. We study a strictly stationary stochastic process equipped with a geometric rough-path lift, observed in non-overlapping blocks of equally-spaced samples, and prove a non-asymptotic mean-squared error (MSE) bound for the block-averaging estimator of its truncated expected signature. Under moment and stationarity assumptions together with a direct covariance-decay condition on block signatures---strictly weaker than $\alpha$-mixing and applicable to long-range-dependent processes---the error separates into a discretization term and a fluctuation term, with rates determined respectively by path regularity and dependence strength. A levelwise rough-factorial variance analysis keeps finite-truncation constants explicit and yields an optimal allocation rule under a fixed observation budget. We verify the assumptions for independent-coordinate fractional Ornstein--Uhlenbeck processes in three regimes: short-range (Hurst $1/4<H<1/2$), semimartingale ($H=1/2$), and long-range ($H>1/2$); in all three, the block-signature covariance is summable, so the fluctuation term decays at the same rate as in the independent-block case, even under long memory at $H>1/2$. Monte Carlo experiments show empirical slopes steeper than the guaranteed upper-bound rates.
\end{abstract}



\end{frontmatter}

% 1. Introduction
\section{Introduction}
\label{sec:introduction}

We establish non-asymptotic bounds for estimating the expected signature from a single long dependent trajectory. The expected signature is the mean of the signature transform~\citep{Lyons1998, FrizVictoir2010}, an iterated-integral feature map that, under a moment-growth condition, characterizes the law of a random rough path~\citep{ChevyrevLyons2016, Chevyrev2022Signature}, with pathwise signatures unique up to tree-like equivalence~\citep{HamblyLyons2010, BoedihardjoGengLyonsYang2016}. Signature methods underpin applications in machine learning~\citep{ChevyrevKormilitzin2016, Kidger2020Neural} and financial model calibration~\citep{Cuchiero2022Signature}.

In many applications---finance, physics, and time series analysis---one observes a single long realization of a dependent process and must estimate its expected signature under a fixed observation budget. A practical design question is therefore unavoidable: how should that budget be split between mesh refinement, which reduces discretization bias, and block count, which reduces estimation variance? Without finite-sample bounds, this allocation remains heuristic. Although asymptotic consistency and normality for block-averaging estimators under temporal dependence have been established~\citep{Lucchese2025Learning}, explicit non-asymptotic error bounds quantifying how path roughness and dependence strength jointly determine the achievable error have remained unavailable. The gap is especially important for long-range-dependent processes such as stationary fractional Ornstein--Uhlenbeck (fOU) with $H > 1/2$, which lie outside standard $\alpha$-mixing-based analyses~\citep{GehringerLi2022}.

Consider a trajectory of a geometric $p$-rough path (a continuous path enhanced with iterated-integral data sufficient to define a pathwise integration theory; see Section~\ref{sec:mathematical_framework}), partitioned into $K$ non-overlapping blocks of length $\Delta$, each observed at $n$ equally-spaced times. Under strict stationarity and a \emph{weak dependence} condition---realized either as classical $\alpha$-mixing or, more permissively, as a direct covariance-decay bound on block signatures (Assumptions~\ref{ass:A} and~\ref{ass:Aprime} respectively, both falling within the broader weak-dependence framework that unifies mixing and covariance-style decay conditions for functionals of stationary sequences~\citealp{DoukhanLouhichi1999, Dedecker2007Weak})---Theorem~\ref{thm:main_result} yields
\begin{equation}
\label{eq:intro_bound}
\MSE \;\le\; \frac{C_{\mathrm{bias}}(\Delta, M)}{(n-1)^\gamma} \;+\; \frac{V(\Delta)}{K^\eta},
\end{equation}
in the summable regime, with an arbitrarily small $\varepsilon$-loss under non-summable $\alpha$-mixing and a logarithmic factor at the direct-covariance borderline (Theorem~\ref{thm:main_result}). Here $\gamma$ tracks path regularity and $\eta$ tracks dependence strength; the two terms move in opposite directions in $n$ and $K$, and the optimal allocation under a fixed budget $N = Kn$ is determined by their interaction.

The closest existing work is~\citet{Lucchese2025Learning}, which proves asymptotic consistency and a $K^{-1/2}$ central limit theorem for expected-signature estimators in a flexible framework accommodating both independent path samples and the single-trajectory ``chop'' scheme, irregular observation meshes, and a martingale correction. The gap addressed here is non-asymptotic. Our covariance-decay condition is strictly weaker than $\alpha$-mixing, and Section~\ref{sec:verification} verifies it directly in these long-range-dependent regimes. The MSE bound is explicit in $K$ and $n$, with a levelwise variance decomposition whose rough-factorial damping yields finite constants for each fixed truncation level $M$. An explicit allocation rule then minimizes the upper bound under a fixed observation budget $N = Kn$.

The hypotheses are verified for independent-coordinate fOU spanning the short-range ($1/4 < H < 1/2$), semimartingale ($H = 1/2$), and long-range ($H > 1/2$) regimes (Section~\ref{sec:verification}); the discretization rate $\gamma$ follows the Hurst-dependent regimes of Proposition~\ref{prop:gamma_values}. The verification yields $\eta = 1$ in all three regimes: block signatures are translation-invariant functionals of increments, so the long memory of the fOU \emph{level} --- which obstructs $\alpha$-mixing for $H > 1/2$ --- never transfers to the block-signature covariance, which decays summably (at rate $\nu = 4-2H > 1$ for $H \neq 1/2$, exponentially at $H = 1/2$; Proposition~\ref{prop:fOU_covariance_decay}). Ground-truth expected signatures are computed via the regularized Wick formulas of~\citet{CassFerrucci2024}, with the $H > 1/4$ threshold sharp for level-2 iterated-integral convergence~\citep[Thm.~2]{CoutinQian2002}. Monte Carlo experiments at $H \in \{0.40, 0.50, 0.60\}$ show empirical convergence rates strictly faster than the upper-bound rates across all three regimes (Section~\ref{sec:numerical_results}).

The remainder of the paper establishes the framework and assumptions (Section~\ref{sec:mathematical_framework}), states the main result with optimal allocation formulas (Section~\ref{sec:convergence}), verifies the assumptions for fOU (Section~\ref{sec:verification}), and illustrates the bounds numerically (Section~\ref{sec:numerical_results}). Limitations and open questions are discussed in Section~\ref{sec:discussion}. Proofs of the main theorem appear in Appendix~\ref{app:proof}; the discretization rates of Proposition~\ref{prop:gamma_values} are verified in Appendix~\ref{app:gamma}; exact expected signature formulas for fOU are derived in Appendix~\ref{app:expected_signature_fOU}.

% 2. Preliminaries
\section{Preliminaries}
\label{sec:mathematical_framework}

We fix the tensor and rough-path notation used in Theorem~\ref{thm:main_result}, then state the standing assumptions. The assumptions isolate four distinct sources of difficulty in signature estimation: path regularity (governing moments), temporal dependence (governing variance decay), stationarity (enabling block averaging), and discretization quality (governing bias).

\subsection{Tensor Algebra}

Equip $\mathbb{R}^d$ with the Euclidean inner product. The truncated tensor algebra of degree $M \ge 1$,
\begin{equation}
T^{(M)}(\mathbb{R}^d) = \bigoplus_{m=0}^{M} (\mathbb{R}^d)^{\otimes m},
\end{equation}
is endowed with the Hilbert--Schmidt inner product, so that
\begin{equation}
\|a\|_{\mathrm{HS}}^2 = \sum_{m=0}^{M} \|a^{(m)}\|^2,
\end{equation}
where $a^{(m)}$ denotes the degree-$m$ component. For a $T^{(M)}(\mathbb{R}^d)$-valued random variable $Y$, write $\|Y\|_{L^2} = (\mathbb{E}[\|Y\|_{\mathrm{HS}}^2])^{1/2}$.

\subsection{Rough Paths}

Fix roughness $p \ge 1$ and truncation $M \ge \lfloor p \rfloor$. Let $X: [0, \infty) \to \mathbb{R}^d$ be continuous, admitting a lift to a random geometric $p$-rough path $\mathbf{X}$ in $G^{\lfloor p \rfloor}(\mathbb{R}^d)$; see~\citep{FrizVictoir2010, FrizHairer2020}. The lift determines the signature to all degrees $M \ge \lfloor p \rfloor$ via the Lyons extension theorem~\citep[Thm.~9.5]{FrizVictoir2010}. The signature of $\mathbf{X}$ over $[s,t]$ is
\begin{equation}
S(\mathbf{X})_{s,t} = \Bigl( 1,\, \int_{s<u<t} dX_u,\, \int_{s<u_1<u_2<t} dX_{u_1} \otimes dX_{u_2},\, \ldots \Bigr),
\end{equation}
where the iterated integrals are defined via rough-path integration (e.g.,~\citealp[Thm.~4.10]{FrizHairer2020} for $p \in [2,3)$). Write $S^{(M)}(\mathbf{X})_{s,t} \in T^{(M)}(\mathbb{R}^d)$ for the truncation to degree $M$ and $\beta = 1/p$ for the H\"older exponent. We write $(a)! := \Gamma(a+1)$ for non-integer arguments, and $\Lambda_p > 0$ for the Lyons rough-factorial constant of~\citet[Thm.~2.2.1]{Lyons1998}. Throughout, the lift $\mathbf{X}_{s,t}$ is assumed measurable with respect to $\sigma(X_u : u \in [s,t])$; for canonical Gaussian lifts this is automatic by~\citep[Thm.~15.33(iv)]{FrizVictoir2010}.

\subsection{Assumptions}
\label{subsec:assumptions}

\begin{assumption}[Moment control]
\label{ass:M}
For every $r \ge 1$ and every fixed $T < \infty$, there exists $C_r(T) < \infty$ such that $\mathbb{E}[\|\mathbf{X}\|_{p\text{-var};[s,t]}^{r}] \le C_r(T) (t-s)^{r\beta}$ for all $0 \le s \le t \le T$. Moreover, $\mathbb{E}\|X_0\|^2 < \infty$.
\end{assumption}

\begin{assumption}[$\alpha$-mixing]
\label{ass:A}
The strong mixing coefficient~\citep{Bradley2005}
\[
\alpha_{X}(t) \;=\; \sup_{s \ge 0}\; \sup\bigl\{|\mathbb{P}(A \cap B) - \mathbb{P}(A)\mathbb{P}(B)| : A \in \mathcal{F}_{\le s},\; B \in \mathcal{F}_{\ge s+t}\bigr\},
\]
where $\mathcal{F}_{\le s} = \sigma(X_u : u \le s)$ and $\mathcal{F}_{\ge t} = \sigma(X_u : u \ge t)$, satisfies one of:
\begin{enumerate}
\item[\emph{(E)}] \textbf{Exponential mixing:} $\alpha_{X}(t) \le C_X e^{-\lambda_X t}$ for some $C_X, \lambda_X > 0$.
\item[\emph{(P)}] \textbf{Polynomial mixing:} $\alpha_{X}(t) \le C_X (1+t)^{-\nu}$ for some $C_X > 0$ and $\nu > 0$.
\end{enumerate}
\end{assumption}

For processes that need not be strongly mixing (e.g., long-memory Gaussian processes), we introduce an alternative to Assumption~\ref{ass:A}: a direct condition on the block-signature covariances. Here and below, $I_k := [(k-1)\Delta, k\Delta]$ denotes a generic block of length $\Delta$.

\begin{assumption}[Block-signature covariance decay]
\label{ass:Aprime}
Let $W_k^{(m)} := S^{(m)}(\mathbf{X})_{I_k} - \mathbb{E}[S^{(m)}(\mathbf{X})_{I_k}]$ denote the centered level-$m$ block signature over $I_k$, and define the lag-$h$ autocovariance $c_{h,m} := \mathbb{E}[\langle W_1^{(m)}, W_{1+h}^{(m)} \rangle_{\mathrm{HS}}]$ for $h \ge 0$ (independent of the base index by stationarity; Assumption~\ref{ass:S} below). For each $1 \le m \le M$, one of the following holds:
\begin{enumerate}
\item[\emph{(E$'$)}] \textbf{Exponential covariance decay:} there exist $\widetilde{C}_m, \lambda > 0$ such that
\[
|c_{h,m}| \;\le\; \frac{d^m\widetilde{C}_m}{\bigl(\Lambda_p(m/p)!\bigr)^2}\Delta^{2m\beta} e^{-\lambda h \Delta} \quad \text{for all } h \ge 1.
\]
\item[\emph{(P$'$)}] \textbf{Polynomial covariance decay:} there exist $\widetilde{C}_m > 0$ and $\nu > 0$ such that
\[
|c_{h,m}| \;\le\; \frac{d^m\widetilde{C}_m}{\bigl(\Lambda_p(m/p)!\bigr)^2}\Delta^{2m\beta}(1+h)^{-\nu} \quad \text{for all } h \ge 1.
\]
\end{enumerate}
The exponent $\nu$ (resp.\ $\lambda$) is assumed uniform across levels $1 \le m \le M$, while the constants $\widetilde{C}_m$ may depend on $m$ and on the block length $\Delta$.
\end{assumption}

\begin{remark}
\label{rem:A_implies_Aprime}
Under the moment control of Assumption~\ref{ass:M}, Assumption~\ref{ass:A} implies Assumption~\ref{ass:Aprime} via Rio's covariance inequality (Lemma~\ref{lem:hilbert_davydov}): under (E), (E$'$) holds for any $\lambda' < \lambda_X$, and under (P) with exponent $\nu$, (P$'$) holds for any $\nu' < \nu$, with the additional dimension factor and the mixing and moment constants absorbed into the level-dependent constant $\widetilde{C}_m$. Both losses are arbitrarily small (Lemma~\ref{lem:hilbert_davydov}, Proposition~\ref{prop:poly_mixing_rate}). Assumption~\ref{ass:Aprime} is strictly weaker than Assumption~\ref{ass:A} and accommodates non-mixing examples (Section~\ref{sec:verification}).
\end{remark}

\begin{assumption}[Strict stationarity]
\label{ass:S}
The base process $X$ is strictly stationary: for all $\tau \ge 0$ and all $0 \le t_1 < \cdots < t_k$,
\[
(X_{\tau+t_1}, \ldots, X_{\tau+t_k}) \stackrel{d}{=} (X_{t_1}, \ldots, X_{t_k}).
\]
In particular, $X_\tau \stackrel{d}{=} X_0$ for all $\tau \ge 0$. We further require that the lift inherit this invariance: $\vartheta_\tau\mathbf{X} \stackrel{d}{=} \mathbf{X}$ for all $\tau \ge 0$, where the time-shift $\vartheta_\tau$ acts on rough paths by $(\vartheta_\tau\mathbf{X})_{s,t} := \mathbf{X}_{\tau+s,\, \tau+t}$ (we reserve $\vartheta$ for this shift to avoid collision with the fOU mean-reversion rate $\theta$ of Section~\ref{sec:verification}). For the canonical lifts used in this paper, lift stationarity is automatic: the lift is the limit (in probability) of the lifted piecewise-linear interpolations of $X$ (naturality,~\citealp[Thm.~15.33(iv)]{FrizVictoir2010}), hence a shift-equivariant measurable functional of the path.
\end{assumption}

\begin{assumption}[Discretization rate]
\label{ass:gamma}
There exist $C_{\mathrm{approx}}(\Delta, M) > 0$ and $\gamma > 0$ such that for every $n \ge 2$, the piecewise-linear interpolation $X^{(n)}$ of $n$ equally-spaced observations at mesh $\delta = \Delta/(n-1)$, with $\mathbf{X}^{(n)}$ denoting its canonical rough-path lift (unique since piecewise-linear paths have bounded variation),
\begin{equation}
\label{eq:approx_rate}
\mathbb{E}\bigl[\|S^{(M)}(\mathbf{X})_{0,\Delta} - S^{(M)}(\mathbf{X}^{(n)})_{0,\Delta}\|_{\mathrm{HS}}^2\bigr] \le C_{\mathrm{approx}}(\Delta, M)\, \delta^\gamma.
\end{equation}
\end{assumption}

The exponent $\gamma$ reflects path regularity:

\begin{proposition}
\label{prop:gamma_values}
Assumption~\ref{ass:gamma} holds with:

\emph{(i)} $\gamma = 1 - \varepsilon$ for any $\varepsilon > 0$ at truncation level $M$, with $\gamma = 1$ attained exactly at signature level~2, for centered continuous Gaussian semimartingales $X = X_0 + \mathsf{N} + A$ with independent coordinates, bounded martingale quadratic-variation density ($d\langle \mathsf{N}^j\rangle/dt \le c_\mathsf{N}^2$ a.s.), $L^2$-bounded drift-density ($\mathbb{E}[(dA^j/du)^2] \le c_A^2$), and each drift coordinate $A^j$ adapted to $\sigma(X^j_s : s \le u)$. This class contains stationary Ornstein--Uhlenbeck and fOU at $H = 1/2$ (Appendix~\ref{app:gamma_semi}; cf.~\citealp{FrizRiedel2014, WongZakai1965}).

\emph{(ii)} $\gamma = 4H-1-\varepsilon$ for any $\varepsilon > 0$ when $H \in (1/4, 1/2)$ (with $\gamma = 4H-1$ attained exactly at signature level~2), and $\gamma = 2H$ exactly when $H \in (1/2, 1)$, for independent-coordinate fractional Ornstein--Uhlenbeck with Hurst $H \in (1/4, 1) \setminus \{1/2\}$ (see Appendix~\ref{app:gamma}). At $H = 1/2$, fOU coincides with stationary Ornstein--Uhlenbeck and case~(i) applies, consistent with the boundary values $4H-1 = 2H = 1$.

\emph{(iii)} $\gamma = 2$ for bounded variation paths satisfying Assumption~\ref{ass:M} with $p = 1$ and $\beta = 1$ (Appendix~\ref{app:gamma_bv}).
\end{proposition}

\begin{proof}
See Appendix~\ref{app:gamma}.
\end{proof}

\begin{remark}
Proposition~\ref{prop:gamma_values}(i) and the rough regime of case~(ii) each admit an arbitrarily small $\varepsilon$-loss. We write $\varepsilon_0$ for the loss in case~(i) ($\gamma = 1 - \varepsilon_0$) and $\varepsilon_1$ for the loss in the rough regime of case~(ii) ($\gamma = 4H-1-\varepsilon_1$ for $H \in (1/4,1/2)$); these labels are used in sample-complexity computations (Corollary~\ref{cor:sample_complexity}). The exponents $\gamma = 2H$ for $H \in (1/2,1)$ in case~(ii) and $\gamma = 2$ in case~(iii) carry no $\varepsilon$-loss. When discussing heuristic sample-allocation exponents we occasionally suppress the arbitrarily small losses, but all rigorous statements below retain them where they occur.
\end{remark}

% 3. Main Results
\section{Main Results}
\label{sec:convergence}

Consider a trajectory over $T = K\Delta$, with blocks $I_k = [(k-1)\Delta, k\Delta]$ for $k = 1, \ldots, K$ and each block observed at $n \ge 2$ equally-spaced times with mesh $\delta = \Delta/(n-1)$. Let $X^{(n)}_k$ denote the piecewise-linear interpolation in block $I_k$, and $\mathbf{X}^{(n)}_k$ its rough-path lift. Define the true and observable block signatures:
\begin{equation}
Z_k := S^{(M)}(\mathbf{X})_{I_k}, \qquad Y_k := S^{(M)}(\mathbf{X}^{(n)}_k)_{I_k}.
\end{equation}

\begin{definition}
\label{def:estimator}
The block-averaging estimator is $\widehat{\Sigma}_{K} := K^{-1} \sum_{k=1}^{K} Y_k$, targeting $\mathbb{E}[S^{(M)}] := \mathbb{E}[Z_1]$.
\end{definition}

\begin{lemma}
\label{lem:stationarity}
Under Assumption~\ref{ass:S}, $(Z_k)_{k \ge 1}$ and $(Y_k)_{k \ge 1}$ are each strictly stationary.
\end{lemma}

\begin{proof}
The truncated signature map $\mathbf{x} \mapsto S^{(M)}(\mathbf{x})$ is the Lyons lift, uniformly continuous on bounded sets in the $p$-variation rough-path metric (\citealp[Cor.~9.11]{FrizVictoir2010}, resting on the quantitative Lyons extension estimate~\citealp[Thm.~9.5]{FrizVictoir2010}; \citealp[Thm.~2.2.1]{Lyons1998}), in particular Borel measurable, so $Z_k = F(\vartheta_{(k-1)\Delta}\mathbf{X})$ for a Borel functional $F$, and stationarity follows from Assumption~\ref{ass:S}. The same argument applies to $Y_k = G(\vartheta_{(k-1)\Delta}\mathbf{X})$, where $G$ maps the shifted process to the piecewise-linear signature after deterministic subsampling and interpolation.
\end{proof}

The estimation error decomposes as
\[
\widehat{\Sigma}_{K} - \mathbb{E}[S^{(M)}] \;=\; \mathcal{B} + \mathcal{V},
\qquad
\mathcal{B} := \frac{1}{K}\sum_{k=1}^K (Y_k - Z_k), \quad
\mathcal{V} := \frac{1}{K}\sum_{k=1}^K (Z_k - \mathbb{E}[Z_k]),
\]
with $\mathcal{B}$ the discretization bias (whose contribution to MSE vanishes as the observation mesh is refined) and $\mathcal{V}$ the statistical fluctuation (whose variance decreases as more blocks are averaged). The discretization rate~$\gamma$ depends on path regularity; the variance rate~$\eta$ depends on the strength of temporal dependence. Their interaction determines the allocation of a fixed observation budget between mesh refinement and block count that minimizes the upper bound (Corollary~\ref{cor:sample_complexity}).

\begin{theorem}
\label{thm:main_result}
Let $\mathbf{X}$ be a random geometric $p$-rough path satisfying Assumptions~\ref{ass:M},~\ref{ass:S}, and~\ref{ass:gamma}, with truncation $M \ge \lfloor p \rfloor$. Suppose in addition that either Assumption~\ref{ass:A} ($\alpha$-mixing) or Assumption~\ref{ass:Aprime} (block-signature covariance decay) holds. Then $\widehat{\Sigma}_{K}$ satisfies
\begin{equation}
\label{eq:mse_bound}
\MSE := \mathbb{E}\bigl[\|\widehat{\Sigma}_{K} - \mathbb{E}[S^{(M)}]\|_{\mathrm{HS}}^2\bigr] \;\le\; \frac{C_{\mathrm{bias}}(\Delta,M)}{(n-1)^\gamma} + V_K(\Delta),
\end{equation}
where $C_{\mathrm{bias}}(\Delta,M) = 2C_{\mathrm{approx}}(\Delta,M)\Delta^\gamma$ and the deterministic variance bound $V_K(\Delta)$ (which dominates the fluctuation contribution $2\,\mathbb{E}\|\mathcal{V}\|_{\mathrm{HS}}^2$ to the MSE bound; it is distinct from the random fluctuation $\mathcal{V}$) depends on the dependence regime:
\begin{enumerate}
\item[\emph{(i)}] \emph{Summable} (Assumption~\ref{ass:A}(E) or~\ref{ass:Aprime}(E$'$); or Assumption~\ref{ass:A}(P) or~\ref{ass:Aprime}(P$'$) with $\nu > 1$): $V_K(\Delta) \le V(\Delta)/K$ with $V(\Delta) := 2V_1(\Delta) + 2\sum_{m=2}^M V_m(\Delta) < \infty$, where $V_1(\Delta) \le 4\,\tr(\Cov(X_0))$ comes from level-1 telescoping (Remark~\ref{rem:level1_telescope}) and the levelwise bound for $m \ge 2$ is
\begin{equation}
\label{eq:Vm_main}
V_m(\Delta) \le \frac{d^m\widetilde{C}_m}{\bigl(\Lambda_p(m/p)!\bigr)^2}\Delta^{2m\beta}\Phi_m(\Delta), \qquad m \ge 2,
\end{equation}
where $\beta = 1/p$, $(m/p)! = \Gamma(1+m/p)$, $\widetilde{C}_m < \infty$ is level dependent, and $\Phi_m(\Delta)$ is a mixing-dependent factor (Appendix~\ref{app:proof}).
\item[\emph{(ii)}] \emph{Non-summable $\alpha$-mixing} (Assumption~\ref{ass:A}(P) with $0 < \nu \le 1$): for every $\varepsilon \in (0, \nu)$, $V_K(\Delta) \le V_\varepsilon(\Delta)\, K^{-\nu+\varepsilon}$ for a finite $\varepsilon$-dependent constant (the bound furnished by the proof for $V_\varepsilon(\Delta)$ is not uniform as $\varepsilon \downarrow 0$; see Appendix~\ref{app:completion}).
\item[\emph{(iii)}] \emph{Non-summable direct covariance decay} (Assumption~\ref{ass:Aprime}(P$'$) with $\nu \le 1$): for $0 < \nu < 1$, $V_K(\Delta) \le V_{\mathrm{poly}}(\Delta)\, K^{-\nu}$; at the borderline $\nu = 1$, $V_K(\Delta) \le V_{\mathrm{log}}(\Delta)\, K^{-1}\log K$ for $K \ge 2$.
\end{enumerate}
The level-1 contribution to the variance term improves to $O(K^{-2})$ via telescoping (Remark~\ref{rem:level1_telescope}); the $\varepsilon$-loss in case~(ii) arises from optimizing the Rio exponent in Lemma~\ref{lem:hilbert_davydov}, and is absent under Assumption~\ref{ass:Aprime}(P$'$). The constants $V_\varepsilon(\Delta)$, $V_{\mathrm{poly}}(\Delta)$, $V_{\mathrm{log}}(\Delta)$ are defined explicitly in Appendix~\ref{app:completion}.
\end{theorem}

\begin{proof}[Proof sketch]
The bias is controlled blockwise via Assumption~\ref{ass:gamma} (Appendix~\ref{app:bias}). The variance bound exploits the orthogonal grading of $T^{(M)}(\mathbb{R}^d)$ level by level: under Assumption~\ref{ass:A}, Lemma~\ref{lem:hilbert_davydov} converts mixing into block-signature covariance bounds, with the Rio exponent sent close to unity under (P) at the cost of an arbitrarily small loss; under Assumption~\ref{ass:Aprime}, the covariance estimates are inserted directly with no $\varepsilon$-loss (Appendix~\ref{app:variance}). Full details appear in Appendix~\ref{app:proof}.
\end{proof}

\begin{remark}[Level-1 telescoping]
\label{rem:level1_telescope}
Block increments at level~1 telescope:
\begin{equation}
\label{eq:level1_telescope}
\bar{S}^{(1)}_K = \frac{1}{K}\sum_{k=1}^K (X_{k\Delta} - X_{(k-1)\Delta}) = \frac{1}{K}(X_{K\Delta} - X_0).
\end{equation}
Strict marginal stationarity ($X_T \stackrel{d}{=} X_0$) gives $\Var(X_{K\Delta} - X_0) \le 4\,\tr(\Cov(X_0))$ for all $K$ (Lemma~\ref{lem:level1_telescope_proof}), hence $\Var(\bar{S}^{(1)}_K) = O(K^{-2})$ --- faster than the levelwise bound~\eqref{eq:Vm_main} suggests.
\end{remark}

\begin{corollary}[Sample complexity]
\label{cor:sample_complexity}
Given $N$ total observations with $N = Kn$, ignoring constants and integer effects, minimizing the displayed upper bound of Theorem~\ref{thm:main_result} gives $n^* \asymp N^{\eta/(\gamma+\eta)}$ and
\[
\MSE \le N^{-\gamma\eta/(\gamma+\eta)+o(1)}.
\]
When $\eta = 1$, this simplifies to $N^{-\gamma/(1+\gamma)}$. For continuous Gaussian semimartingales with independent coordinates (Proposition~\ref{prop:gamma_values}(i)), in the summable regime ($\eta = 1$), $\gamma = 1 - \varepsilon_0$ yields $\MSE \le N^{-1/2 + o(1)}$ as $\varepsilon_0 \downarrow 0$. For fOU, the block-signature covariance is summable in all three regimes (Proposition~\ref{prop:fOU_covariance_decay}: $\nu = 4-2H > 1$ for $H \neq 1/2$, exponential decay at $H = 1/2$), so $\eta = 1$ throughout: for $H \in (1/4, 1/2)$, $\gamma = 4H-1$ gives $\MSE \le N^{-(4H-1)/(4H)+o(1)}$, while for $H > 1/2$ the exact exponents $\gamma = 2H$ (Proposition~\ref{prop:gamma_values}(ii)) and $\eta = 1$ give
\[
\MSE \le C\,N^{-\frac{2H}{2H+1}}
\]
with no $\varepsilon$-loss.
\end{corollary}

% 4. Fractional Ornstein--Uhlenbeck Processes
\section{Fractional Ornstein--Uhlenbeck Processes}
\label{sec:verification}

We verify the hypotheses of Theorem~\ref{thm:main_result} for fractional Ornstein--Uhlenbeck (fOU) processes, whose Hurst index $H$ parameterizes short-range ($H < 1/2$, relevant to rough volatility models~\citealp{GatheralJaissonRosenbaum2018}), semimartingale ($H = 1/2$), and long-range ($H > 1/2$) dependence. Expected signatures are computed via the regularized Wick formulas of~\citet{CassFerrucci2024}; explicit formulas appear in Appendix~\ref{app:expected_signature_fOU}. All fOU verification results in this section assume independent coordinates.

Consider the $d$-dimensional fOU process with independent coordinates:
\begin{equation}
\label{eq:fOU_sde}
dX^i_t = -\theta_i (X^i_t - \mu_i)\,dt + \sigma_i \, dB^{H,i}_t, \qquad i = 1, \ldots, d,
\end{equation}
where $B^{H,i}_t$ are independent fractional Brownian motions (fBm) with Hurst parameter $H \in (0, 1)$, $\theta_i > 0$ are mean-reversion rates, and $\sigma_i > 0$ are volatilities. Since the signature is translation-invariant, we set $\mu_i = 0$ without loss of generality. Initialized from the stationary distribution, the process has stationary autocovariance $R^i(\tau) := \Cov(X^i_t, X^i_{t+\tau})$.

The verification proceeds assumption by assumption; the key challenge is verifying the dependence condition (Assumption~\ref{ass:A} or~\ref{ass:Aprime}), which requires spectral analysis and Gaussian covariance arguments.

\begin{proposition}[Standard hypotheses]
\label{prop:fOU_assumptions}
The independent-coordinate fOU process~\eqref{eq:fOU_sde} with $\theta_i > 0$ and $H \in (1/4, 1)$, initialized from its stationary distribution, admits a canonical geometric rough-path lift $\mathbf{X}$~\citep[Thm.~15.33]{FrizVictoir2010} with finite $p$-variation for every $p > \max(2, 1/H)$, and satisfies the following hypotheses of Theorem~\ref{thm:main_result} with $\beta = 1/p$ for any fixed $p > \max(2, 1/H)$:
\begin{enumerate}
\item[\emph{(M)}] Assumption~\ref{ass:M} holds by the Gaussian rough-path moment bounds of~\citet[Thm.~15.33]{FrizVictoir2010}: the fOU covariance kernel $K_X(s,t) = R(|t-s|)$ has finite 2D $\rho$-variation on $[0,\Delta]^2$ for every $\rho > \rho_*(H) := \max(1, 1/(2H))$, and~\citet[Thm.~15.33]{FrizVictoir2010} requires $p > 2\rho$. The control is moreover H\"older-dominated, $|K_X|^{\rho}_{\rho\text{-var};[s,t]^2} \le K\,|t-s|$ for all $s < t$ in $[0,\Delta]$: in the decomposition $R = R(0) - \kappa|\tau|^{2H} + \psi$ of Lemma~\ref{lem:fOU_cov}, the fractional part satisfies this for $H \le 1/2$ by the $\rho_* = 1/(2H)$ H\"older-dominated control of~\citet[Prop.~15.5]{FrizVictoir2010} (which only improves for $\rho > \rho_*$ on the bounded window) and for $H > 1/2$ by $V_1 \lesssim (t-s)^{2H} \le \Delta^{2H-1}(t-s)$, while the $C^2$ remainder has $V_1(\tilde{\psi}; [s,t]^2) \lesssim (t-s)^2$. The H\"older variant of~\citet[Thm.~15.33(iii)]{FrizVictoir2010} then yields the $(t-s)^{r\beta}$ moment scaling required by Assumption~\ref{ass:M}, with $\beta = 1/p$.
\item[\emph{(S)}] Strict stationarity holds by initialization from the Gaussian invariant measure.
\item[\emph{($\gamma$)}] The discretization rate is given by Proposition~\ref{prop:gamma_values}(ii) for $H \neq 1/2$ and by Proposition~\ref{prop:gamma_values}(i) at $H = 1/2$, with the level-2 exact rates and arbitrarily small $\varepsilon$-loss at higher truncation as stated there.
\item[\emph{(Meas.)}] The lift $\mathbf{X}_{s,t}$ is $\sigma(X_u : u \in [s,t])$-measurable: it is the a.s.\ limit of the step-3 lifts of dyadic piecewise-linear interpolations, which depend only on path values in~$[s,t]$ (naturality,~\citealp[Thm.~15.33(iv)]{FrizVictoir2010}; a.s.\ convergence along dyadics,~\citealp[Ex.~15.44]{FrizVictoir2010}, whose H\"older-dominated-control hypothesis is verified in~(M)).
\end{enumerate}
\end{proposition}

\begin{proof}
Items (M), (S), and (Meas.) follow from the cited references; the discretization rates underlying~($\gamma$) are proved in Appendix~\ref{app:gamma}.
\end{proof}

The remaining hypothesis is the dependence condition, whose verification for fOU is the most technically demanding step. For $H > 1/2$ the stationary fOU is not strongly mixing~\citep{GehringerLi2022}: although the spectral density $f(\lambda) = c\,|\lambda|^{1-2H}/(\theta^2+\lambda^2)$ (for some $c = c(H, \theta, \sigma) > 0$) belongs to $L^1(\mathbb{R})$, it is unbounded at the origin, and the covariance decays non-summably as $R(\tau) \sim c'\,\tau^{2H-2}$ (a distinct positive constant $c'$); in the stationary Gaussian setting, this long-range dependence is incompatible with the complete-regularity criterion for strong mixing~\citep[Ch.~IV]{IbragimovRozanov1978}. For $H < 1/2$ the spectral density vanishes at the origin; we do not use this to claim strong mixing. Instead, we verify Assumption~\ref{ass:Aprime}(P$'$) directly for all $H \neq 1/2$ via the Gaussian structure. The key point is that block signatures are translation-invariant, so they depend on the path only through its \emph{increments}: the long memory of the fOU level never enters. The proof rests on a single analytical input (Lemma~\ref{lem:gaussian_inputs}(a) below): the second derivative of the stationary autocovariance --- the kernel governing increment correlations --- decays as $|R''(\tau)| \le C_R\,\tau^{2H-4}$, two full powers faster than $R$ itself. Separated-support chaos calculus converts this into a $((h-1)\Delta)^{2H-4}$ block-covariance decay, so Assumption~\ref{ass:Aprime}(P$'$) holds with $\nu = 4-2H > 1$ for every $H \neq 1/2$ (at the boundary $H = 1/2$, exponential $\alpha$-mixing holds instead, Proposition~\ref{prop:fOU_covariance_decay}): in all three regimes the block-signature covariance is summable and $\eta = 1$ with no $\varepsilon$-loss, even though the driver itself is long-range dependent for $H > 1/2$. Working with Assumption~\ref{ass:Aprime} rather than $\alpha$-mixing is what makes this sharpening visible.

Let $\tilde{X}_t := X_t - X_0$ denote the centered increment process. By translation invariance of the signature, $S^{(M)}(\mathbf{X})_{s,t} = S^{(M)}(\tilde{\mathbf{X}})_{s,t}$, so we may work with $\tilde{X}$ throughout. Let $\mathcal{H}$ be the Cameron--Martin space of $\tilde{X}$ with reproducing kernel $\mathcal{K}(s,t) := \Cov(\tilde{X}_s, \tilde{X}_t) = R(|t-s|) - R(s) - R(t) + R(0)$, and $I_j^{\mathcal{H}}: \mathcal{H}^{\odot j} \to L^2(\Omega)$ the $j$-th multiple Wiener integral~\citep[\S1.1]{Nualart2006}. Locally write $I_k := [k\Delta, (k+1)\Delta]$, so the separation between $I_0$ and $I_h$ is $(h-1)\Delta$ for $h \ge 1$. The proof of Proposition~\ref{prop:fOU_covariance_decay} rests on the following four-part lemma.

\begin{lemma}[Separated-block Gaussian signature inputs]
\label{lem:gaussian_inputs}
For the stationary fOU process with $H \in (1/4, 1) \setminus \{1/2\}$, fix a block index $k \ge 0$. There exist coordinate-uniform constants $C_R = C_R(H, (\theta_i), (\sigma_i), \Delta) < \infty$ and $C_{\mathrm{CF}} = C_{\mathrm{CF}}(H, (\theta_i), (\sigma_i), \Delta, m) < \infty$, uniform over words $w$ of length $|w| = m$, such that:
\begin{enumerate}
\item[\emph{(a)}] (\emph{Tail bounds}) For each coordinate $i$, the scalar autocovariance $R^i$ satisfies $|R^i(\tau)| \le C_R\,\tau^{2H-2}$ and $|(R^i)''(\tau)| \le C_R\,\tau^{2H-4}$ for all $\tau \ge \Delta$.
\item[\emph{(b)}] (\emph{Chaos expansion}) The block-signature coordinate admits the finite Wiener chaos expansion
\[
\langle w, S^{(m)}(\tilde{\mathbf{X}})_{I_k}\rangle = \sum_{\substack{0 \le j \le m \\ j \equiv m\,(2)}} I_j^{\mathcal{H}}(f_{k,w,j}^{(m)}),
\]
where $f_{k,w,j}^{(m)} \in \mathcal{H}^{\otimes j}$ is the partial-contraction kernel of~\citet[Def.~3.3]{CassFerrucci2024}, supported on the ordered simplex $\Delta^j[k\Delta, (k+1)\Delta]$ and extended by zero to the box $I_k^j$.
\item[\emph{(c)}] (\emph{Kernel $L^\infty$ bound}) For each admissible chaos level $j$, the kernel satisfies the simplex bound
\[
\|f_{k,w,j}^{(m)}\|_{L^\infty(\Delta^j[k\Delta, (k+1)\Delta])} \le C_{\mathrm{CF}}\,\Delta^{(m-j)H},
\]
which the sorting-based symmetric extension carries to $L^\infty(I_k^j)$ unchanged; the symmetrization (all the multiple Wiener integral $I_j^{\mathcal{H}}$ sees) obeys the same bound.
\item[\emph{(d)}] (\emph{Cameron--Martin integral on separated supports}) For each admissible $j$, $h \ge 2$, every slot permutation $\pi \in S_j$, and each scalar coordinate of $\tilde X$, the plain (unsymmetrized) tensor-product pairing satisfies
\[
\bigl\langle f_{0,w,j}^{(m)}, (f_{h,w,j}^{(m)})^{\pi}\bigr\rangle_{\mathcal{H}^{\otimes j}} = (-1)^j \int_{I_0^j \times I_h^j} f_{0,w,j}^{(m)}(u)\, (f_{h,w,j}^{(m)})^{\pi}(v) \prod_{\ell=1}^j R''(v_\ell - u_\ell)\,du\,dv,
\]
where $R$ denotes the scalar autocovariance of the relevant coordinate and the multi-coordinate case follows by independence of coordinates (Appendix~\ref{app:expected_signature_fOU}).
\end{enumerate}
Proofs of \emph{(a)--(d)}, together with the verification of the Cass--Ferrucci standing hypotheses for the increment kernel $\mathcal{K}$ and the properties of the sorting-based symmetric extension, appear in Appendix~\ref{app:expected_signature_fOU}.
\end{lemma}

\begin{proposition}[Block-signature covariance decay]
\label{prop:fOU_covariance_decay}
Under the conditions of Proposition~\ref{prop:fOU_assumptions}, the independent-coordinate fOU process satisfies Assumption~\ref{ass:A}(E) when $H = 1/2$ and Assumption~\ref{ass:Aprime}(P$'$) with $\nu = 4-2H$ when $H \neq 1/2$. Specifically, for $H = 1/2$ each coordinate OU process $X^i$ is Gaussian with exponentially decaying correlation $\Corr(X^i_0,X^i_t)=e^{-\theta_i t}$; by the Gaussian mixing criterion~\citep[eq.~(1.9), Ch.~IV.1]{IbragimovRozanov1978} and independence of coordinates, the vector process $X$ is exponentially $\alpha$-mixing with rate $\theta_{\min} := \min_i \theta_i$. For $H \neq 1/2$ the block-signature covariances satisfy
\begin{equation}
\label{eq:chm_gaussian}
|c_{h,m}| \;\le\; \frac{d^m\widetilde{C}_m}{\bigl(\Lambda_p(m/p)!\bigr)^2}\Delta^{2m\beta} ((h-1)\Delta)^{2H-4} \quad\text{for all } h \ge 2,
\end{equation}
where $\widetilde{C}_m$ depends on $m, H, \theta, \sigma, \Delta$ but not on $h$. The adjacent-block case $h = 1$ is handled separately via Cauchy--Schwarz, yielding the unified Assumption~\ref{ass:Aprime}(P$'$) envelope
\[
|c_{h,m}| \le d^m \widetilde{C}_m\bigl(\Lambda_p(m/p)!\bigr)^{-2}\Delta^{2m\beta}(1+h)^{-\nu} \quad \text{for all } h \ge 1.
\]
Since $\nu = 4-2H \in (2, 7/2)$ for every $H \in (1/4,1) \setminus \{1/2\}$, the block-signature covariance is summable in all three regimes, and Theorem~\ref{thm:main_result} applies in its summable case~(i): $\eta = 1$ with no $\varepsilon$-loss.
\end{proposition}

\begin{proof}
The $H = 1/2$ case follows from the Gaussian mixing criterion of~\citet[eq.~(1.9), Ch.~IV.1 and Thm.~2, Ch.~IV.2]{IbragimovRozanov1978}, applied componentwise.

For $H \neq 1/2$, fix a word $w$ of length $m$. By Lemma~\ref{lem:gaussian_inputs}(b) and the Wiener chaos isometry, the centered block-signature covariance decomposes as
\begin{equation}
\label{eq:chm_chaos_sum_main}
c_{h,m}(w) = \sum_{\substack{1 \le j \le m \\ j \equiv m\,(2)}} j!\,\bigl\langle \tilde{f}_{0,w,j}^{(m)}, \tilde{f}_{h,w,j}^{(m)}\bigr\rangle_{\mathcal{H}^{\otimes j}}
= \sum_{\substack{1 \le j \le m \\ j \equiv m\,(2)}}\, \sum_{\pi \in S_j} \bigl\langle f_{0,w,j}^{(m)}, (f_{h,w,j}^{(m)})^{\pi}\bigr\rangle_{\mathcal{H}^{\otimes j}},
\end{equation}
where $\tilde{f}$ denotes the symmetrization, the second form expands it over the $j!$ slot permutations in plain pairings, and centering eliminates the $j = 0$ contribution.

\emph{Bound for $h \ge 2$.} For each admissible chaos level $j$, the integral representation of Lemma~\ref{lem:gaussian_inputs}(d) gives, uniformly over the slot permutations $\pi$ (the right-hand side is $\pi$-invariant),
\[
\bigl|\bigl\langle f_{0,w,j}^{(m)}, (f_{h,w,j}^{(m)})^{\pi}\bigr\rangle_{\mathcal{H}^{\otimes j}}\bigr| \le \|f_{0,w,j}^{(m)}\|_\infty \|f_{h,w,j}^{(m)}\|_\infty \cdot |I_0^j \times I_h^j| \cdot \sup_{u \in I_0^j,\, v \in I_h^j}\prod_{\ell=1}^j |R''(v_\ell - u_\ell)|.
\]
On $I_0 \times I_h$ with $h \ge 2$, the separation $v_\ell - u_\ell \ge (h-1)\Delta$ together with Lemma~\ref{lem:gaussian_inputs}(a) gives $|R''(v_\ell - u_\ell)| \le C_R\,((h-1)\Delta)^{2H-4}$. Combined with the kernel bound (c) and the box volume $|I_0^j \times I_h^j| = \Delta^{2j}$:
\[
\bigl|\bigl\langle f_{0,w,j}^{(m)}, (f_{h,w,j}^{(m)})^{\pi}\bigr\rangle_{\mathcal{H}^{\otimes j}}\bigr| \le C_{\mathrm{CF}}^2\,\Delta^{2(m-j)H + 2j}\, C_R^j\,((h-1)\Delta)^{j(2H-4)}.
\]
Since $2H - 4 < 0$ and $h - 1 \ge 1$ for $h \ge 2$, the inequality
\[
((h-1)\Delta)^{j(2H-4)} \;\le\; ((h-1)\Delta)^{2H-4}\,\Delta^{(j-1)(2H-4)} \qquad (j \ge 1)
\]
extracts the uniform $((h-1)\Delta)^{2H-4}$ envelope at the cost of a $j$- and $\Delta$-dependent factor. Summing the $j!$ slot permutations in~\eqref{eq:chm_chaos_sum_main}, the finitely many admissible $j$, and the $d^m$ words of length $m$, and absorbing all resulting $\Delta$- and $m$-dependent factors into a level-dependent constant, yields
\begin{equation}
\sum_{|w| = m} |c_{h,m}(w)| \;\le\; \frac{d^m \widetilde{C}_m}{\bigl(\Lambda_p(m/p)!\bigr)^2}\,\Delta^{2m\beta}\,((h-1)\Delta)^{2H-4},
\end{equation}
which is~\eqref{eq:chm_gaussian}. Since $h - 1 \ge (1+h)/3$ for $h \ge 2$ and $2H - 4 < 0$,
\[
((h-1)\Delta)^{2H-4} \le 3^{4-2H}\,\Delta^{2H-4}\,(1+h)^{-(4-2H)};
\]
absorbing $3^{4-2H}\Delta^{2H-4}$ into $\widetilde{C}_m$ gives Assumption~\ref{ass:Aprime}(P$'$) with $\nu = 4-2H$.

\emph{Adjacent block ($h=1$).} The bound~\eqref{eq:chm_gaussian} cannot be used at $h = 1$ since $((h-1)\Delta)^{2H-4}$ diverges. Instead, Cauchy--Schwarz and stationarity give $|c_{1,m}| \le c_{0,m}$, which is bounded by Lemma~\ref{lem:level_moments} (Appendix~\ref{app:proof}) as $c_{0,m} \le d^m C^{\mathrm{sig}}_{m,2}(\Lambda_p(m/p)!)^{-2}\Delta^{2m\beta}$. Enlarging $\widetilde{C}_m$ to dominate $2^\nu C^{\mathrm{sig}}_{m,2}$ ensures $|c_{1,m}| \le d^m\widetilde{C}_m(\Lambda_p(m/p)!)^{-2}\Delta^{2m\beta}\,(1+1)^{-\nu}$, completing the unified envelope for all $h \ge 1$.

\emph{Multi-coordinate reduction.} For $d > 1$, by independence of the fOU coordinates each Wick-pairing component of the chaos kernel $f_{k,w,j}^{(m)}$ lies in $\bigotimes_{i=1}^d \mathcal{H}_i^{\otimes a_i}$ for its own free-coordinate profile $(a_i)_{i=1}^d$ with $\sum_i a_i = j$ (distinct pairings may leave distinct profiles). Orthogonality across distinct profiles forces surviving component pairings to factor coordinatewise --- likewise for every slot-permuted pairing, since a slot permutation preserves the profile --- so the scalar argument applies componentwise, the finitely many components enlarging $\widetilde{C}_m$ only. Full details appear in Appendix~\ref{subsec:proof_fOU_covariance_decay}.
\end{proof}

Propositions~\ref{prop:fOU_assumptions} and~\ref{prop:fOU_covariance_decay} together supply the exponents required by Theorem~\ref{thm:main_result} for fOU on $H \in (1/4, 1)$ --- with $\eta = 1$ throughout --- and with the $H = 1/2$ boundary treated as a Gaussian semimartingale via Proposition~\ref{prop:gamma_values}(i).

Beyond verifying the hypotheses, the same Gaussian-chaos analysis yields the fOU increment signature in closed form, providing the exact deterministic ground truth for the experiments of Section~\ref{sec:numerical_results}.

\begin{proposition}[Closed-form fOU expected signatures]
\label{prop:fou_closed_forms}
For the centered, independent-coordinate fOU increment process over a block of length $\Delta$, write $R_i := R^i$ for the stationary autocovariance of coordinate $i$ and $R'_i$ for its derivative. The odd signature levels vanish in expectation, and for every $H \in (1/4, 1)$ the level-2 and level-4 expected signatures admit the closed forms
\begin{equation}
\label{eq:gt_level2}
\mathbb{E}[S^{(2)}_{ij}] = v_i\,\mathbb{1}\{i = j\}, \qquad v_i := R_i(0) - R_i(\Delta),
\end{equation}
and, for distinct coordinates $i \neq j$,
\begin{equation}
\label{eq:gt_level4}
\begin{aligned}
\mathbb{E}[S^{(4)}_{iijj}] &= \iint_{s + w < \Delta,\, s, w > 0} R'_i(s)\, R'_j(w)\, ds\, dw, \\
\mathbb{E}[S^{(4)}_{ijij}] &= \iint_{0 < t_2 < t_3 < \Delta} [R'_i(t_3) - R'_i(t_3 - t_2)]\,[R'_j(\Delta - t_2) - R'_j(t_3 - t_2)]\, dt_2\, dt_3, \\
\mathbb{E}[S^{(4)}_{ijji}] &= \int_0^\Delta R'_i(t)\,[R_j(t) - R_j(0)]\, dt - \int_0^\Delta (\Delta - w)\, R'_i(w)\, R'_j(w)\, dw.
\end{aligned}
\end{equation}
The diagonal word $\mathbb{E}[S^{(4)}_{iiii}]$ is the sum of the three expressions evaluated at $R_i = R_j$. These follow from the Cass--Ferrucci Wiener-chaos framework~\citep{CassFerrucci2024}; the derivation and the $H \le 1/2$ regularization appear in Appendix~\ref{app:expected_signature_fOU}.
\end{proposition}

% 5. Numerical Experiments
\section{Numerical Experiments}
\label{sec:numerical_results}

For the fOU model of Section~\ref{sec:verification}, we illustrate the upper-bound rates of Theorem~\ref{thm:main_result} through Monte Carlo experiments testing: (i)~discretization bias $O((n-1)^{-\gamma})$; (ii)~variance decay $O(K^{-\eta})$ with $\eta = 1$; and (iii)~optimal allocation $O(N^{-\gamma/(1+\gamma)})$.

\subsection{Experimental Framework}

\paragraph{Process configuration}
We simulate bivariate ($d = 2$) fOU processes with independent coordinates, each satisfying $dX^i_t = -\theta X^i_t\,dt + \sigma\,dB^{H,i}_t$ with $\theta = \sigma = 1$ and $\mu = 0$. The Hurst parameters $H \in \{0.40, 0.50, 0.60\}$ span the short-range ($1/4 < H < 1/2$), semimartingale ($H = 1/2$), and long-range ($H > 1/2$) regimes. Under Proposition~\ref{prop:fOU_covariance_decay}, $H = 0.50$ satisfies Assumption~\ref{ass:A}(E), while $H = 0.40$ and $H = 0.60$ satisfy Assumption~\ref{ass:Aprime}(P$'$) with $\nu = 4 - 2H$ ($= 3.2$ and $2.8$ respectively) $> 1$; all three configurations fall in case~(i) of Theorem~\ref{thm:main_result} ($\eta = 1$). Discretization rates $\gamma$ are given by Proposition~\ref{prop:gamma_values}; the corresponding targets are $\gamma \in \{0.6, 1, 1.2\}$ and $\eta = 1$ for all three. All experiments use truncation level $M = 4$. Stationary fOU paths are generated with the Davies--Harte circulant embedding algorithm~\citep{DaviesHarte1987}, which is exact in distribution given the stationary autocovariance; the autocovariance is computed numerically (spectral quadrature), with discretization-side approximations far below the Monte Carlo noise floor.

\paragraph{Statistical methodology}
All experiments use $1{,}000$ independent replications. Each replication generates an fOU path, computes $\widehat{\Sigma}_K$, and records $\|\widehat{\Sigma}_K - \mathbb{E}[S^{(M)}]\|_{\mathrm{HS}}^2$. Convergence rates are estimated via OLS regression of $\log(\MSE)$ on the logarithm of the relevant parameter; the resulting OLS slope $\hat{b}$ estimates the power-law exponent. We report 95\% confidence intervals from pairs bootstrap ($10{,}000$ resamples). The reported \emph{Bound} in Table~\ref{tab:rate_results} is the theoretical target slope; for $H \in \{0.40, 0.50\}$ it suppresses the arbitrarily small $\varepsilon$-loss carried by the rigorous truncation-$M=4$ discretization statement (the level-2 discretization rates, all variance rates, and all $H = 0.60$ rates are exact). Since Theorem~\ref{thm:main_result} provides upper bounds, empirical slopes strictly more negative than the bound are consistent with the theory.

\paragraph{Ground truth}
The deterministic target $\mathbb{E}[S^{(M)}]$ is computed from the closed forms of Proposition~\ref{prop:fou_closed_forms} for $H > 1/2$ (and for the level-2 and $P_2$ components at every $H$), and for the remaining components at $H \le 1/2$ by discretizing the Cass--Ferrucci approximants of~\citep[Def.~3.3, Prop.~3.9]{CassFerrucci2024} via Riemann--Stieltjes summation of the covariance kernel $\mathcal{K}$, Richardson-extrapolated over a geometric sequence of meshes --- the closed forms remain valid there (Proposition~\ref{prop:fou_closed_forms}), but the singularity of $R'$ at the origin makes their direct quadrature numerically delicate; the two routes agree where both are applied. As a non-circular self-consistency check, the shuffle product identity for geometric rough paths~\citep{FrizVictoir2010} forces the sum of the six level-$4$ words obtained by shuffling $(i,i)$ and $(j,j)$ for independent coordinates $i \neq j$ to equal $\mathbb{E}[S^{(2)}_{ii}]\,\mathbb{E}[S^{(2)}_{jj}]$; the computed ground truth satisfies this identity to within $0.8\%$ relative error at every tested $(H, \Delta)$ configuration.

\subsection{Results}

\paragraph{Experiment 1: Discretization bias}
We test $\mathbb{E}[\|\mathcal{B}\|_{\mathrm{HS}}^2] = O((n-1)^{-\gamma})$ with targets $\gamma \in \{0.6, 1, 1.2\}$ for $H \in \{0.40, 0.50, 0.60\}$. We take block length $\Delta = 1$ (larger than in Experiments~2--3, to isolate the discretization bias) and partition each of $1{,}000$ independent long paths into $K = 10^6$ consecutive blocks, so the variance contribution lies well below the total MSE and $\log(\MSE)$ in Table~\ref{tab:rate_results} reads as a proxy for $\log(\mathbb{E}[\|\mathcal{B}\|_{\mathrm{HS}}^2])$. We vary $n \in \{3, 5, 9, 17, 33\}$.

\paragraph{Experiment 2: Variance decay}
We test $\mathbb{E}[\|\mathcal{V}\|_{\mathrm{HS}}^2] = O(K^{-\eta})$, with $\eta = 1$ in all three regimes (Proposition~\ref{prop:fOU_covariance_decay}: the block-signature covariance is summable --- $\nu = 4-2H > 1$ at $H \in \{0.40, 0.60\}$, exponential decay at $H = 0.50$ --- with no $\varepsilon$-loss). We take $\Delta = 0.1$ with $n = 1{,}000$ points per block, so the squared bias lies well below the total MSE and $\log(\MSE)$ in Table~\ref{tab:rate_results} reads as a proxy for $\log(\mathbb{E}[\|\mathcal{V}\|_{\mathrm{HS}}^2])$. We vary $K \in \{64, 128, 256, 512, 1024, 2048\}$.

The empirical slopes are more negative than the upper bounds, partly reflecting the $O(K^{-2})$ level-1 telescoping contribution (Remark~\ref{rem:level1_telescope}); the steepest excess occurs at $H = 0.60$ ($-1.47$ against the guaranteed $-1.0$), consistent with the summable-covariance prediction $\eta = 1$ together with finite-sample cancellations beyond it.

\paragraph{Experiment 3: Optimal allocation}
We test the combined rate $\MSE^* = O(N^{-\gamma/(1+\gamma)})$ under the budget constraint $N = Kn$. For $N \in \{10^2, 10^3, 10^4, 10^5, 10^6, 10^7\}$ with $\Delta = 0.1$, we allocate $n^* = \lceil N^{1/(1+\gamma)} \rceil$ (with $\eta = 1$ and $\gamma = \min(4H-1,\,2H)$, ignoring constants and, for $H \le 1/2$, the arbitrarily small $\varepsilon$-losses) and $K^* = \lfloor N/n^* \rfloor$.

The theoretical upper bounds quantify the cost of roughness: obtaining $\MSE \le \varepsilon_{\mathrm{tol}}$ through this calculation requires $N \propto \varepsilon_{\mathrm{tol}}^{-8/3 + o(1)}$ observations for $H = 0.40$, versus $N \propto \varepsilon_{\mathrm{tol}}^{-2 + o(1)}$ for the continuous Gaussian semimartingale case.

Table~\ref{tab:rate_results} reports the fitted convergence rates for all three experiments, and Figure~\ref{fig:convergence} displays the corresponding log-log plots.

\begin{table}[tb]
\caption{Convergence-rate validation for the three experiments at $H \in \{0.40, 0.50, 0.60\}$: OLS slope $\hat{b}$ of $\log(\MSE)$ on the logarithm of the design variable, with the theoretical upper-bound slope. $1{,}000$ replications; 95\% CIs from $10{,}000$-resample pairs bootstrap.}
\label{tab:rate_results}
\centering
\begin{tabular}{lcccc}
\toprule
$\boldsymbol{H}$ & \textbf{Bound} & $\boldsymbol{\hat{b}}$ & \textbf{95\% CI} & $\boldsymbol{R^2}$ \\
\midrule
\multicolumn{5}{l}{\emph{Experiment 1: Discretization bias} ($\log\MSE$ vs.\ $\log(n-1)$)} \\
$0.40$ & $\le -0.6$ & $-1.01$ & $[-1.05, -0.97]$ & $0.999$ \\
$0.50$ & $\le -1.0$ & $-1.82$ & $[-1.88, -1.74]$ & $0.999$ \\
$0.60$ & $\le -1.2$ & $-2.08$ & $[-2.20, -1.89]$ & $0.998$ \\
\addlinespace
\multicolumn{5}{l}{\emph{Experiment 2: Variance decay} ($\log\MSE$ vs.\ $\log K$)} \\
$0.40$ & $\le -1.0$ & $-1.10$ & $[-1.19, -1.02]$ & $0.997$ \\
$0.50$ & $\le -1.0$ & $-1.26$ & $[-1.39, -1.13]$ & $0.995$ \\
$0.60$ & $\le -1.0$ & $-1.47$ & $[-1.60, -1.34]$ & $0.996$ \\
\addlinespace
\multicolumn{5}{l}{\emph{Experiment 3: Optimal allocation} ($\log\MSE$ vs.\ $\log N$)} \\
$0.40$ & $\le -0.375$ & $-0.50$ & $[-0.54, -0.46]$ & $0.997$ \\
$0.50$ & $\le -0.50$ & $-0.67$ & $[-0.75, -0.58]$ & $0.994$ \\
$0.60$ & $\le -0.545$ & $-0.81$ & $[-0.92, -0.68]$ & $0.991$ \\
\bottomrule
\end{tabular}
\end{table}

\begin{figure}[tb]
\centering
\includegraphics[width=\linewidth]{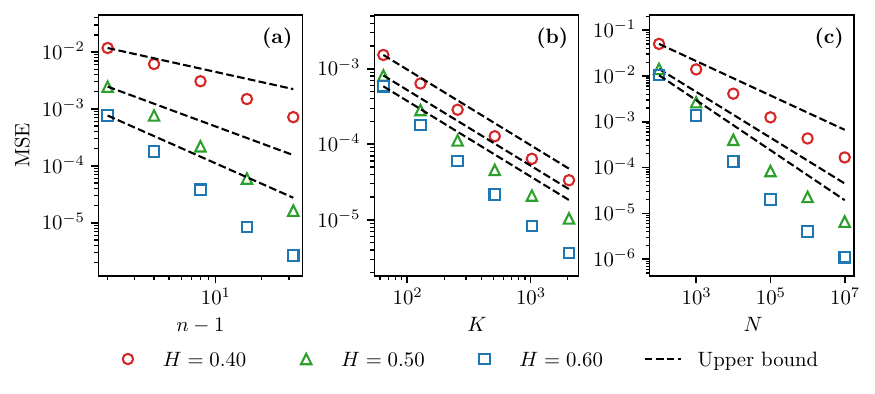}
\caption{Log-log MSE convergence for $H \in \{0.40, 0.50, 0.60\}$ (Experiments~1--3). Markers are empirical; dashed lines carry the upper-bound slopes, normalized at the leftmost point: (a)~discretization bias, $-\gamma = -\min(4H-1,\,2H)$; (b)~variance decay, $-\eta = -1$; (c)~optimal allocation, $-\gamma/(1+\gamma)$. In each panel the empirical error decays at least as fast as its bound.}
\label{fig:convergence}
\end{figure}

\paragraph{Levelwise variance concentration}
At $K = 1{,}000$, $\Delta = 0.1$, $n = 1{,}000$, level~2 alone captures $86$--$96\%$ of the $m \ge 2$ block-signature variance across $H \in \{0.40, 0.50, 0.60\}$, with levels~3 and~4 contributing at most $11.6\%$ and $2.0\%$ respectively (truncation $M = 4$, $d = 2$). Although $T^{(4)}(\mathbb{R}^2)$ has dimension~31, this rough-factorial level concentration is consistent with the levelwise structure of Theorem~\ref{thm:main_result}.

In these fOU experiments, the empirical convergence rates are consistent with --- and strictly more negative than --- the upper bounds of Theorem~\ref{thm:main_result}.

% 6. Discussion
\section{Discussion}
\label{sec:discussion}

The non-asymptotic bounds of Theorem~\ref{thm:main_result} decompose the estimation error into path-regularity and dependence contributions, with empirical convergence rates faster than the upper-bound rates in all tested fOU configurations (Section~\ref{sec:numerical_results}). The analysis isolates four structural ingredients (moment control, dependence decay, stationarity, and discretization quality), any one of which may be the binding constraint in a given application.

The optimal allocation formula $n^* \asymp N^{\eta/(\gamma+\eta)}$ provides explicit guidance for splitting a computational budget between mesh refinement and block count, given knowledge of the regularity and dependence parameters. For instance, with $H = 0.40$ and a budget of $N = 10^6$ observations, the upper-bound-optimal split allocates $n^* \approx N^{5/8} \approx 5{,}624$ points per block and $K^* \approx 177$ blocks; the resulting MSE bound is $\MSE \le N^{-0.375 + o(1)}$. The levelwise variance constants carry rough-factorial damping and are finite for fixed truncation $M$; for the tested $d=2$, level~2 alone accounts for $87$--$96\%$ of the block-signature variance (levels $m \ge 2$; Section~\ref{sec:numerical_results}).

In Section~\ref{sec:numerical_results}, observed slopes are consistently more negative than the upper bounds; for variance decay (Experiment~2), ratios of empirical to theoretical exponents range from $1.1\times$ ($H = 0.40$) to $1.5\times$ ($H = 0.60$). Much of fOU's Gaussian structure is already exploited: the separated-block chaos calculus behind Proposition~\ref{prop:fOU_covariance_decay} shows that block-signature covariances are governed by the second derivative of the driver's autocovariance, $R''(\tau) \sim \tau^{2H-4}$, rather than by $R(\tau) \sim \tau^{2H-2}$ itself --- block signatures see only increments --- which is why $\eta = 1$ holds even in the long-range regime. The remaining empirical--theoretical gap is attributable to the $O(K^{-2})$ level-1 telescoping contribution (Remark~\ref{rem:level1_telescope}), to non-asymptotic constants, and to the conservatism of the unified envelope $\nu = 4-2H$ across chaos levels (higher chaos orders decay faster). We also note that this increment-covariance mechanism produces a summable envelope whenever the second derivative of the driver's stationary autocovariance decays summably, as it does for fOU --- there $|R''(\tau)| \le C_R\,\tau^{2H-4}$ (Lemma~\ref{lem:gaussian_inputs}(a)) gains two full powers over $R$ itself --- so the non-summable branch of Theorem~\ref{thm:main_result}(iii) is not exercised by the fOU class; it remains relevant for non-Gaussian or cross-correlated dependence structures, and for Gaussian drivers whose autocovariance has a slowly decaying oscillatory tail, for which the envelope $\sup_{\tau \ge (h-1)\Delta}|R''(\tau)|$ need not be summable.

The current scope is constrained by four structural restrictions: independent coordinates in the fOU verification (correlated fractional drivers require joint spectral analysis; see the proof of Proposition~\ref{prop:fOU_covariance_decay}); the strict-stationarity requirement of the block-averaging estimator (which excludes trends, regime changes, and non-constant diffusion coefficients); the piecewise-linear interpolation scheme (other schemes may yield different exponents); and the fixed truncation $M$ (the bound's constants are finite at each $M$ but their behavior as $M$ grows is not addressed). The empirical evidence for slack of the bounds is itself restricted to fOU at $d = 2$, $M = 4$, and three Hurst values. Several questions remain open: whether the rates $\gamma$ and $\eta$ are sharp (which would require matching minimax lower bounds over appropriate function classes), whether the framework extends to non-Gaussian rough paths (where Wick-based ground-truth formulas are unavailable), and whether direct $L^2$ techniques can close the empirical gap for the Gaussian case.

% -------------------------------------------------------------------
% Acknowledgments
% -------------------------------------------------------------------
\begin{acks}[Acknowledgments]
This work builds on the author's Master's thesis, completed within the MSc UZH ETH in Quantitative Finance program under the supervision of Professor Josef Teichmann, whom the author thanks for suggesting the Ornstein--Uhlenbeck application and for advice on presentation. The author also acknowledges the use of large-language-model assistants (Anthropic Claude, OpenAI GPT, and Google Gemini) during manuscript preparation for editorial review, notation consistency checks, and prose drafting and revision. All such text was reviewed and edited by the author. All mathematical results are the author's own, and the author bears full responsibility for the contents.
\end{acks}

% -------------------------------------------------------------------
% Appendices
% -------------------------------------------------------------------
\appendix

% Appendix A: Proofs of Theorem~\ref{thm:main_result}
\section{Proofs}
\label{app:proof}

This appendix provides the complete proof of Theorem~\ref{thm:main_result}, including the levelwise variance bounds. We work throughout in the Hilbert space $L^2(\Omega; T^{(M)}(\mathbb{R}^d))$ equipped with norm $\|Y\|_{L^2} = (\mathbb{E}[\|Y\|_{\mathrm{HS}}^2])^{1/2}$. The proof proceeds in four parts: auxiliary moment and mixing estimates (Appendix~\ref{app:auxiliary}), bias analysis (Appendix~\ref{app:bias}), variance analysis via orthogonal levelwise decomposition (Appendix~\ref{app:variance}), and final assembly (Appendix~\ref{app:completion}). Verification of the discretization rates used in Proposition~\ref{prop:gamma_values} is deferred to Appendix~\ref{app:gamma}.

\paragraph{Notation}
Throughout this appendix: $Z_k := S^{(M)}(\mathbf{X})_{I_k}$ denotes the true block signature over the $k$-th block $I_k = [(k-1)\Delta, k\Delta]$; $Y_k := S^{(M)}(\mathbf{X}^{(n)}_k)_{I_k}$ denotes the observable signature computed from piecewise-linear interpolation of $n$ observations; and $W_k := Z_k - \mathbb{E}[Z_k]$ denotes the centered block signature. We write $\|\cdot\|$ for $\|\cdot\|_{\mathrm{HS}}$ when the context is clear.

%%%%%%%%%%%%%%%%%%%%%%%%%%%%%%%%%%%%%%%%%%%%%%%%%%%%%%%%%%%%%%%%%%%%%%%%%%%%%%%
\subsection{Auxiliary Estimates}
\label{app:auxiliary}
%%%%%%%%%%%%%%%%%%%%%%%%%%%%%%%%%%%%%%%%%%%%%%%%%%%%%%%%%%%%%%%%%%%%%%%%%%%%%%%

\begin{lemma}[Level-specific moment bounds]
\label{lem:level_moments}
Under Assumption~\ref{ass:M}, for the degree-$m$ signature component $S^{(m)}(\mathbf{X})_{0,\Delta} \in (\mathbb{R}^d)^{\otimes m}$ with $1 \le m \le M$,
\begin{equation}
\label{eq:level_moment_L2}
\mathbb{E}\left[\|S^{(m)}(\mathbf{X})_{0,\Delta}\|_{\mathrm{HS}}^2\right] \le \frac{d^m C^{\mathrm{sig}}_{m,2}}{\bigl(\Lambda_p(m/p)!\bigr)^2}\Delta^{2m\beta},
\end{equation}
where $(m/p)! = \Gamma(1+m/p)$ and $C^{\mathrm{sig}}_{m,2}<\infty$ is level dependent (and depends on $\Delta$ through Assumption~\ref{ass:M}) and contains the deterministic constants from the signature map. More generally, for any $q \ge 1$,
\begin{equation}
\label{eq:level_moment_Lq}
\mathbb{E}\left[\|S^{(m)}(\mathbf{X})_{0,\Delta}\|_{\mathrm{HS}}^q\right] \le \frac{d^{mq/2} C^{\mathrm{sig}}_{m,q}}{\bigl(\Lambda_p(m/p)!\bigr)^q}\Delta^{mq\beta}.
\end{equation}
\end{lemma}

\begin{proof}
A \emph{word} of length $m$ is a multi-index $w = (i_1, \ldots, i_m)$ with $i_j \in \{1, \ldots, d\}$; we write $|w| = m$. The pairing $\langle w, S \rangle$ denotes the corresponding iterated-integral coefficient of the signature.

The Hilbert--Schmidt norm decomposes over the $d^m$ tensor basis elements:
\begin{equation}
\|S^{(m)}(\mathbf{X})_{0,\Delta}\|_{\mathrm{HS}}^2 = \sum_{|w|=m} |\langle w, S(\mathbf{X})_{0,\Delta} \rangle|^2.
\end{equation}

This estimate uses the rough-factorial extension bound, not the bounded-variation simplex factorial. The Lyons extension theorem~\citep[Thm.~9.5]{FrizVictoir2010} determines the order-$M$ lift from the order-$\lfloor p\rfloor$ rough path, and Lyons' factorial-decay estimate~\citep[Thm.~2.2.1]{Lyons1998} gives, for a suitable rough-path control and a Lyons constant $\Lambda_p>0$, the bound $\|X^i_{s,t}\| \le \omega(s,t)^{i/p}/(\Lambda_p(i/p)!)$ both for $i \le p$ and for the extended levels $i>p$ (for $i \le p$ the bound is the theorem's hypothesis, arranged by rescaling the control $\omega$; the theorem then extends it to all $i > p$, and the rescaling constant is absorbed into $K_{m,p,M}$ below). Applying this with $i=m$, using equivalence of homogeneous norms on $G^M(\mathbb{R}^d)$~\citep[Thm.~7.44]{FrizVictoir2010}, and absorbing finite-dimensional tensor-norm constants gives deterministic constants $K_{m,p,M}<\infty$ such that
\begin{equation}
\label{eq:level_rough_factorial_pathwise}
|\langle w, S(\mathbf{X})_{0,\Delta}\rangle|
\le \frac{K_{m,p,M}}{\Lambda_p(m/p)!}\|\mathbf{X}\|_{p\text{-var};[0,\Delta]}^m,
\qquad |w|=m.
\end{equation}
Squaring the preceding display, summing over the $d^m$ words, taking expectations, and applying Assumption~\ref{ass:M} with exponent $2m$ yields~\eqref{eq:level_moment_L2}. The same pathwise bound with exponent $q$ and Assumption~\ref{ass:M} with exponent $mq$ yields~\eqref{eq:level_moment_Lq}.

For the Gaussian examples verified in Appendices~\ref{app:gamma} and~\ref{app:expected_signature_fOU}, the same constants admit a Wiener chaos derivation via~\citet[Thm.~2.3]{CassFerrucci2024} combined with hypercontractivity on each finite chaos component~\citep[Thm.~5.10]{Janson1997}.
\end{proof}

The next lemma is stated for general exponents to enable optimization of the mixing rate.

\begin{lemma}[Finite-dimensional Hilbert-valued covariance bound]
\label{lem:hilbert_davydov}
Let $U, V$ be centered random variables in a finite-dimensional Hilbert space $\mathsf{E}$ with $\dim(\mathsf{E}) = D$ and $\alpha$-mixing coefficient $\alpha$. For $q > 2$ and $r > 1$ satisfying $2/q + 1/r = 1$:
\begin{equation}
\label{eq:davydov_hilbert}
|\mathbb{E}[\langle U, V \rangle_{\mathsf{E}}]| \le 2^{1+1/r}\, D\, \alpha^{1/r}\, \|U\|_{L^q(\mathsf{E})}\, \|V\|_{L^q(\mathsf{E})}.
\end{equation}
\end{lemma}

\begin{proof}
Fix an orthonormal basis $(e_i)_{i=1}^{D}$ of $\mathsf{E}$ and write $U_i := \langle U, e_i \rangle_{\mathsf{E}}$, $V_i := \langle V, e_i \rangle_{\mathsf{E}}$. Apply Rio's scalar covariance inequality for $\alpha$-mixing variables in moment form --- obtained from the quantile bound $|\Cov(X,Y)| \le 2\int_0^{2\alpha} Q_{|X|}(u)\,Q_{|Y|}(u)\,du$ of~\citet[Thm.~1.1]{Rio1993} ($Q_{|X|}$ the quantile function of $|X|$) by H\"older with exponents $(r,q,q)$ and the identity $\int_0^1 Q_{|X|}^q\,du = \mathbb{E}|X|^q$; cf.~\citep[Ineq.~(1.1)]{Davydov1968} for the earlier $\varphi$-mixing form --- to each coordinate pair $(U_i, V_i)$, whose mixing coefficient satisfies $\alpha(\sigma(U_i), \sigma(V_i)) \le \alpha$ because $\sigma(U_i) \subseteq \sigma(U)$ and $\sigma(V_i) \subseteq \sigma(V)$, so the lemma's $\alpha$ bounds every pair:
\begin{equation}
\label{eq:rio_scalar}
|\mathbb{E}[\langle U, V \rangle_{\mathsf{E}}]| \le \sum_{i=1}^D |\mathbb{E}[U_i V_i]| \le 2^{1+1/r}\, \alpha^{1/r} \sum_{i=1}^D \|U_i\|_{L^q} \|V_i\|_{L^q}.
\end{equation}
By the pointwise bound $\|U_i\|_{L^q} \le \|U\|_{L^q(\mathsf{E})}$ (since $|U_i| \le \|U\|_{\mathsf{E}}$), summing over the $D$ coordinates gives $2^{1+1/r}\, D\, \alpha^{1/r}\, \|U\|_{L^q(\mathsf{E})}\, \|V\|_{L^q(\mathsf{E})}$.
\end{proof}

\begin{remark}
\label{rem:dim_factor}
The factor $D = d^m$ in~\eqref{eq:davydov_hilbert} arises from applying the triangle inequality to the coordinate sum in~\eqref{eq:rio_scalar}. Replacing it by Cauchy--Schwarz on the coordinate sum, combined with the finite-dimensional interpolation $\|\cdot\|_{\ell^2_D} \le D^{1/2-1/q}\|\cdot\|_{\ell^q_D}$ (valid for $q \ge 2$), yields the sharper bound $\sum_i \|U_i\|_{L^q}\|V_i\|_{L^q} \le D^{1/r}\|U\|_{L^q(\mathsf{E})}\|V\|_{L^q(\mathsf{E})}$ (using $1/r = 1 - 2/q$), reducing the dimension prefactor from $D$ to $D^{1/r}$. This affects only finite level-dependent constants in the bounds below.
\end{remark}

%%%%%%%%%%%%%%%%%%%%%%%%%%%%%%%%%%%%%%%%%%%%%%%%%%%%%%%%%%%%%%%%%%%%%%%%%%%%%%%
\subsection{Bias Analysis}
\label{app:bias}
%%%%%%%%%%%%%%%%%%%%%%%%%%%%%%%%%%%%%%%%%%%%%%%%%%%%%%%%%%%%%%%%%%%%%%%%%%%%%%%

Recall that $Z_k := S^{(M)}(\mathbf{X})_{I_k}$ denotes the true block signature and $Y_k := S^{(M)}(\mathbf{X}^{(n)}_k)_{I_k}$ the observable signature computed from piecewise-linear interpolation. The estimation error decomposes as:
\begin{align}
\widehat{\Sigma}_{K} - \mathbb{E}[S^{(M)}] &= \frac{1}{K}\sum_{k=1}^{K} (Y_k - \mathbb{E}[Z_1]) \nonumber \\
&= \underbrace{\frac{1}{K}\sum_{k=1}^{K} (Y_k - Z_k)}_{\mathcal{B}\ \text{(bias)}} + \underbrace{\frac{1}{K}\sum_{k=1}^{K} (Z_k - \mathbb{E}[Z_k])}_{\mathcal{V}\ \text{(fluctuation)}}.
\label{eq:bias_variance_decomp}
\end{align}
By the elementary inequality $\|a + b\|^2 \le 2\|a\|^2 + 2\|b\|^2$:
\begin{equation}
\label{eq:mse_split}
\MSE = \mathbb{E}[\|\widehat{\Sigma}_K - \mathbb{E}[S^{(M)}]\|_{\mathrm{HS}}^2] \le 2\,\mathbb{E}[\|\mathcal{B}\|_{\mathrm{HS}}^2] + 2\,\mathbb{E}[\|\mathcal{V}\|_{\mathrm{HS}}^2].
\end{equation}
The discretization error $\mathcal{B} = K^{-1}\sum_k(Y_k - Z_k)$ captures both systematic and random components of the approximation. Our bound controls $\mathbb{E}[\|\mathcal{B}\|_{\mathrm{HS}}^2]$ without separating them (see Remark~\ref{rem:bias_conservatism} below). By Assumption~\ref{ass:gamma} (discretization approximation rate), with mesh $\delta = \Delta/(n-1)$:
\begin{equation}
\mathbb{E}[\|Y_k - Z_k\|_{\mathrm{HS}}^2] = \mathbb{E}[\|S^{(M)}(\mathbf{X}^{(n)}_k)_{I_k} - S^{(M)}(\mathbf{X})_{I_k}\|_{\mathrm{HS}}^2] \le C_{\mathrm{approx}}(\Delta,M)\, \delta^\gamma,
\end{equation}
where the distributional equality across blocks follows from Assumption~\ref{ass:S} (stationarity). Substituting $\delta = \Delta/(n-1)$:
\begin{equation}
\label{eq:bias_per_block}
\mathbb{E}[\|Y_k - Z_k\|_{\mathrm{HS}}^2] \le \frac{C_{\mathrm{approx}}(\Delta,M) \cdot \Delta^\gamma}{(n-1)^\gamma}.
\end{equation}

Since the squared Hilbert--Schmidt norm $\|\cdot\|_{\mathrm{HS}}^2$ is convex, Jensen's inequality applied to the average $K^{-1}\sum_k v_k$ gives $\|K^{-1}\sum_k v_k\|_{\mathrm{HS}}^2 \le K^{-1}\sum_k \|v_k\|_{\mathrm{HS}}^2$. Taking expectations:
\begin{equation}
\label{eq:bias_bound}
\mathbb{E}[\|\mathcal{B}\|_{\mathrm{HS}}^2] \le \frac{C_{\mathrm{approx}}(\Delta, M) \cdot \Delta^\gamma}{(n-1)^\gamma}.
\end{equation}

The values of $\gamma$ are given in Proposition~\ref{prop:gamma_values}; see Appendix~\ref{app:gamma} for proofs.

\begin{remark}[Conservatism of the bias bound]
\label{rem:bias_conservatism}
The bound~\eqref{eq:bias_bound} controls the $L^2$ pathwise approximation error without exploiting cancellations in the mean; in particular, $\|\mathbb{E}[Y_k - Z_k]\|_{\mathrm{HS}}^2$ may decay faster than $\mathbb{E}[\|Y_k - Z_k\|_{\mathrm{HS}}^2]$. This conservatism explains the gap between theoretical and empirical bias rates observed in Section~\ref{sec:numerical_results}.
\end{remark}

%%%%%%%%%%%%%%%%%%%%%%%%%%%%%%%%%%%%%%%%%%%%%%%%%%%%%%%%%%%%%%%%%%%%%%%%%%%%%%%
\subsection{Variance Analysis}
\label{app:variance}
%%%%%%%%%%%%%%%%%%%%%%%%%%%%%%%%%%%%%%%%%%%%%%%%%%%%%%%%%%%%%%%%%%%%%%%%%%%%%%%

The fluctuation $\mathcal{V} = K^{-1}\sum_k(Z_k - \mathbb{E}[Z_k])$ captures the statistical variance from averaging dependent block signatures. We exploit the orthogonal decomposition of $T^{(M)}(\mathbb{R}^d)$ to analyze level-1 (path increments) separately from levels $m \ge 2$ (iterated integrals), as these exhibit fundamentally different decay rates.

\subsubsection{Level-1 Variance: Telescoping}

\begin{lemma}[Level-1 telescoping]
\label{lem:level1_telescope_proof}
The level-1 component of the block-averaging estimator satisfies
\begin{equation}
\bar{S}^{(1)}_K := \frac{1}{K}\sum_{k=1}^K S^{(1)}(\mathbf{X})_{I_k} = \frac{1}{K}(X_{K\Delta} - X_0),
\end{equation}
with variance
\begin{equation}
\label{eq:level1_var}
\Var(\bar{S}^{(1)}_K) = \frac{1}{K^2}\Var(X_{K\Delta} - X_0) \le \frac{V_1(\Delta)}{K^2},
\end{equation}
where $V_1(\Delta) := \sup_{T > 0} \Var(X_T - X_0) \le 4\,\tr(\Cov(X_0))$ under Assumption~\ref{ass:S} (note that $V_1$ is independent of $\Delta$; we retain the argument for notational uniformity with $V_m(\Delta)$, $m \ge 2$).
\end{lemma}

\begin{proof}
The level-1 signature on block $I_k = [(k-1)\Delta, k\Delta]$ is the path increment: $S^{(1)}(\mathbf{X})_{I_k} = X_{k\Delta} - X_{(k-1)\Delta}$. Summing over blocks:
\begin{equation}
\sum_{k=1}^K S^{(1)}(\mathbf{X})_{I_k} = \sum_{k=1}^K (X_{k\Delta} - X_{(k-1)\Delta}) = X_{K\Delta} - X_0.
\end{equation}
Assumption~\ref{ass:S} gives marginal stationarity ($X_T \stackrel{d}{=} X_0$ for all $T$), hence $\mathbb{E}[X_T] = \mathbb{E}[X_0]$ and $X_T-X_0$ is centered. Writing $R^i(t) := \Cov(X_0^i, X_t^i)$ for the centered coordinate autocovariance, $\Var(X_T - X_0) = \mathbb{E}[\|X_T - X_0\|^2] = 2\sum_{i=1}^d(R^i(0) - R^i(T)) \le 4\sum_{i=1}^d R^i(0) = 4\,\tr(\Cov(X_0))$, using $|R^i(T)| \le R^i(0)$.
\end{proof}

\subsubsection{Levels $m \ge 2$: Orthogonal Decomposition by Tensor Level}

Let $W_k := Z_k - \mathbb{E}[Z_k]$ denote the centered block signature. By the stationarity of $\mathbf{X}$ (Assumption~\ref{ass:S}), the sequence $(W_k)_{k \ge 1}$ is strictly stationary with mean zero (Lemma~\ref{lem:stationarity}). The Hilbert--Schmidt inner product on $T^{(M)}(\mathbb{R}^d) = \bigoplus_{m=0}^{M} (\mathbb{R}^d)^{\otimes m}$ satisfies
\begin{equation}
\langle S, S' \rangle_{\mathrm{HS}} = \sum_{m=0}^{M} \langle S^{(m)}, S'^{(m)} \rangle,
\end{equation}
since tensor spaces of different degrees are orthogonal. Define the level-$m$ centered components $W_k^{(m)} := Z_k^{(m)} - \mathbb{E}[Z_k^{(m)}]$ and the level-$m$ sample mean $\bar{S}^{(m)}_K := K^{-1}\sum_{k=1}^K S^{(m)}(\mathbf{X})_{I_k}$, so that $\Var(\bar{S}^{(m)}_K) = \mathbb{E}\|K^{-1}\sum_{k=1}^K W_k^{(m)}\|^2$.

For stationary sequences, the variance of the sample mean admits the time-domain (autocovariance) representation (see \citealp[eq.~(2.4.2)]{BrockwellDavis2016} for the scalar case). For Hilbert-valued sequences, it follows from expanding $\|K^{-1}\sum_{k=1}^K W_k\|_{\mathrm{HS}}^2 = K^{-2}\sum_{j,k=1}^K \langle W_j, W_k \rangle_{\mathrm{HS}}$ and applying stationarity; cf.~\citep[Ch.~2]{Bosq2000}:
\begin{equation}
\label{eq:variance_timedomain}
\mathbb{E}\left[\left\| \frac{1}{K}\sum_{k=1}^K W_k \right\|_{\mathrm{HS}}^2\right] = \frac{1}{K} \left( c_0 + 2\sum_{h=1}^{K-1} \left(1 - \frac{h}{K}\right) c_h \right),
\end{equation}
where $c_h := \mathbb{E}[\langle W_1, W_{1+h} \rangle_{\mathrm{HS}}]$ is the lag-$h$ autocovariance for $h \ge 0$ (independent of the base index by stationarity; the factor of $2$ counts the two orientations of each lag-$h$ pair $(j,k)$ with $|j-k|=h$ in the double sum). By orthogonality:
\begin{equation}
\label{eq:cov_decomp}
c_h = \sum_{m=0}^{M} \mathbb{E}[\langle W_1^{(m)}, W_{1+h}^{(m)} \rangle] =: \sum_{m=0}^{M} c_{h,m}.
\end{equation}
Since $S^{(0)} = 1$ is deterministic, $c_{h,0} = 0$ for all $h$, and the nontrivial contributions begin at $m = 1$.

\subsubsection{Level-specific Covariance Bounds}

Under Assumption~\ref{ass:Aprime}, the bounds $|c_{h,m}| \le d^m\widetilde{C}_m\bigl(\Lambda_p(m/p)!\bigr)^{-2}\Delta^{2m\beta}\,\psi(h)$ are assumed directly, where $\psi(h) = e^{-\lambda h\Delta}$ under (E$'$) or $\psi(h) = (1+h)^{-\nu}$ under (P$'$). In that case, the proof proceeds directly to the levelwise variance bound below.

Under Assumption~\ref{ass:A}, we derive the same structure from the Hilbert-valued covariance inequality of Lemma~\ref{lem:hilbert_davydov}. Define the block $\sigma$-algebras $\mathcal{G}_k := \sigma(\mathbf{X}|_{I_k})$, where $I_k = [(k-1)\Delta, k\Delta]$. By the standing measurability convention in Section~\ref{sec:mathematical_framework}, each increment $\mathbf{X}_{s,t}$ is $\sigma(X_u : u \in [s,t])$-measurable (for canonical rough-path lifts of Gaussian processes this follows from a.s.\ convergence of piecewise-linear approximations on dyadic partitions of $[s,t]$; cf.~\citealp[Thm.~15.33(iv), Ex.~15.44]{FrizVictoir2010}). Since $\mathbf{X}$ has continuous sample paths (in $p$-variation), $\mathbf{X}|_{I_k}$ is generated by the countable collection of rational-time increments $\{\mathbf{X}_{s,t} : s, t \in \mathbb{Q} \cap I_k\}$, each measurable with respect to $\sigma(X_u : u \in I_k)$, so $\mathcal{G}_k := \sigma(\mathbf{X}|_{I_k}) \subseteq \sigma(X_u : u \in I_k) \subseteq \mathcal{F}_{\le k\Delta}$.

Since $Z_k$ is $\mathcal{G}_k$-measurable and the blocks $I_1$ and $I_{1+h}$ are separated by the time gap $(h-1)\Delta$ for $h \ge 1$, the mixing coefficient satisfies $\alpha(W_1^{(m)}, W_{1+h}^{(m)}) \le \alpha_{X}((h-1)\Delta)$. Under Assumption~\ref{ass:A}:
\begin{equation}
\label{eq:alpha_bound}
\alpha(W_1^{(m)}, W_{1+h}^{(m)}) \le \begin{cases}
C_X\, e^{-\lambda_X(h-1)\Delta} & \text{under (E)}, \\
C_X\, (1 + (h-1)\Delta)^{-\nu} & \text{under (P) with exponent } \nu.
\end{cases}
\end{equation}
Thus the mixing bound for $W_1^{(m)}$ and $W_{1+h}^{(m)}$ follows directly from their blockwise measurability and the separation gap between the underlying sigma-fields. By stationarity, the same bound holds for every lag-$h$ pair $(W_k^{(m)}, W_{k+h}^{(m)})$ with $k \ge 1$.

\paragraph{Lag-zero covariance} For $h = 0$, we have $c_{0,m} = \Var(Z_1^{(m)}) \le \mathbb{E}[\|Z_1^{(m)}\|_{\mathrm{HS}}^2]$. By Lemma~\ref{lem:level_moments}:
\begin{equation}
\label{eq:c0m_bound}
c_{0,m} \le \frac{d^m C^{\mathrm{sig}}_{m,2}}{\bigl(\Lambda_p(m/p)!\bigr)^2}\Delta^{2m\beta}.
\end{equation}

\paragraph{Cross-lag covariances: exponential mixing} Under Assumption~\ref{ass:A}(E), we apply Lemma~\ref{lem:hilbert_davydov} with $\mathsf{E} = (\mathbb{R}^d)^{\otimes m}$, $q = 4$ (hence $r = 2$), and $D = d^m$. Using stationarity:
\begin{equation}
\label{eq:chm_exp}
|c_{h,m}| \le 2\sqrt{2}\, d^m\, \alpha^{1/2}\, \|W_1^{(m)}\|_{L^4(\mathsf{E})}^2.
\end{equation}

We bound $\|W_1^{(m)}\|_{L^4(\mathsf{E})}^2$. For centered $W = Z - \mathbb{E}[Z]$, the triangle inequality and $L^p$-contraction give $\|W\|_{L^4(\mathsf{E})} \le 2\|Z\|_{L^4(\mathsf{E})}$. By Lemma~\ref{lem:level_moments} with $q = 4$:
\begin{equation}
\label{eq:L4_bound}
\|W_1^{(m)}\|_{L^4(\mathsf{E})}^2 \le 4\|Z_1^{(m)}\|_{L^4(\mathsf{E})}^2 \le \frac{4d^m \sqrt{C^{\mathrm{sig}}_{m,4}}}{\bigl(\Lambda_p(m/p)!\bigr)^2}\Delta^{2m\beta}.
\end{equation}

Combining with the mixing bound $\alpha^{1/2} \le C_X^{1/2} e^{-\lambda'(h-1)\Delta}$ where $\lambda' := \lambda_X/2$ (the square root halves the exponential decay rate; any fixed positive rate yields a summable geometric envelope, so only the constants below are affected):
\begin{equation}
\label{eq:chm_bound_exp}
|c_{h,m}| \le \frac{A_m}{\bigl(\Lambda_p(m/p)!\bigr)^2}\Delta^{2m\beta} e^{-\lambda'(h-1)\Delta},
\end{equation}
where $A_m := 8\sqrt{2}\, C_X^{1/2}d^{2m}\sqrt{C^{\mathrm{sig}}_{m,4}}$. The choice $q = 4$ (hence $r = 2$) is for concreteness: applying Lemma~\ref{lem:hilbert_davydov} instead with arbitrary $r \in (1, \infty)$ and $q = 2r/(r-1)$ --- signatures have moments of all orders under Assumption~\ref{ass:M} (Lemma~\ref{lem:level_moments}) --- replaces the decay rate $\lambda_X/2$ in~\eqref{eq:chm_bound_exp} by $\lambda_X/r$ and $A_m$ by an $r$-dependent constant, exactly as in the $r$-optimization of the polynomial case below. As $r$ ranges over $(1,\infty)$ this attains every rate $\lambda' \in (0, \lambda_X)$, which is the form quoted in Remark~\ref{rem:A_implies_Aprime}.

\paragraph{Cross-lag covariances: polynomial mixing} Under Assumption~\ref{ass:A}(P) with exponent $\nu$, eq.~\eqref{eq:alpha_bound} gives $\alpha(W_1^{(m)}, W_{1+h}^{(m)}) \le C_X(1+(h-1)\Delta)^{-\nu}$. For $h = 1$, this is $C_X$. For $h \ge 2$, we bound $(1+(h-1)\Delta)^{-\nu} \le ((h-1)\Delta)^{-\nu} \le (2/\Delta)^{\nu}\, h^{-\nu}$, since $h - 1 \ge h/2$. Writing $C_{X,\Delta} := C_X \max(1,(2/\Delta)^{\nu})$, which is finite for each fixed $\Delta > 0$, we therefore have
\begin{equation}
\label{eq:alpha_h_poly}
\alpha_h := \alpha(W_1^{(m)}, W_{1+h}^{(m)}) \le C_{X,\Delta}\, h^{-\nu} \quad \text{for all } h \ge 1.
\end{equation}
To achieve the optimal decay rate, we optimize over the exponent $r$ in Lemma~\ref{lem:hilbert_davydov}.

For $r > 1$, the constraint $2/q + 1/r = 1$ gives $q = 2r/(r-1)$. As $r \downarrow 1$, we have $q \uparrow \infty$, and Lemma~\ref{lem:hilbert_davydov} yields:
\begin{equation}
\label{eq:chm_poly_general}
|c_{h,m}| \le 2^{1+1/r}\, d^m\, \alpha_h^{1/r}\, \|W_1^{(m)}\|_{L^q(\mathsf{E})}^2.
\end{equation}
Since signatures have moments of all orders under Assumption~\ref{ass:M}, $\|W_1^{(m)}\|_{L^q(\mathsf{E})}$ is finite for all $q < \infty$.

\begin{proposition}[Polynomial mixing rate]
\label{prop:poly_mixing_rate}
Under Assumption~\ref{ass:A}(P) with exponent $\nu > 0$, for any $\varepsilon \in (0, \nu)$ (the Rio exponent loss, distinct from the discretization losses $\varepsilon_0, \varepsilon_1$ in Proposition~\ref{prop:gamma_values}) there exists a finite constant $C_{\varepsilon,m} = C_{\varepsilon,m}(\Delta)$, depending on the level $m$, on $\varepsilon$, and on $\Delta$, such that
\begin{equation}
\label{eq:chm_poly_optimal}
|c_{h,m}| \le \frac{C_{\varepsilon,m}}{\bigl(\Lambda_p(m/p)!\bigr)^2}\Delta^{2m\beta} h^{-\nu + \varepsilon} \quad \text{for all } h \ge 1.
\end{equation}
Consequently, the variance satisfies
\begin{equation}
\label{eq:var_poly_rate}
\sum_{m=2}^M \Var(\bar{S}^{(m)}_K) = O\bigl(K^{-\min(1, \nu-\varepsilon)}\bigr) \quad \text{for any } \varepsilon \in (0, \nu),
\end{equation}
which is $O(K^{-1})$ when $\nu > 1$ (taking $\varepsilon < \nu-1$) and $O(K^{-\nu+\varepsilon})$ when $\nu \le 1$.
\end{proposition}

\begin{proof}
Choose $r = 1 + \varepsilon/\nu$ so that $1/r = \nu/(\nu + \varepsilon)$. Then by~\eqref{eq:alpha_h_poly}, $\alpha_h^{1/r} \le (C_{X,\Delta}\, h^{-\nu})^{\nu/(\nu+\varepsilon)} = C_{X,\Delta}^{\nu/(\nu+\varepsilon)}\, h^{-\nu^2/(\nu+\varepsilon)}$. Since $\nu^2/(\nu+\varepsilon) = \nu - \nu\varepsilon/(\nu+\varepsilon) > \nu - \varepsilon$, we obtain the decay $h^{-\nu+\varepsilon}$ (absorbing $\Delta$-dependent constants into $C_{\varepsilon,m}$).

The corresponding $q = 2(\nu+\varepsilon)/\varepsilon$ is finite, so Lemma~\ref{lem:level_moments} gives
\[
\|W_1^{(m)}\|_{L^q(\mathsf{E})}^2
\le \frac{4\,d^m\,(C^{\mathrm{sig}}_{m,q})^{2/q}}{\bigl(\Lambda_p(m/p)!\bigr)^2}\Delta^{2m\beta}.
\]
The $L^q$ norm depends on $\varepsilon$ but is bounded for each fixed $\varepsilon \in (0, \nu)$, yielding~\eqref{eq:chm_poly_optimal} after absorbing the finite level constants into $C_{\varepsilon,m}$.

For the variance, using~\eqref{eq:variance_timedomain}:
\begin{equation}
\Var(\bar{S}^{(m)}_K) \le \frac{1}{K}\left(c_{0,m} + 2\sum_{h=1}^{K-1} |c_{h,m}|\right).
\end{equation}
With $|c_{h,m}| = O(h^{-\nu+\varepsilon})$:
\begin{itemize}
\item If $\nu - \varepsilon > 1$: $\sum_{h=1}^{K-1} h^{-\nu+\varepsilon} = O(1)$, so $\Var = O(K^{-1})$.
\item If $\nu - \varepsilon < 1$: $\sum_{h=1}^{K-1} h^{-\nu+\varepsilon} = O(K^{1-\nu+\varepsilon})$, so $\Var = O(K^{-\nu+\varepsilon})$.
\item If $\nu - \varepsilon = 1$: apply~\eqref{eq:chm_poly_optimal} with $\varepsilon/2$ in place of $\varepsilon$, so that $\nu - \varepsilon/2 > 1$ and the first case gives $\Var = O(K^{-1}) = O(K^{-\min(1,\nu-\varepsilon)})$.
\end{itemize}
Since $\varepsilon \in (0, \nu)$ is otherwise arbitrary, we obtain~\eqref{eq:var_poly_rate}.
\end{proof}

\subsubsection{The Levelwise Variance Bound}

Define the level-$m$ variance contribution:
\begin{equation}
V_m(\Delta) := c_{0,m} + 2\sum_{h=1}^{\infty} |c_{h,m}|.
\end{equation}
Under exponential decay, or polynomial decay with $\nu > 1$, this series converges absolutely (for Assumption~\ref{ass:A}(P), choose the Rio loss in Proposition~\ref{prop:poly_mixing_rate} smaller than $\nu - 1$; under Assumption~\ref{ass:Aprime}(P$'$), use $\sum_h h^{-\nu} < \infty$). Under polynomial mixing with $\nu \le 1$, we instead work with the finite-$K$ analogue $V_m^{(K)}(\Delta) := c_{0,m} + 2\sum_{h=1}^{K-1}|c_{h,m}|$. The variance bounds of Proposition~\ref{prop:poly_mixing_rate} are derived directly from this truncated sum.

Summing the geometric series in~\eqref{eq:chm_bound_exp}:
\begin{equation}
\sum_{h=1}^{\infty} |c_{h,m}| \le \frac{A_m}{\bigl(\Lambda_p(m/p)!\bigr)^2}\Delta^{2m\beta}\frac{1}{1 - e^{-\lambda'\Delta}}.
\end{equation}

Combining with~\eqref{eq:c0m_bound} and absorbing $C^{\mathrm{sig}}_{m,2}$ and $A_m d^{-2m}$ into a single level-dependent constant
\begin{equation}
\label{eq:Ctilde_def}
\widetilde{C}_m := \max\bigl(C^{\mathrm{sig}}_{m,2},\, 8\sqrt{2}\, C_X^{1/2}\sqrt{C^{\mathrm{sig}}_{m,4}}\bigr),
\end{equation}
yields
\begin{equation}
\label{eq:Vm_bound}
V_m(\Delta) \le \frac{d^m\widetilde{C}_m}{\bigl(\Lambda_p(m/p)!\bigr)^2}\Delta^{2m\beta}\,\Phi_m(\Delta), \qquad \Phi_m(\Delta) := 1 + \frac{2 d^m}{1 - e^{-\lambda'\Delta}}.
\end{equation}
The inner $d^m$ in $\Phi_m$ absorbs the dimension factor from Lemma~\ref{lem:hilbert_davydov} (which is not present in the direct-covariance bound below, where $d^m$ is already in the assumed envelope).
Under Assumption~\ref{ass:Aprime}, the same structure holds with the assumed bound
\[
|c_{h,m}| \le d^m\widetilde{C}_m\bigl(\Lambda_p(m/p)!\bigr)^{-2}\Delta^{2m\beta}\psi(h),
\qquad h \ge 1,
\]
where $\widetilde{C}_m$ is enlarged if necessary to dominate $C^{\mathrm{sig}}_{m,2}$ from Lemma~\ref{lem:level_moments} so the lag-zero term $c_{0,m}$ fits the same envelope. When $\psi$ is summable (under (E$'$), or (P$'$) with $\nu > 1$), this yields
\[
V_m(\Delta) \le d^m\widetilde{C}_m\bigl(\Lambda_p(m/p)!\bigr)^{-2}\Delta^{2m\beta}\Phi_m(\Delta),
\qquad
\Phi_m(\Delta) := 1 + 2\sum_{h=1}^{\infty}\psi(h).
\]
Under (P$'$) with $\nu \le 1$, the infinite series $\Phi_m(\Delta)$ diverges. We instead bound the truncation
\[
V_m^{(K)}(\Delta) \le d^m\widetilde{C}_m\bigl(\Lambda_p(m/p)!\bigr)^{-2}\Delta^{2m\beta}\Phi_m^{(K)}(\Delta),
\qquad
\Phi_m^{(K)}(\Delta) := 1 + 2\sum_{h=1}^{K-1}\psi(h),
\]
and the resulting rate is recorded in Proposition~\ref{prop:mixing_exponent}(iii).

\subsubsection{Aggregation and Total Variance}

Combining level-1 (Lemma~\ref{lem:level1_telescope_proof}) with levels $m \ge 2$, the total variance decomposes as
\begin{equation}
\mathbb{E}[\|\mathcal{V}\|_{\mathrm{HS}}^2] = \Var(\bar{S}^{(1)}_K) + \sum_{m=2}^M \Var(\bar{S}^{(m)}_K).
\end{equation}
By~\eqref{eq:level1_var}, the level-1 term is $O(K^{-2})$. For levels $m \ge 2$, define the finite-$K$ variance constant $V_m^{(K)}(\Delta) := c_{0,m} + 2\sum_{h=1}^{K-1}|c_{h,m}|$. Using $(1 - h/K) \le 1$ in~\eqref{eq:variance_timedomain}:
\begin{align}
\sum_{m=2}^M \Var(\bar{S}^{(m)}_K) &\le \frac{1}{K}\sum_{m=2}^{M} V_m^{(K)}(\Delta).
\label{eq:variance_bound}
\end{align}
Under Assumption~\ref{ass:A}(E), Assumption~\ref{ass:Aprime}(E$'$), or either Assumption~\ref{ass:A}(P) or Assumption~\ref{ass:Aprime}(P$'$) with $\nu > 1$, the limit $V_m(\Delta) := \lim_{K \to \infty} V_m^{(K)}(\Delta)$ converges, and the bound gives $O(K^{-1})$ (i.e., $\eta = 1$). Under Assumption~\ref{ass:A}(P) with $\nu \le 1$, the truncated sum satisfies $V_m^{(K)}(\Delta) = O(K^{1-\nu+\varepsilon})$ for any $\varepsilon \in (0, \nu)$ by Proposition~\ref{prop:poly_mixing_rate}, so the bound gives $O(K^{-\nu+\varepsilon})$ (i.e., $\eta = \nu$). Under Assumption~\ref{ass:Aprime}(P$'$) with $\nu \le 1$, the direct covariance decay gives $V_m^{(K)}(\Delta) = O(K^{1-\nu})$ for $\nu < 1$ and $V_m^{(K)}(\Delta) = O(\log K)$ at $\nu = 1$. This yields the rates of Proposition~\ref{prop:mixing_exponent}(iii) without $\varepsilon$-loss.

\begin{proposition}[Variance exponent]
\label{prop:mixing_exponent}
The variance bound~\eqref{eq:variance_bound} holds with:
\begin{enumerate}
\item[\emph{(i)}] $\eta = 1$ under Assumption~\ref{ass:A}(E) or~\ref{ass:Aprime}(E$'$) (exponential decay).
\item[\emph{(ii)}] Under Assumption~\ref{ass:A}(P) with exponent $\nu$ ($\alpha$-mixing): $\eta = 1$ with variance $O(K^{-1})$ when $\nu > 1$; $\eta = \nu$ with variance $O(K^{-\nu+\varepsilon})$ for any $\varepsilon \in (0, \nu)$ when $\nu \le 1$ (the $\varepsilon$-loss arises from Rio's exponent in Lemma~\ref{lem:hilbert_davydov}).
\item[\emph{(iii)}] Under Assumption~\ref{ass:Aprime}(P$'$) with exponent $\nu$ (direct covariance decay): $\eta = 1$ if $\nu > 1$; $\eta = \nu$ if $\nu < 1$; $\eta = 1$ with a logarithmic correction $O(K^{-1}\log K)$ if $\nu = 1$. No $\varepsilon$-loss arises.
\end{enumerate}
\end{proposition}

\begin{proof}
Part (i) follows from geometric/exponential decay of $|c_{h,m}|$: $\sum_{h=1}^\infty |c_{h,m}| < \infty$, so $\Var(\bar{S}^{(m)}_K) = O(K^{-1})$.

Part (ii) is Proposition~\ref{prop:poly_mixing_rate}.

Part (iii): under Assumption~\ref{ass:Aprime}(P$'$), $|c_{h,m}| \le C h^{-\nu}$ directly. If $\nu > 1$, then $\sum_h h^{-\nu} < \infty$ and $\Var = O(K^{-1})$. If $0 < \nu < 1$, then $\sum_{h=1}^{K-1} h^{-\nu} = \Theta(K^{1-\nu})$, so $\Var = O(K^{-\nu})$. If $\nu = 1$, then $\sum_{h=1}^{K-1} h^{-1} = \Theta(\log K)$, so $\Var = O(K^{-1}\log K)$.
\end{proof}

%%%%%%%%%%%%%%%%%%%%%%%%%%%%%%%%%%%%%%%%%%%%%%%%%%%%%%%%%%%%%%%%%%%%%%%%%%%%%%%
\subsection{Completion of Proof}
\label{app:completion}
%%%%%%%%%%%%%%%%%%%%%%%%%%%%%%%%%%%%%%%%%%%%%%%%%%%%%%%%%%%%%%%%%%%%%%%%%%%%%%%

With the bias and variance bounds established, we assemble the proof of Theorem~\ref{thm:main_result}.

\begin{proof}[Proof of Theorem~\ref{thm:main_result}]
Combining the bias bound~\eqref{eq:bias_bound}, level-1 variance bound~\eqref{eq:level1_var}, and levels $m \ge 2$ variance bound~\eqref{eq:variance_bound} via~\eqref{eq:mse_split}, we obtain in the summable regime --- where the limits $V_m(\Delta)$ below exist and Proposition~\ref{prop:mixing_exponent} gives $\eta = 1$ ---
\begin{equation}
\label{eq:mse_refined}
\MSE \le \frac{2\,C_{\mathrm{approx}}(\Delta,M)\,\Delta^\gamma}{(n-1)^\gamma} + \frac{2V_1(\Delta)}{K^2} + \frac{2}{K^\eta}\sum_{m=2}^{M} V_m(\Delta),
\end{equation}
where $V_1(\Delta) \le 4\,\tr(\Cov(X_0))$ bounds level-1 variance, $V_m(\Delta) := \lim_{K\to\infty} V_m^{(K)}(\Delta)$ is finite for $m \ge 2$ (given by the closed form~\eqref{eq:Vm_bound} under exponential decay, and by the summable polynomial series $c_{0,m} + 2\sum_{h \ge 1}|c_{h,m}|$ otherwise), and the variance exponent $\eta$ is specified by Proposition~\ref{prop:mixing_exponent}. The non-summable cases (Assumption~\ref{ass:A}(P) or~\ref{ass:Aprime}(P$'$) with $\nu \le 1$) are treated below via the truncated $V_m^{(K)}$ from Proposition~\ref{prop:mixing_exponent}(ii)--(iii).

Since $\eta \le 1$, the telescoping bound $2V_1(\Delta)K^{-2}$ from~\eqref{eq:mse_refined} is dominated by $2V_1(\Delta)K^{-\eta}$ for $K \ge 1$, so the level-1 term is absorbed into the $K^{-\eta}$ envelope. Set $C_{\mathrm{bias}}(\Delta,M) := 2C_{\mathrm{approx}}(\Delta,M)\Delta^\gamma$. We obtain the bounds of Theorem~\ref{thm:main_result} case-wise:

\emph{Summable regime} (Assumption~\ref{ass:A}(E), Assumption~\ref{ass:Aprime}(E$'$), or Assumption~\ref{ass:A}(P) / Assumption~\ref{ass:Aprime}(P$'$) with $\nu > 1$): defining $V(\Delta) := 2V_1(\Delta) + 2\sum_{m=2}^{M} V_m(\Delta) < \infty$ with $V_m(\Delta)$ finite by Proposition~\ref{prop:mixing_exponent} (given explicitly by~\eqref{eq:Vm_bound} under exponential decay) and $V_1(\Delta)$ from Lemma~\ref{lem:level1_telescope_proof} (the sharper $K^{-2}$ telescoping rate is absorbed into the $K^{-\eta}$ envelope since $\eta \le 1$),
\begin{equation}
\label{eq:mse_final}
\MSE \le \frac{C_{\mathrm{bias}}(\Delta,M)}{(n-1)^\gamma} + \frac{V(\Delta)}{K^\eta},
\end{equation}
matching~\eqref{eq:mse_bound}.

\emph{Non-summable $\alpha$-mixing} (Assumption~\ref{ass:A}(P) with $0 < \nu \le 1$): for any $\varepsilon \in (0, \nu)$, set
\[
V_\varepsilon(\Delta) := 2V_1(\Delta) + 2\sup_{K \ge 1} K^{\nu - 1 - \varepsilon}\sum_{m=2}^M V_m^{(K)}(\Delta) < \infty,
\]
where finiteness follows from Proposition~\ref{prop:poly_mixing_rate} since $V_m^{(K)}(\Delta) = O(K^{1-\nu+\varepsilon})$ for $m \ge 2$ when $\nu \le 1$. We emphasize that $V_\varepsilon(\Delta)$ is finite for each fixed $\varepsilon \in (0, \nu)$ but is not asserted to be uniform as $\varepsilon \downarrow 0$: the Rio exponent $q = 2(\nu+\varepsilon)/\varepsilon$ in Proposition~\ref{prop:poly_mixing_rate} blows up as $\varepsilon \downarrow 0$, and the Wiener chaos hypercontractivity bound $\|W^{(m)}\|_{L^q} \lesssim (q-1)^{m/2}\|W^{(m)}\|_{L^2}$~\citep[Thm.~5.10]{Janson1997} correspondingly diverges. The level-1 contribution from Lemma~\ref{lem:level1_telescope_proof} satisfies $\Var(\bar{S}^{(1)}_K) \le V_1(\Delta)/K^2$, which is absorbed into the additive $V_1(\Delta)$ in $V_\varepsilon$. Then
\[
\MSE \le \frac{C_{\mathrm{bias}}(\Delta,M)}{(n-1)^\gamma} + \frac{V_\varepsilon(\Delta)}{K^{\nu-\varepsilon}},
\]
matching case~(ii) of Theorem~\ref{thm:main_result}.

\emph{Non-summable direct covariance decay} (Assumption~\ref{ass:Aprime}(P$'$) with $\nu \le 1$): the direct decay $|c_{h,m}| \le C\, h^{-\nu}$ gives finite partial sums without $\varepsilon$-loss. For $0 < \nu < 1$, set
\[
V_{\mathrm{poly}}(\Delta) := 2V_1(\Delta) + 2\sup_{K \ge 1} K^{\nu-1}\sum_{m=2}^M V_m^{(K)}(\Delta) < \infty;
\]
This constant is finite by Proposition~\ref{prop:mixing_exponent}(iii), since $V_m^{(K)} = O(K^{1-\nu})$ for $0 < \nu < 1$.
Then
\[
\MSE \le \frac{C_{\mathrm{bias}}(\Delta,M)}{(n-1)^\gamma} + \frac{V_{\mathrm{poly}}(\Delta)}{K^\nu}.
\]
For $\nu = 1$, set $V_{\mathrm{log}}(\Delta) := 2V_1(\Delta) + 2\sup_{K \ge 2} (\log K)^{-1}\sum_{m=2}^M V_m^{(K)}(\Delta) < \infty$. This constant is finite by Proposition~\ref{prop:mixing_exponent}(iii), since $V_m^{(K)} = O(\log K)$ at $\nu = 1$. Then, for all $K \ge 2$,
\[
\MSE \le \frac{C_{\mathrm{bias}}(\Delta,M)}{(n-1)^\gamma} + V_{\mathrm{log}}(\Delta)\, \frac{\log K}{K}.
\]
Together, these establish Theorem~\ref{thm:main_result}.
\end{proof}

% Appendix B: Verification of Discretization Rates (Proposition~\ref{prop:gamma_values})
%%%%%%%%%%%%%%%%%%%%%%%%%%%%%%%%%%%%%%%%%%%%%%%%%%%%%%%%%%%%%%%%%%%%%%%%%%%%%%%
\section{Verification of Discretization Rates}
\label{app:gamma}
%%%%%%%%%%%%%%%%%%%%%%%%%%%%%%%%%%%%%%%%%%%%%%%%%%%%%%%%%%%%%%%%%%%%%%%%%%%%%%%

It remains to verify Assumption~\ref{ass:gamma} for the three process classes stated in Proposition~\ref{prop:gamma_values}. The discretization rate $\gamma$ quantifies how well piecewise-linear signatures approximate true signatures as the mesh $\delta \to 0$.

%%%%%%%%%%%%%%%%%%%%%%%%%%%%%%%%%%%%%%%%%%%%%%%%%%%%%%%%%%%%%%%%%%%%%%%%%%%%%%%
\subsection{Gaussian Semimartingales with Finite 2D \texorpdfstring{$1$}{1}-Variation Covariance ($\gamma = 1 - \varepsilon$ for any $\varepsilon > 0$; $\gamma = 1$ at level~2)}
\label{app:gamma_semi}
%%%%%%%%%%%%%%%%%%%%%%%%%%%%%%%%%%%%%%%%%%%%%%%%%%%%%%%%%%%%%%%%%%%%%%%%%%%%%%%

We work with centered continuous Gaussian semimartingales admitting the decomposition $X = X_0 + \mathsf{N} + A$ where:
\begin{itemize}
\item $X = (X^1, \ldots, X^d)$ has \emph{independent coordinates}, each $X^j$ a centered Gaussian process;
\item $\mathsf{N} = (\mathsf{N}^1, \ldots, \mathsf{N}^d)$ is a continuous local martingale \emph{with respect to the filtration} $(\mathcal{F}_t)$, $\mathcal{F}_t := \sigma(X_s : s \le t)$ --- so that $X = X_0 + \mathsf{N} + A$ is the semimartingale decomposition of $X$ in its own filtration --- satisfying $d\langle \mathsf{N}^j\rangle_t/dt \le c_\mathsf{N}^2$ a.s.\ on $[0, \Delta]$;
\item $A = (A^1, \ldots, A^d)$ is a continuous finite-variation adapted process with $\mathbb{E}[(dA^j_u/du)^2] \le c_A^2$ for each $j \in \{1, \ldots, d\}$ and $u \in [0, \Delta]$;
\item each drift coordinate $A^j$ (and hence, by Fubini plus absolute continuity, its density $dA^j_u/du$) is adapted to $\sigma(X^j_s : s \le u)$ alone; in particular, $A^j$ and $dA^j_u/du$ are independent of $X^k$ for $j \ne k$;
\item each coordinate covariance $R^j(s, t) := \mathbb{E}[X^j_s X^j_t]$ has finite 2D $1$-variation on $[0, \Delta]^2$, a property that follows from the bounded martingale quadratic-variation density and $L^2$ drift density above. Indeed, rectangle increments annihilate the constant and single-variable terms, leaving the martingale, drift, and cross contributions: the martingale part $\mathbb{E}[\mathsf{N}^j_s \mathsf{N}^j_t] = \mathbb{E}[\langle \mathsf{N}^j\rangle_{s\wedge t}]$ is a non-decreasing function of $s \wedge t$ with total increment at most $c_\mathsf{N}^2\Delta$, hence has 2D $1$-variation $\le c_\mathsf{N}^2\Delta$; the drift part $\mathbb{E}[A^j_s A^j_t]$ is absolutely continuous with density $\mathbb{E}[(dA^j_u/du)(dA^j_v/dv)]$ bounded by $c_A^2$ (Cauchy--Schwarz), hence has 2D $1$-variation $\le c_A^2\Delta^2$; and for each cross part, the martingale projection $\mathbb{E}[\mathsf{N}^j_{s,s'} \mid \mathcal{F}_v] = \mathsf{N}^j_{s\wedge v,\,s'\wedge v}$ and Fubini express the rectangle increment over $[s,s']\times[t,t']$ as $\int_t^{t'}\mathbb{E}\bigl[\mathsf{N}^j_{s\wedge v,\,s'\wedge v}\,(dA^j_v/dv)\bigr]\,dv$; for each fixed $v$, Cauchy--Schwarz across the partition followed by Bessel's inequality for the pairwise orthogonal martingale increments $\{\mathsf{N}^j_{s_k\wedge v,\,s_{k+1}\wedge v}\}_k$ along any partition $\{s_k\}$ of $[0,\Delta]$ gives $\sum_k \bigl|\mathbb{E}[\mathsf{N}^j_{s_k\wedge v,\,s_{k+1}\wedge v}\,(dA^j_v/dv)]\bigr| \le \bigl(\sum_k \mathbb{E}[(\mathsf{N}^j_{s_k\wedge v,\,s_{k+1}\wedge v})^2]\bigr)^{1/2} c_A \le c_\mathsf{N} c_A \Delta^{1/2}$, whence each cross part has 2D $1$-variation $\le c_\mathsf{N} c_A \Delta^{3/2}$; see also~\citep[Sec.~15.3]{FrizVictoir2010}.
\end{itemize}
This subclass contains the paper's main applications: stationary Ornstein--Uhlenbeck ($dX = -\theta X\,dt + \sigma\,dB$, so $\mathsf{N} = \sigma B$ with $c_\mathsf{N} = \sigma$ and $A_t = -\theta\int_0^t X_s\,ds$ with $\mathbb{E}[(dA^j_u/du)^2] = \theta^2\sigma^2/(2\theta) = \theta\sigma^2/2$); and fOU at $H = 1/2$ (that is, stationary OU). This scope is \emph{not} claimed to extend to the full Gaussian rough-path framework of~\citet{FrizRiedel2014} with $\rho > 1$ (e.g., fBm at $H \ne 1/2$, which is not a semimartingale).

Under these assumptions, $X$ satisfies Assumption~\ref{ass:M} with $\beta = 1/p$ for any $p > 2$: the three 2D $1$-variation bounds above localize from $[0,\Delta]^2$ to any square $[s,t]^2 \subseteq [0,\Delta]^2$ (replace $\Delta$ by $t-s$ in their derivations), giving $V_1(R^j;[s,t]^2) \le c_\mathsf{N}^2\,(t-s) + 2\,c_\mathsf{N}c_A\,(t-s)^{3/2} + c_A^2\,(t-s)^2 \le C\,(t-s)$ --- a H\"older-dominated control --- so the Gaussian integrability estimates of~\citet[Thm.~15.33(iii)]{FrizVictoir2010} (applied with $\rho = 1$) bound $\mathbb{E}[\|\mathbf{X}\|^r_{p\text{-var};[s,t]}]$ by $C_r\,(t-s)^{r/p}$ for every $r \ge 1$; for the martingale part alone, the same moment control also follows from the Burkholder--Davis--Gundy inequality in homogeneous $p$-variation norm~\citep[Sec.~14.1, Thm.~14.12]{FrizVictoir2010}. The canonical rough-path lift $\mathbf{X}$ --- the limit of lifted piecewise-linear approximations supplied by~\citet[Thm.~15.33]{FrizVictoir2010} --- coincides with the Stratonovich (semimartingale) lift: piecewise-linear approximations of a continuous local martingale converge in probability to its martingale lift~\citep[Thm.~14.16]{FrizVictoir2010}, and the remaining level-2 cross and drift terms involve the finite-variation part $A$, for which convergence holds by uniform convergence of the interpolants together with Riemann--Stieltjes (respectively It\^o-isometry) continuity. The finite 2D $1$-variation covariance is also the Gaussian input for the coordinate $L^2$ estimate in~\citet[Prop.~15.24]{FrizVictoir2010}; finite-chaos membership follows by taking $L^2$ limits of polynomial iterated integrals along piecewise-linear approximations, as in the Wiener chaos argument below. Throughout, write $\Delta X_i := X_{t_{i+1}} - X_{t_i}$, $e_u := X_u - X^{(n)}_u$, and $\mathcal{F}_t := \sigma(X_s : s \le t)$. For the increment bound, write $X^j = X^j_0 + \mathsf{N}^j + A^j$, so $(X^j_t - X^j_s)^2 \le 2(\mathsf{N}^j_t - \mathsf{N}^j_s)^2 + 2(A^j_t - A^j_s)^2$; after taking expectations, applying It\^o isometry to the martingale part, and applying Cauchy--Schwarz in time to the drift,
\begin{equation}
\label{eq:CR_derivation}
\mathbb{E}[(X^j_t - X^j_s)^2] \le 2\,c_\mathsf{N}^2\,|t - s| + 2\,c_A^2\,|t - s|^2 \le (2\,c_\mathsf{N}^2 + 2\,c_A^2\,\Delta)\,|t - s| =: C_R\,|t - s|
\end{equation}
for all $s, t \in [0, \Delta]$. The defining bound $C_R \ge 2\,c_\mathsf{N}^2 + 2\,c_A^2\,\Delta$ is loose (for standard Brownian motion $\mathbb{E}[(B^j_t - B^j_s)^2] = |t-s|$ exactly), but it suffices for the rate analysis below.

We establish Assumption~\ref{ass:gamma} with:
\begin{itemize}
\item $\gamma = 1$ exactly at signature level~2 (direct $L^2$ estimate below);
\item $\gamma = 1 - \varepsilon$ for any $\varepsilon > 0$ at truncation level $M$, via Friz--Riedel Theorem~5 applied at truncation level $N = M$.
\end{itemize}

\paragraph{Level 1: exact agreement}
$S^{(1)}(\mathbf{X})_{0, \Delta} = X_\Delta - X_0 = S^{(1)}(\mathbf{X}^{(n)})_{0, \Delta}$.

\paragraph{Level 2: direct $L^2$ estimate}
Throughout the level-2 analysis, $a \wedge b := \tfrac{1}{2}(a \otimes b - b \otimes a)$ denotes the antisymmetric tensor product.

The symmetric part of $S^{(2)}$ depends only on endpoints (see~\citealp[eq.~(2.6)]{FrizHairer2020}) and therefore coincides for $\mathbf{X}$ and $\mathbf{X}^{(n)}$. Applying Chen's identity (\citealp[Thm.~7.11]{FrizVictoir2010} for the piecewise-linear lift $\mathbf{X}^{(n)}$; for $\mathbf{X}$, by the multiplicativity of the weak geometric rough-path lift supplied by~\citealp[Thm.~15.33]{FrizVictoir2010}) at level~2 to the grid partition, together with the fact that piecewise-linear segments have zero L\'evy area and $X^{(n)}_{t_{i+1}} - X^{(n)}_{t_i} = \Delta X_i$, gives the cancellation
\begin{equation}
\label{eq:area_telescope}
\Delta S^{(2)}_{0, \Delta} := S^{(2)}(\mathbf{X})_{0, \Delta} - S^{(2)}(\mathbf{X}^{(n)})_{0, \Delta} = \sum_{i=0}^{n-2}\mathbb{A}_{t_i, t_{i+1}},
\end{equation}
where $\mathbb{A}_{s, t} = \int_s^t(X_u - X_s) \wedge \circ dX_u$. Stratonovich--It\^o conversion gives $\int Y \wedge \circ dX = \int Y \wedge dX + \tfrac{1}{2}\int d[Y, X]$ with $Y_u = X_u - X_{t_i}$; the covariation $[Y, X] = [X, X] = \langle \mathsf{N}\rangle$ (since $A$ has finite variation, contributing zero covariation), and $d\langle \mathsf{N}^a, \mathsf{N}^b\rangle_u$ is symmetric in $(a, b)$. Under the antisymmetric wedge, $\sum_{a, b} d\langle \mathsf{N}^a, \mathsf{N}^b\rangle_u\,e_a \wedge e_b = 0$ pointwise. Hence the Stratonovich correction vanishes, yielding the pure It\^o form
\begin{equation}
\mathbb{A}_{t_i, t_{i+1}} = \underbrace{\int_{t_i}^{t_{i+1}}(X_u - X_{t_i}) \wedge d\mathsf{N}_u}_{=:\,\mathcal{M}_i} + \underbrace{\int_{t_i}^{t_{i+1}}(X_u - X_{t_i}) \wedge dA_u}_{=:\,\mathcal{D}_i}.
\end{equation}

\subparagraph{Martingale contribution.} For $a \ne b$, the parallelogram bound $|z_1 - z_2|^2 \le 2(|z_1|^2 + |z_2|^2)$ followed by the scalar It\^o isometry and $d\langle \mathsf{N}^b\rangle_u \le c_\mathsf{N}^2\,du$ gives
\begin{equation}
\mathbb{E}[|\mathcal{M}_i^{ab}|^2] \le \tfrac{1}{2}\,c_\mathsf{N}^2\,\mathbb{E}\!\left[\int_{t_i}^{t_{i+1}}\bigl((X^a_u - X^a_{t_i})^2 + (X^b_u - X^b_{t_i})^2\bigr)du\right].
\end{equation}
The coordinate increment bound $\mathbb{E}[(X^j_u - X^j_{t_i})^2] \le C_R\,(u - t_i)$ yields $\mathbb{E}[\int_{t_i}^{t_{i+1}}(X^j_u - X^j_{t_i})^2\,du] \le C_R\,\delta^2/2$, and summing over the $d(d-1)$ off-diagonal components $(a, b)$:
\begin{equation}
\label{eq:Mi_rate}
\mathbb{E}[\|\mathcal{M}_i\|_{\mathrm{HS}}^2] \le \sum_{a \ne b}\mathbb{E}[|\mathcal{M}_i^{ab}|^2] \le \tfrac{1}{2}\,d(d-1)\,c_\mathsf{N}^2\,C_R\,\delta^2 =: C_\mathcal{M}\,\delta^2.
\end{equation}

\subparagraph{Drift contribution (Gaussian $L^2$ moments).} For $a \ne b$, Cauchy--Schwarz in $du$ and coordinate independence give
\begin{equation}
\mathbb{E}[|\mathcal{D}_i^{ab}|^2] \le \tfrac{\delta}{2}\int_{t_i}^{t_{i+1}}\!\!\mathbb{E}\!\left[(X^a_u - X^a_{t_i})^2\right]\mathbb{E}\!\left[(dA^b_u/du)^2\right]du + (\text{sym}),
\end{equation}
where by independence $\mathbb{E}[(X^a_u - X^a_{t_i})^2 (dA^b_u/du)^2] = \mathbb{E}[(X^a_u - X^a_{t_i})^2]\,\mathbb{E}[(dA^b_u/du)^2]$ for $a \ne b$. Using $\mathbb{E}[(X^a_u - X^a_{t_i})^2] \le C_R\,(u - t_i)$ and the assumption $\mathbb{E}[(dA^b/du)^2] \le c_A^2$:
\begin{equation}
\label{eq:Di_rate}
\mathbb{E}[|\mathcal{D}_i^{ab}|^2] \le \delta\,c_A^2\,C_R\,\delta^2/2, \qquad \mathbb{E}[\|\mathcal{D}_i\|_{\mathrm{HS}}^2] \le \tfrac{1}{2}\,d(d-1)\,c_A^2\,C_R\,\delta^3 =: C_\mathcal{D}\,\delta^3.
\end{equation}

\subparagraph{Summation.} By the martingale-increment property, the $\{\mathcal{M}_i\}$ are pairwise $L^2$-orthogonal: the integrand of $\mathcal{M}_j$ vanishes at $u = t_j$ and is adapted from $t_j$ onward, so $\mathbb{E}[\mathcal{M}_j \mid \mathcal{F}_{t_j}] = 0$, while $\mathcal{M}_i$ for $i < j$ is $\mathcal{F}_{t_{i+1}} \subseteq \mathcal{F}_{t_j}$-measurable.
\begin{equation}
\mathbb{E}\bigl[\|\textstyle\sum_i \mathcal{M}_i\|^2\bigr] = \sum_i \mathbb{E}[\|\mathcal{M}_i\|^2] \le (n-1)\,C_\mathcal{M}\,\delta^2 \le C_\mathcal{M}\,\Delta\,\delta.
\end{equation}
The $\{\mathcal{D}_i\}$ need not be orthogonal; Minkowski in $L^2$ gives $\|\sum_i \mathcal{D}_i\|_{L^2(\mathrm{HS})} \le (n-1)\sqrt{C_\mathcal{D}\,\delta^3}$, so $\mathbb{E}[\|\sum_i \mathcal{D}_i\|^2] \le (n-1)^2\,C_\mathcal{D}\,\delta^3 \le (\Delta/\delta)^2\,C_\mathcal{D}\,\delta^3 = C_\mathcal{D}\,\Delta^2\,\delta$. Combining via $\|\cdot\|^2 \le 2\|\mathcal{M}\|^2 + 2\|\mathcal{D}\|^2$:
\begin{equation}
\mathbb{E}[\|\Delta S^{(2)}_{0, \Delta}\|_{\mathrm{HS}}^2] \le 2\,C_\mathcal{M}\,\Delta\,\delta + 2\,C_\mathcal{D}\,\Delta^2\,\delta = C_2(\Delta)\,\delta,
\end{equation}
so $\gamma = 1$ at level~2.

\paragraph{Levels $m \ge 3$: Friz--Riedel Theorem~5}

For signature levels beyond 2, we apply~\citet[Thm.~5, p.~188]{FrizRiedel2014}. For a centered Gaussian process with independent coordinates whose covariance has finite 2D $\rho$-variation ($\rho \in [1, 2)$), for any rough-path roughness parameter $p \in (2\rho, 4)$, any auxiliary H\"older-type parameter $\gamma_{\mathrm{FR}} > \rho$ satisfying $1/\gamma_{\mathrm{FR}} + 1/\rho > 1$, any $q > 2\gamma_{\mathrm{FR}}$, any truncation level $N \in \mathbb{N}$, and any $r \ge 1$,
\begin{equation}
\label{eq:FR_Thm5}
\bigl\|\varrho_{q\text{-var};[0, \Delta]}(S_N(\mathbf{X}^{(n)}),\,S_N(\mathbf{X}))\bigr\|_{L^r} \le C(q, \rho, \gamma_{\mathrm{FR}}, K, N)\,r^{N/2}\,\sup_{0 \le t \le \Delta}\|X^{(n)}_t - X_t\|_{L^2}^{1 - \rho/\gamma_{\mathrm{FR}}},
\end{equation}
where $K$ denotes, as in~\citet[Thm.~5]{FrizRiedel2014}, any upper bound on the 2D $\rho$-variation of the covariance of $X$ over the relevant square. The factor $r^{N/2}$ arises from the $L^2$ moment bounds for Wiener chaos projections of the Gaussian rough-path lift (via~\citealp[Prop.~15.24]{FrizVictoir2010} in the proof of~\citealp[Thm.~5]{FrizRiedel2014}, p.~189, which invokes Nelson's hypercontractive equivalence of $L^r$ and $L^2$ norms on Wiener chaos); it is \emph{not} the mesh index, and fixing $r = 2$ gives $r^{N/2} = 2^{N/2}$, constant in $n$. In our setting $\rho = 1$.

\subparagraph{Sup-$L^2$ bound on the PL error.} For $t \in [t_i, t_{i+1}]$ with $\lambda := (t - t_i)/\delta \in [0, 1]$, the linear interpolation satisfies $X^{(n)}_t - X_t = \lambda(X_{t_{i+1}} - X_t) + (1 - \lambda)(X_{t_i} - X_t)$. Triangle inequality in $L^2$ and the Lipschitz increment bound $\|X^j_s - X^j_t\|_{L^2}^2 \le C_R\,|s - t|$ give
\begin{equation}
\label{eq:sup_L2}
\sup_{0 \le t \le \Delta}\bigl\|X^{(n)}_t - X_t\bigr\|_{L^2}^2 \le 4\,d\,C_R\,\delta.
\end{equation}

\subparagraph{From the rough-path distance to the Hilbert--Schmidt error.} We apply~\eqref{eq:FR_Thm5} at truncation level $N = M$ with $r = 2$. The piecewise-linear family over uniform partitions satisfies the approximation conditions of~\citet[Sec.~6.2]{FrizRiedel2014}; the statement on $[0,1]$ transfers to $[0,\Delta]$ by the deterministic time change $t \mapsto t/\Delta$, under which signatures are invariant (\citealp[Prop.~7.10]{FrizVictoir2010} for the bounded-variation approximants; the invariance passes to the Gaussian lift in the limit by naturality,~\citealp[Thm.~15.33(iv)]{FrizVictoir2010}) and only the constant changes. The inhomogeneous distance $\varrho_{q\text{-var}}$ is the maximum over levels $k = 1, \ldots, M$ of the level-$k$ $q/k$-variation distances~\citep[p.~159]{FrizRiedel2014}, so taking the trivial partition $\{0, \Delta\}$ at level $m$,
\[
\|S^{(m)}(\mathbf{X})_{0,\Delta} - S^{(m)}(\mathbf{X}^{(n)})_{0,\Delta}\|_{\mathrm{HS}} \;\le\; C_m\, \varrho_{q\text{-var};[0,\Delta]}\bigl(S_M(\mathbf{X}^{(n)}), S_M(\mathbf{X})\bigr), \qquad 1 \le m \le M,
\]
where $C_m$ accounts for the equivalence of norms on the finite-dimensional space $(\mathbb{R}^d)^{\otimes m}$. No localization or moment control of $\mathbf{X}^{(n)}$ is needed: the Friz--Riedel bound is applied directly at the truncation level of interest, and the level-$m$ endpoint differences are read off from the distance itself.

Fix $\varepsilon' \in (0,1)$ and choose the FR parameter $\gamma_{\mathrm{FR}} = 2/\varepsilon'$, so that $1 - \rho/\gamma_{\mathrm{FR}} = 1 - \varepsilon'/2$ (using $\rho = 1$; the admissibility conditions $\gamma_{\mathrm{FR}} > \rho$ and $1/\gamma_{\mathrm{FR}} + 1/\rho > 1$ are automatic), and fix any $q > 2\gamma_{\mathrm{FR}}$. Summing the level bounds, squaring, and combining~\eqref{eq:FR_Thm5} with the sup-$L^2$ bound~\eqref{eq:sup_L2}:
\begin{equation}
\label{eq:final_rate}
\begin{aligned}
\mathbb{E}\bigl[\|S^{(M)}(\mathbf{X})_{0, \Delta} - S^{(M)}(\mathbf{X}^{(n)})_{0, \Delta}\|_{\mathrm{HS}}^2\bigr]
&\le C_M\,\bigl\|\varrho_{q\text{-var}}\bigr\|_{L^2}^2
\le C_M(\Delta,\varepsilon')\,\bigl(4dC_R\,\delta\bigr)^{1-\varepsilon'/2}\\
&\le C_M'(\Delta, \varepsilon')\,\delta^{1 - \varepsilon'},
\end{aligned}
\end{equation}
which verifies Assumption~\ref{ass:gamma} with $\gamma = 1 - \varepsilon'$ for any $\varepsilon' > 0$.

\paragraph{Summary}
Assumption~\ref{ass:gamma} holds for centered continuous Gaussian semimartingales with independent coordinates, with diffusion and drift satisfying $d\langle \mathsf{N}^j\rangle/dt \le c_\mathsf{N}^2$ and $\mathbb{E}[(dA^j/du)^2] \le c_A^2$, whose coordinate covariances have finite 2D $1$-variation. This class contains stationary Ornstein--Uhlenbeck and fractional Ornstein--Uhlenbeck at $H = 1/2$. The rates are:
\begin{itemize}
\item $\gamma = 1$ exactly at signature level~2, via the direct argument~\eqref{eq:area_telescope}--\eqref{eq:Di_rate};
\item $\gamma = 1 - \varepsilon$ for any $\varepsilon > 0$ at truncation level $M$, via~\citet[Thm.~5]{FrizRiedel2014} and~\eqref{eq:final_rate}.
\end{itemize}
The level-2 rate $\gamma = 1$ is exact; the $\varepsilon$-loss enters only at truncation level $M \ge 3$, where the rigorous rate is $\gamma = 1 - \varepsilon$ for any $\varepsilon > 0$. In Corollary~\ref{cor:sample_complexity}, the heuristic exponent $\gamma = 1$ is used for sample-complexity formulas; the rigorous form carries an arbitrarily small $\varepsilon$-loss, as does the rough fOU regime $H \in (1/4,1/2)$, whereas the Young regime $H > 1/2$ is exact (Proposition~\ref{prop:gamma_values}(ii)).

%%%%%%%%%%%%%%%%%%%%%%%%%%%%%%%%%%%%%%%%%%%%%%%%%%%%%%%%%%%%%%%%%%%%%%%%%%%%%%%
\subsection{Fractional Ornstein--Uhlenbeck}
\label{app:fOU_rate}
%%%%%%%%%%%%%%%%%%%%%%%%%%%%%%%%%%%%%%%%%%%%%%%%%%%%%%%%%%%%%%%%%%%%%%%%%%%%%%%

For the fractional Ornstein--Uhlenbeck (fOU) process
\begin{equation}
dX_t = -\theta X_t\,dt + \sigma\,dB^H_t, \quad \theta > 0,\ H \in (1/4, 1),
\end{equation}
we establish the discretization rate $\gamma = 4H-1-\varepsilon$ for any $\varepsilon > 0$ at truncation level $M$ when $H \in (1/4, 1/2)$, with $\gamma = 4H-1$ attained exactly at signature level~2 via Lemma~\ref{lem:fOU_area}. For $H \in (1/2, 1)$, we establish $\gamma = 2H$ exactly, with no $\varepsilon$-loss at any level. The proof treats the rough ($H \in (1/4, 1/2)$) and Young ($H > 1/2$) regimes separately. The boundary $H = 1/2$ is handled by the Gaussian-semimartingale analysis of Appendix~\ref{app:gamma_semi} (giving $\gamma = 1-\varepsilon$ at truncation level $M$), consistent with the boundary values $4H-1 = 2H = 1$. The constraint $H > 1/4$ is intrinsic: below this threshold, the second-level rough paths of dyadic approximations fail to converge (in $L^1$ with respect to the $p$-variation distance) to a geometric rough path~\citep[Thm.~2]{CoutinQian2002}.

\begin{remark}[Trapezoidal vs.\ Euler schemes]
\label{rem:trap_vs_euler}
The piecewise-linear interpolation $X^{(n)}$ computes the \emph{trapezoidal scheme} of~\citet{NeuenkirchTindelUnterberger2010}, which for fractional Brownian motion converges at rate $n^{-(4H-1)}$ for all $H > 1/4$~\citep[Thm.~2]{NeuenkirchTindelUnterberger2010}, whereas the standard Euler scheme saturates at rate $n^{-2}$ for $H > 3/4$~\citep[Thm.~1]{NeuenkirchTindelUnterberger2010}. The piecewise-linear signature thus avoids this barrier; for fOU specifically (with $H \neq 1/2$; the boundary case is treated in Appendix~\ref{app:gamma_semi}), the SDE-decomposition argument of Lemma~\ref{lem:fOU_area} below transfers the fBm rate, yielding $\delta^{\min(4H-1,\,2)}$ at level~2.
\end{remark}

The proof proceeds by establishing the rate for fBm via~\citet{NeuenkirchTindelUnterberger2010}, then extending to fOU via explicit $L^2$ Gaussian estimates.

\subsubsection{Covariance Structure of fOU}

We first establish that fOU inherits the local covariance regularity of its fBm driver.

\begin{lemma}[fOU covariance regularity]
\label{lem:fOU_cov}
Let $X$ be the stationary fOU process with mean-reversion $\theta > 0$ and Hurst parameter $H \in (0,1)$. The stationary covariance $R(\tau) := \Cov(X_t, X_{t+\tau})$ satisfies:
\begin{equation}
\label{eq:fOU_cov_local}
R(0) - R(\tau) = \kappa_{H,\theta,\sigma}\, |\tau|^{2H} + O(|\tau|^{2}) \quad \text{as } \tau \to 0,
\end{equation}
where $\kappa_{H,\theta,\sigma} = \sigma^2/2 > 0$. In particular, $R$ is $C^1$ on $(0,\Delta]$ with
\begin{equation}
\label{eq:R_prime_local}
|R'(\tau)| \;\le\; C(H,\theta,\sigma,\Delta)\,\tau^{2H-1}, \qquad \tau \in (0,\Delta].
\end{equation}
Moreover, the 2D covariance kernel $K_X(s,t) := R(|s-t|)$ has finite 2D $\rho$-variation on $[0,\Delta]^2$:
\begin{equation}
\label{eq:fOU_2D_rho_var}
V_\rho(K_X;\,[0,\Delta]^2) < \infty \qquad \text{for every } \rho > \rho_*(H) := \max\bigl(1,\ 1/(2H)\bigr).
\end{equation}
In particular $\rho_*(H) < 2$ for all $H \in (1/4, 1)$, so $K_X$ satisfies the hypothesis of~\citet[Thm.~15.33]{FrizVictoir2010} and the canonical geometric rough-path lift of $X$ exists with finite moments of all orders.
\end{lemma}

\begin{proof}
Via the spectral representation~\eqref{eq:fOU_spectral} and the substitution $\mu = \lambda\tau$:
\begin{equation}
R(0) - R(\tau) = \frac{\sigma^2 \Gamma(2H+1) \sin(\pi H)}{\pi} \cdot |\tau|^{2H} \underbrace{\int_0^\infty \frac{(1 - \cos\mu)\, \mu^{1-2H}}{\theta^2\tau^2 + \mu^2}\, d\mu}_{=:\, I(\tau)}.
\end{equation}

\textbf{Leading term.} The function $(1-\cos\mu)\mu^{-1-2H}$ is integrable for $H \in (0,1)$, so dominated convergence gives $I(\tau) \to I(0) = \int_0^\infty (1-\cos\mu)\mu^{-1-2H}\, d\mu = \pi/(2\Gamma(2H+1)\sin(\pi H))$ as $\tau \to 0$, yielding $\kappa_{H,\theta,\sigma} = \sigma^2/2$.

\textbf{Remainder.} Write
\begin{equation}
I(0) - I(\tau) = \int_0^\infty (1-\cos\mu)\,\mu^{-1-2H} \cdot \frac{\theta^2\tau^2}{\theta^2\tau^2+\mu^2}\, d\mu.
\end{equation}
Setting $h(x) := 2(1-\cos x)/x^2$ (so $|h| \le 1$, $h(0) = 1$) and substituting $v = \mu/(\theta\tau)$:
\begin{equation}
I(0) - I(\tau) = \tfrac{1}{2}(\theta\tau)^{2-2H} \int_0^\infty \frac{v^{1-2H}\, h(\theta\tau v)}{1+v^2}\, dv.
\end{equation}
Since $v^{1-2H}/(1+v^2)$ is integrable for $H \in (0,1)$, dominated convergence yields $I(0) - I(\tau) = O(\tau^{2-2H})$, hence $R(0) - R(\tau) - \kappa\tau^{2H} = O(\tau^{2H} \cdot \tau^{2-2H}) = O(\tau^2)$.

\textbf{2D $\rho$-variation of $K_X$.} Decompose $R(\tau) = R(0) - \kappa|\tau|^{2H} + \psi(\tau)$, where $\psi(\tau) := R(\tau) - R(0) + \kappa|\tau|^{2H}$. Writing $\kappa_0 := \sigma^2\Gamma(2H+1)\sin(\pi H)/\pi$ (so that $\kappa_0 I(0) = \kappa$) and using the spectral identity above for $I(0) - I(\tau)$,
\[
\psi(\tau) = \kappa_0|\tau|^{2H}\bigl(I(0) - I(\tau)\bigr) = \tfrac{\kappa_0\,\theta^{2-2H}}{2}\,\tau^2\,J(\tau), \qquad J(\tau) := \int_0^\infty \frac{v^{1-2H}\,h(\theta\tau v)}{1+v^2}\, dv.
\]
The function $h(x) = 2(1-\cos x)/x^2$ is $C^\infty(\mathbb{R})$ with $h(0) = 1$, $h'(0) = 0$ (since $h$ is even), and $|h^{(k)}(x)| \le C_k/(1+x^2)$ for $k = 0, 1, 2$ (from the Taylor expansion at $0$ and the $O(1/x^2)$ decay of each derivative at infinity). Dominated convergence with the $\tau$-independent envelope $v^{1-2H}/(1+v^2)$ (integrable for $H \in (0,1)$) yields $J$ continuous on $[0,\Delta]$ with $J(0) = \int_0^\infty v^{1-2H}/(1+v^2)\,dv > 0$. For $\tau > 0$, differentiating under the integral sign and splitting at $v_\tau := 1/(\theta\tau)$ gives the pointwise estimates
\[
|J'(\tau)| \le C_1 \tau^{2H-1}, \qquad |J''(\tau)| \le C_2 \tau^{2H-2}, \qquad \tau \in (0,\Delta]
\]
(on $\{v \le v_\tau\}$ use the near-origin bounds $|h'(\theta\tau v)| \le C|\theta\tau v|$ and $|h''(\theta\tau v)| \le C$; on $\{v > v_\tau\}$ use $|h^{(k)}(\theta\tau v)| \le C_k/(\theta\tau v)^2$; the dominant contribution in each case comes from $v \sim v_\tau$, yielding the stated $\tau^{2H-1}$ and $\tau^{2H-2}$ rates). The $\tau^2$ prefactor in $\psi = (\kappa_0\theta^{2-2H}/2)\,\tau^2\,J$ then suppresses these boundary singularities: for $\tau \in (0,\Delta]$,
\[
\psi'(\tau) = \tfrac{\kappa_0\theta^{2-2H}}{2}\bigl[2\tau J(\tau) + \tau^2 J'(\tau)\bigr] = O(\tau) + O(\tau^{2H+1}),
\]
\[
\psi''(\tau) = \tfrac{\kappa_0\theta^{2-2H}}{2}\bigl[2J(\tau) + 4\tau J'(\tau) + \tau^2 J''(\tau)\bigr] = \kappa_0\theta^{2-2H}\,J(\tau) + O(\tau^{2H}),
\]
so $\psi'(\tau) \to 0$ and $\psi''(\tau) \to c_2 := \kappa_0\theta^{2-2H}\,J(0) > 0$ as $\tau \to 0^+$. Extending by $\psi'(0) := 0$ and $\psi''(0) := c_2$ yields $\psi \in C^2([0,\Delta])$ with $\psi(0) = \psi'(0) = 0$. In particular, $R'(\tau) = -2H\kappa\,\tau^{2H-1} + \psi'(\tau)$ on $(0,\Delta]$ with $\psi'(\tau) = O(\tau)$, and $\tau \le \Delta^{2-2H}\,\tau^{2H-1}$ there, which proves~\eqref{eq:R_prime_local}. The constant $R(0)$ contributes zero to any 2D rectangle increment, so rectangle increments of $K_X$ on $[0,\Delta]^2$ decompose as
\begin{equation*}
K_X([s,s']\times[t,t']) = -\kappa\bigl[K_{\mathrm{fBm}}([s,s']\times[t,t'])\bigr] + \tilde\psi([s,s']\times[t,t']),
\end{equation*}
where $K_{\mathrm{fBm}}(s,t) := |s-t|^{2H}$ and $\tilde\psi(s,t) := \psi(|s-t|)$. The first term is, up to the constant factor $-2$, the rectangle increment of the fractional-Brownian covariance $R_{\mathrm{fBm}}(s,t) = \tfrac{1}{2}(s^{2H} + t^{2H} - |s-t|^{2H})$ (rectangle increments annihilate the single-variable terms); its 2D $\rho$-variation on $[0,\Delta]^2$ is finite for every $\rho > 1/(2H)$ (for $H \in (0, 1/2]$ by~\citealp[Prop.~15.5]{FrizVictoir2010}; for $H > 1/2$, the function $(s,t) \mapsto |s-t|^{2H}$ is $C^1$ with locally integrable mixed partial $\partial^2_{s,t}|s-t|^{2H} = -2H(2H-1)|s-t|^{2H-2}$ on $(0,\Delta)$, yielding finite 2D $1$-variation). For $\tilde\psi \in C^2([0,\Delta]^2)$, the 2D $1$-variation is bounded by $\Delta^2\sup|\partial^2_{s,t}\tilde\psi| < \infty$, hence finite 2D $\rho$-variation for every $\rho \ge 1$. Combining both contributions via the triangle inequality for $\rho$-variation semi-norms yields~\eqref{eq:fOU_2D_rho_var} for every $\rho > \rho_*(H) = \max(1, 1/(2H))$.
\end{proof}

\subsubsection{L\'evy Area Discretization}

\begin{lemma}[fOU L\'evy area discretization]
\label{lem:fOU_area}
Let $(X^1, X^2)$ be two independent stationary fOU processes with Hurst parameter $H \in (1/4, 1) \setminus \{1/2\}$, and let $\mathbf{X}$ denote the canonical geometric rough-path lift of $X = (X^1, X^2)$. Let $\mathbb{A}_{0,\Delta}$ denote the L\'evy area (the antisymmetric part of the level-2 component of $\mathbf{X}$) and $\mathbb{A}^{(n)}_{0,\Delta}$ the L\'evy area of the piecewise-linear interpolation through $n$ equally-spaced observations. Then:
\begin{equation}
\label{eq:fOU_area_rate}
\mathbb{E}\left[|\mathbb{A}_{0,\Delta} - \mathbb{A}^{(n)}_{0,\Delta}|^2\right] \le C(H,\theta,\sigma,\Delta) \cdot \delta^{\min(4H-1,\,2)}.
\end{equation}
(The boundary case $H = 1/2$ is excluded only because Appendix~\ref{app:gamma_semi} treats it directly: at $H = 1/2$ the argument below goes through with the modification that the off-diagonal sums in Terms~2--3 become $\sum_{2 \le k \le n} k^{2H-2} = \sum_{2 \le k \le n} k^{-1} = O(\log n)$, yielding $O(\delta^{2}\log(1/\delta))$ for those terms --- still $O(\delta^{\min(4H-1,\,2)}) = O(\delta)$ there.)
\end{lemma}

\begin{proof}
The signature of the piecewise-linear interpolation corresponds to the \emph{trapezoidal approximation} of the L\'evy area. We establish the rate by appealing to~\citet{NeuenkirchTindelUnterberger2010} for the fBm component and controlling the drift contribution via the SDE structure.

\paragraph{Step 1: Rate for fBm}
For pure fractional Brownian motion $B^H$, \citet[Thm.~2]{NeuenkirchTindelUnterberger2010} establish that the trapezoidal scheme
\begin{equation}
\hat{X}^n_T = \frac{1}{2}\sum_{i=0}^{n-1}\bigl(B^{(1)}_{iT/n} + B^{(1)}_{(i+1)T/n}\bigr)\bigl(B^{(2)}_{(i+1)T/n} - B^{(2)}_{iT/n}\bigr)
\end{equation}
converges with mean-squared error
\begin{equation}
\mathbb{E}|X_T - \hat{X}^n_T|^2 \le C(H) \cdot T^{4H} \cdot n^{-(4H-1)},
\end{equation}
valid for all $H > 1/4$ (where $n$ denotes the number of subintervals in the Neuenkirch--Tindel--Unterberger convention, corresponding to $n-1$ in the notation of the present paper). The rate relies on an algebraic cancellation in the interpolation error covariance that regularizes the singular kernel $|s-r|^{2H-2}$ to the integrable $|s-r|^{4H-2}$~\citep[Secs.~3.1 and~3.3]{NeuenkirchTindelUnterberger2010}; an analogous cancellation for dyadic approximations appears in~\citet[Lem.~12]{CoutinQian2002}.

\paragraph{Step 2: Transfer to fOU via SDE decomposition}
Integrating the SDE $dX^i_t = -\theta_i(X^i_t - \mu_i)\,dt + \sigma_i\,dB^{H,i}_t$ from the initial time, the centered increment process $\tilde{X}^i_t := X^i_t - X^i_0$ decomposes as
\begin{equation}
\label{eq:fOU_decomp}
\tilde{X}^i_t = \sigma_i \tilde{B}^{H,i}_t + D^i_t, \qquad D^i_t := -\theta_i\!\int_0^t \tilde{X}^i_s\,ds - \theta_i(X^i_0 - \mu_i)\,t,
\end{equation}
where $\tilde{B}^{H,i}_t := B^{H,i}_t - B^{H,i}_0$ is a centered fBm increment and $D^i$ is $C^1$ with derivative $(D^i)'(t) = -\theta_i\tilde{X}^i_t - \theta_i(X^i_0 - \mu_i)$, which is continuous and has bounded moments on $[0,\Delta]$ by stationarity and Fernique's theorem~\citep{Fernique1975}.

Since the piecewise-linear interpolation is linear in the path values, $\tilde{X}^{i,(n)} = \sigma_i \tilde{B}^{H,i,(n)} + D^{i,(n)}$, where $\tilde{B}^{H,i,(n)}$ and $D^{i,(n)}$ are the respective piecewise-linear interpolations. The cross-integral $\int_0^\Delta \tilde{X}^1\,d\tilde{X}^2$ and its trapezoidal approximation both decompose by bilinearity into four terms. Writing $E_{AB}$ for the trapezoidal error of $\int A\,dB$ (i.e., $E_{AB} := \int A\,dB - \int A^{(n)}\,dB^{(n)}$), the total error for the cross-integral is
\begin{equation}
\label{eq:area_decomp}
E_{12} = \sigma_1\sigma_2\, E_{\tilde{B}^1\tilde{B}^2} + \sigma_1\, E_{\tilde{B}^1 D^2} + \sigma_2\, E_{D^1 \tilde{B}^2} + E_{D^1 D^2},
\end{equation}
and the L\'evy area error satisfies $|\mathbb{A} - \mathbb{A}^{(n)}|^2 \le 2|E_{12}|^2 + 2|E_{21}|^2$. We bound each term in~\eqref{eq:area_decomp} separately.

\medskip\noindent\emph{Term 1: $E_{\tilde{B}^1\tilde{B}^2}$ (fBm area).}
Since $\tilde{B}^{H,1}$, $\tilde{B}^{H,2}$ are independent fBm increments, \citet[Thm.~2]{NeuenkirchTindelUnterberger2010} gives $\mathbb{E}[|E_{\tilde{B}^1\tilde{B}^2}|^2] \le C(H)\,\Delta^{4H}\,(n-1)^{-(4H-1)} = O(\delta^{4H-1})$.

\medskip\noindent\emph{Term 2: $E_{\tilde{B}^1 D^2}$ (fBm integrand, drift integrator).}
Since $D^2$ is $C^1$, this reduces to a Lebesgue integral: $\int \tilde{B}^{H,1}\,dD^2 = \int \tilde{B}^{H,1}_s\,(D^2)'(s)\,ds$, and similarly for its trapezoidal approximation. The error on each subinterval $[t_i, t_{i+1}]$ is $\int_{t_i}^{t_{i+1}}(\tilde{B}^{H,1}_s - \bar{B}^{H,1}_i)\,(D^2)'(s)\,ds$, where $\bar{B}^{H,1}_i := \tfrac{1}{2}(\tilde{B}^{H,1}_{t_i} + \tilde{B}^{H,1}_{t_{i+1}})$ is the trapezoidal midpoint value. By coordinate independence ($\tilde{B}^{H,1} \perp D^2$):
\begin{multline}
\label{eq:EBD_cov}
\mathbb{E}[|E_{\tilde{B}^1 D^2}|^2] = \sum_{i,j=0}^{n-2} \iint_{[t_i,t_{i+1}]\times[t_j,t_{j+1}]}
\underbrace{\mathbb{E}\bigl[(\tilde{B}^{H,1}_s - \bar{B}^{H,1}_i)(\tilde{B}^{H,1}_u - \bar{B}^{H,1}_j)\bigr]}_{=:\,\phi_{ij}(s,u)}\\
\times\; \underbrace{\mathbb{E}\bigl[(D^2)'(s)\,(D^2)'(u)\bigr]}_{=:\,\psi(s,u)}\,ds\,du.
\end{multline}
Since $(D^2)'$ has bounded second moments on $[0,\Delta]$, $|\psi(s,u)| \le C$. For the fBm trapezoidal interpolation error covariance $\phi_{ij}$: when $i = j$, Cauchy--Schwarz gives $|\phi_{ii}(s,u)| \le C\delta^{2H}$; the same Cauchy--Schwarz bound applies to adjacent intervals $|i - j| = 1$, since $\|\tilde{B}^{H,1}_s - \bar{B}^{H,1}_i\|_{L^2} \le C\delta^H$ uniformly in $i, s$, contributing $O(n) \cdot \delta^2 \cdot C\delta^{2H} = O(\Delta\,\delta^{1+2H})$ to~\eqref{eq:EBD_cov}, which is absorbed into the rate of~\eqref{eq:EBD_bound}. When $|i - j| \ge 2$, the function $|s - u|^{2H}$ is $C^2$ away from the diagonal with $|\partial_s\partial_u |s-u|^{2H}| = O((|i-j|\delta)^{2H-2})$ on $[t_i,t_{i+1}]\times[t_j,t_{j+1}]$, and a Taylor expansion around the segment midpoints gives
\begin{equation}
\label{eq:phi_offdiag}
|\phi_{ij}(s,u)| \le C\,\delta^{2H}\,|i-j|^{2H-2}, \qquad |i-j| \ge 2.
\end{equation}
Summing~\eqref{eq:EBD_cov}: for $H < 1/2$, $2H-2 < -1$ so $\sum_{k \ge 2} k^{2H-2}$ converges and we obtain $O(\delta^{1+2H})$; for $H > 1/2$, $\sum_{k=2}^{n} k^{2H-2} = O(n^{2H-1})$, yielding $O(n^{2H}\,\delta^{2+2H}) = O(\Delta^{2H}\,\delta^2) = O(\delta^2)$. In both cases:
\begin{equation}
\label{eq:EBD_bound}
\mathbb{E}[|E_{\tilde{B}^1 D^2}|^2] = O\bigl(\delta^{\min(1+2H,\,2)}\bigr).
\end{equation}
Since $\min(1+2H,\,2) \ge \min(4H-1,\,2)$ for $H < 1$, this term is at worst comparable to the fBm area rate.

\medskip\noindent\emph{Term 3: $E_{D^1\tilde{B}^2}$ (drift integrand, fBm integrator).}
Writing $g_i(s) := D^1_s - \bar{D}^1_i$ for the per-cell trapezoidal error of the integrand, where $\bar{D}^1_i := \tfrac{1}{2}(D^1_{t_i} + D^1_{t_{i+1}})$, one has $g_i(t_i) = -\tfrac{1}{2}\Delta D^1_i = -\tfrac{1}{2}\int_{t_i}^{t_{i+1}}(D^1)'(r)\,dr$, and Fubini applied to the Young integral over the cell gives the exact representation
\[
\int_{t_i}^{t_{i+1}} g_i(s)\,d\tilde{B}^{H,2}_s \;=\; -\int_{t_i}^{t_{i+1}} (D^1)'(r)\,\bigl(\tilde{B}^{H,2}_r - \bar{B}^{H,2}_i\bigr)\,dr,
\]
i.e.\ the per-cell error integrates $(D^1)'$ against the \emph{same} centered fBm interpolation-error functional $\tilde{B}^{H,2}_r - \bar{B}^{H,2}_i$ that appears in Term~2. By coordinate independence ($D^1 \perp \tilde{B}^{H,2}$),
\begin{equation}
\label{eq:EDB_bound}
\begin{aligned}
\mathbb{E}[|E_{D^1\tilde{B}^2}|^2] &= \sum_{i,j=0}^{n-2}\iint_{[t_i,t_{i+1}]\times[t_j,t_{j+1}]} \mathbb{E}\bigl[(D^1)'(r)\,(D^1)'(r')\bigr]\;\phi_{ij}(r,r')\,dr\,dr'\\
&\;\le\; C\,\delta^2 \sum_{i,j=0}^{n-2} \sup_{r,r'}\bigl|\phi_{ij}(r,r')\bigr|,
\end{aligned}
\end{equation}
where $\phi_{ij}$ is exactly the Term-2 covariance (with $\tilde{B}^{H,2}$ in place of $\tilde{B}^{H,1}$) and $|\mathbb{E}[(D^1)'(r)(D^1)'(r')]| \le C$ by bounded second moments. The bounds of Term~2 apply verbatim: $|\phi_{ij}| \le C\delta^{2H}$ for $|i-j| \le 1$ and $|\phi_{ij}| \le C\,\delta^{2H}|i-j|^{2H-2}$ for $|i-j| \ge 2$~\eqref{eq:phi_offdiag}. Summing as in Term~2: for $H < 1/2$, $\sum_k k^{2H-2} < \infty$ gives $O(n\,\delta^{2+2H}) = O(\delta^{1+2H})$; for $H > 1/2$, $\sum_{i,j}|i-j|^{2H-2} = O(n^{2H})$ gives $O(n^{2H}\delta^{2+2H}) = O(\Delta^{2H}\delta^2) = O(\delta^2)$. In both cases:
\begin{equation}
\mathbb{E}[|E_{D^1\tilde{B}^2}|^2] = O\bigl(\delta^{\min(1+2H,\,2)}\bigr).
\end{equation}
Since $\min(1+2H,\,2) \ge 2$ for $H \ge 1/2$ and $1 + 2H > 4H - 1$ for $H < 1$, this term is $O(\delta^{\min(4H-1,\,2)})$.

\medskip\noindent\emph{Term 4: $E_{D^1 D^2}$ (drift area).}
Both $D^1, D^2$ are $C^1$ with $\|(D^i)'\|_{L^\infty[0,\Delta]}$ having finite second moments by~\eqref{eq:fOU_decomp} and Fernique. Hence $\|D^i - D^{i,(n)}\|_\infty \le \delta\,\|(D^i)'\|_{L^\infty[0,\Delta]}$. Decomposing
\[
E_{D^1 D^2} = \int_0^\Delta (D^1 - D^{1,(n)})\,dD^2 + \int_0^\Delta D^{1,(n)}\,d(D^2 - D^{2,(n)}),
\]
the first term is bounded by $\|D^1 - D^{1,(n)}\|_\infty \int_0^\Delta |(D^2)'(s)|\,ds \le \delta\,\Delta\,\|(D^1)'\|_\infty\,\|(D^2)'\|_\infty$. For the second term, $D^2 - D^{2,(n)}$ vanishes at every grid point, so the integral may be recentered cell by cell,
\[
\int_0^\Delta D^{1,(n)}\,d(D^2 - D^{2,(n)}) \;=\; \sum_{i=0}^{n-2} \int_{t_i}^{t_{i+1}} \bigl(D^{1,(n)}_s - D^{1,(n)}_{t_i}\bigr)\,d(D^2 - D^{2,(n)})_s,
\]
where, with $|D^{1,(n)}_s - D^{1,(n)}_{t_i}| \le \delta\,\|(D^1)'\|_\infty$ (piecewise-linear slopes are cell averages of $(D^1)'$) and
\[
\sum_i \int_{t_i}^{t_{i+1}} |d(D^2 - D^{2,(n)})| \le 2\Delta\,\|(D^2)'\|_\infty,
\]
the second term is bounded by $2\delta\,\Delta\,\|(D^1)'\|_\infty\,\|(D^2)'\|_\infty$. Hence $|E_{D^1 D^2}| \le C\,\delta\,\|(D^1)'\|_\infty\,\|(D^2)'\|_\infty$ pathwise; squaring and using moment finiteness (independence of $D^1$ and $D^2$ together with Fernique) gives $\mathbb{E}[|E_{D^1D^2}|^2] = O(\delta^2)$.

\medskip\noindent\emph{Combined area bound.}
By~\eqref{eq:area_decomp} and the elementary inequality $|a+b+c+d|^2 \le 4(|a|^2 + |b|^2 + |c|^2 + |d|^2)$:
\begin{equation}
\mathbb{E}[|\mathbb{A} - \mathbb{A}^{(n)}|^2] \le C\bigl(\underbrace{O(\delta^{4H-1})}_{\text{fBm area}} + \underbrace{O(\delta^{\min(1+2H,\,2)})}_{\text{Term 2}} + \underbrace{O(\delta^{\min(1+2H,\,2)})}_{\text{Term 3}} + \underbrace{O(\delta^2)}_{\text{Term 4}}\bigr) = O(\delta^{\min(4H-1,\,2)}),
\end{equation}
since $\min(1+2H,\,2) \ge \min(4H-1,\,2)$ and $4H - 1 \le 2$ for $H \le 3/4$ (where the fBm term dominates), while for $H > 3/4$ the drift terms contribute $O(\delta^2)$ matching the stated bound.
\end{proof}

\subsubsection{Full Signature Discretization}

\begin{lemma}[fOU discretization rate]
\label{lem:fOU_rate}
Let $X$ be the stationary fOU process with canonical geometric rough-path lift $\mathbf{X}$, and let $X^{(n)}$ be its piecewise-linear interpolation at mesh $\delta = \Delta/(n-1)$ over the interval $[0, \Delta]$, with rough-path lift $\mathbf{X}^{(n)}$. For all $M \ge 1$ and $H \in (1/4, 1) \setminus \{1/2\}$:
\begin{equation}
\label{eq:fOU_rate_bound}
\mathbb{E}\left[\|S^{(M)}(\mathbf{X})_{0,\Delta} - S^{(M)}(\mathbf{X}^{(n)})_{0,\Delta}\|_{\mathrm{HS}}^2\right] \le
\begin{cases}
C_\varepsilon(M, H, \theta, \sigma, \Delta) \cdot \delta^{4H-1-\varepsilon}, & H \in (1/4, 1/2),\ \varepsilon > 0,\\
C(M, H, \theta, \sigma, \Delta) \cdot \delta^{2H}, & H \in (1/2, 1).
\end{cases}
\end{equation}
For $H \in (1/4, 1/2)$, the level-2 L\'evy area bound from Lemma~\ref{lem:fOU_area} attains the exact rate $\delta^{4H-1}$ without $\varepsilon$-loss; the $\varepsilon$-loss at truncation level $M$ in~\eqref{eq:fOU_rate_bound} enters only from higher levels via the Friz--Riedel Theorem~5 admissibility constraint (see~\eqref{eq:rough_level_m} below). The Young regime $H > 1/2$ carries no $\varepsilon$-loss at any level. At the boundary $H = 1/2$, fOU coincides with stationary Ornstein--Uhlenbeck and Proposition~\ref{prop:gamma_values}(i) (Appendix~\ref{app:gamma_semi}) gives $\gamma = 1-\varepsilon$ at truncation level $M$, consistent with the boundary values $4H-1 = 2H = 1$.
\end{lemma}

\begin{proof}
We establish the bound level by level.

\paragraph{Level 1: Zero discretization error}
The level-1 signature is the path increment: $S^{(1)}(\mathbf{X})_{0,\Delta} = X_\Delta - X_0$. Since piecewise-linear interpolation preserves endpoints, $S^{(1)}(\mathbf{X}^{(n)})_{0,\Delta} = X_\Delta - X_0$ identically.

\paragraph{Level 2: L\'evy area}
Since $X^{(n)}$ preserves endpoints, the symmetric part of level~2 is identical for $\mathbf{X}$ and $\mathbf{X}^{(n)}$, and the level-2 error reduces to the antisymmetric (L\'evy area) part:
\[
\|S^{(2)}(\mathbf{X})_{0,\Delta} - S^{(2)}(\mathbf{X}^{(n)})_{0,\Delta}\|_{\mathrm{HS}}^2 = 2\sum_{1 \le i < j \le d}|\mathbb{A}^{ij}_{0,\Delta} - \mathbb{A}^{ij,(n)}_{0,\Delta}|^2.
\]
Applying Lemma~\ref{lem:fOU_area} to each coordinate pair $(X^i, X^j)$ (independent 2D fOU by the coordinate independence of $\mathbf{X}$ in Proposition~\ref{prop:fOU_assumptions}) and summing absorbs the $\binom{d}{2}$ combinatorial factor into $C$:
\begin{equation}
\mathbb{E}\bigl[\|S^{(2)}(\mathbf{X})_{0,\Delta} - S^{(2)}(\mathbf{X}^{(n)})_{0,\Delta}\|_{\mathrm{HS}}^2\bigr] \le C \cdot \delta^{\min(4H-1,\,2)}.
\end{equation}
Since $\min(4H-1,\,2) \ge \min(4H-1,\,2H)$ for $H < 1$, this is absorbed by the rate $\delta^{\min(4H-1,\,2H)}$ in both regimes below.

\paragraph{Levels $m \ge 3$: Regime-dependent analysis}
We analyze the two regularity regimes separately, as they require fundamentally different analytical techniques.

\medskip
\noindent\textbf{Case 1: Young regime ($H > 1/2$).}
When $H > 1/2$, paths have finite $q$-variation for $q \in (1/H, 2)$, and the signature is defined via Young integration~\citep[Thm.~6.8]{FrizVictoir2010}. Two error mechanisms contribute to the total discretization error, which we analyze separately.

Since piecewise-linear interpolation preserves endpoints, the symmetric part of level~2 has zero discretization error; only the L\'evy area contributes, with $\mathbb{E}[|\mathbb{A}_{0,\Delta} - \mathbb{A}^{(n)}_{0,\Delta}|^2] = O(\delta^{\min(4H-1,\,2)})$ by Lemma~\ref{lem:fOU_area}.

\paragraph{Path-interpolation error in $L^2$}
For iterated integrals at levels $m \ge 3$, path interpolation effects arise from terms where the error $e := X - X^{(n)}$ enters the iterated integrals. For general $H$-H\"older paths, the pathwise Young--Lo\`eve inequality would require controlling the H\"older semi-norm $[e]_{C^\alpha}$, which scales as $\delta^{H-\alpha}$, worse than the sup-norm $\|e\|_\infty \sim \delta^H$, and would yield the suboptimal rate $O(\delta^{4H-2})$. The Gaussian chaos comparison below avoids this entirely: the \emph{only} path-level input it requires is the uniform pointwise $L^2$ interpolation error, which carries no such loss and no logarithmic factor. For $t \in [t_i, t_{i+1}]$ with $\lambda := (t - t_i)/\delta$, the piecewise-linear interpolation satisfies $e_t = \lambda(X_t - X_{t_{i+1}}) + (1-\lambda)(X_t - X_{t_i})$ coordinatewise, so the triangle inequality in $L^2$ and the increment bound $\mathbb{E}[(X_s - X_u)^2] = 2(R(0) - R(|s-u|)) \le C_v\,|s-u|^{2H}$ (Lemma~\ref{lem:fOU_cov}, valid for all $H \in (1/4,1)$) give
\begin{equation}
\label{eq:modulus_bound}
\sup_{0 \le t \le \Delta} \mathbb{E}[e_t^2] \;\le\; C(H,\theta,\sigma,\Delta)\,\delta^{2H}.
\end{equation}
No modulus-of-continuity or sup-norm control of $e$ is needed: only the supremum over $t$ of the pointwise $L^2$ error enters the argument.

\paragraph{Extension to all levels $m \ge 3$: Cass--Ferrucci chaos comparison}
For $H > 1/2$, the signature is defined via Young integration: since $q < 2$, the Lyons extension theorem~\citep[Thm.~9.5]{FrizVictoir2010} determines $S^{(M)}(\mathbf{X})$ uniquely from the $\mathbb{R}^d$-valued path~$X$ without area enhancement. We work in the isonormal space of the centered increment process, with Cameron--Martin space $\mathcal{H}$ and chaos spaces $\mathcal{H}_j$ as defined in Appendix~\ref{subsec:isonormal_setup}. Let $h_t \in \mathcal{H}$ represent $X_t-X_0$, let $h_t^{(n)}$ represent $X^{(n)}_t-X_0$, and put $h_t^e := h_t-h_t^{(n)}$, so that $e_t=X_t-X_t^{(n)}=I_1^\mathcal{H}(h_t^e)$. Then
\begin{equation}
\label{eq:error_gate}
e_n^2 := \sup_{0\le t\le\Delta}\|h_t^e\|_{\mathcal{H}}^2
= \sup_{0\le t\le\Delta}\mathbb{E}[e_t^2]
\le C\,\delta^{2H}
\end{equation}
by~\eqref{eq:modulus_bound}. The Cameron--Martin elements admit uniform-in-$n$ control: $\|h_t\|_{\mathcal{H}}^2 = \Var(X_t - X_0) = 2(R(0) - R(t)) \le 4R(0)$ by stationarity and $|R(t)| \le R(0)$ (Cauchy--Schwarz), and $h_t^{(n)} = \sum_i \phi_i(t)\,h_{t_i}$ is a convex combination of grid vectors with tent-function weights $\phi_i(t) \ge 0$ summing to one, so $\sup_t \|h_t^{(n)}\|_{\mathcal{H}} \le \max_i \|h_{t_i}\|_{\mathcal{H}} \le 2\sqrt{R(0)}$ uniformly in $n$. Hence $\sup_t\|h_t\|_{\mathcal{H}} + \sup_t\|h_t^{(n)}\|_{\mathcal{H}} \le C(H,\theta,\sigma,\Delta)$.

Fix a word $w=(i_1,\ldots,i_m)$. Rather than expand the difference into Young integrals containing clusters of $de$ differentials, we directly compare the finite Wiener chaos kernels of the two Gaussian paths (matching the $\mathcal{H}^{\otimes j}$-valued partial-contraction kernel notation of Lemma~\ref{lem:gaussian_inputs}(b) in Section~\ref{sec:verification}). Throughout this proof we abbreviate $f_{w,j}^X := f_{0,w,j}^{(m)}[X]$ for the kernel of Lemma~\ref{lem:gaussian_inputs}(b) applied to the limiting path $X$ (block index $k=0$, length-$m$ word $w$, chaos order $j$), and $f_{w,j}^{X^{(n)}} := f_{0,w,j}^{(m)}[X^{(n)}]$ for the corresponding kernel of the piecewise-linear approximant. The key structural input is the following error-gate decomposition.

\begin{lemma}[Error-gate decomposition]
\label{lem:error_gate}
For each word $w$ of length $m$ and admissible chaos level $j$, the kernel difference $g_{w,j}^{(n)} := f_{w,j}^X - f_{w,j}^{X^{(n)}}$ admits a finite decomposition $g_{w,j}^{(n)} = \sum_l G_l$ in which each summand $G_l$ contains at least one factor of one of two types:
\begin{enumerate}
\item[\emph{(i)}] a free factor $h_t^e \in \mathcal{H}$ for some $t \in [0,\Delta]$, satisfying $\|h_t^e\|_\mathcal{H} \le e_n$;
\item[\emph{(ii)}] a paired-covariance gate $\langle h_s^e, h_t \rangle_\mathcal{H}$ or $\langle h_s^{(n)}, h_t^e \rangle_\mathcal{H}$, bounded in absolute value by $C\,e_n$ via Cauchy--Schwarz and the uniform-in-$n$ bound $\sup_t(\|h_t\|_\mathcal{H} + \|h_t^{(n)}\|_\mathcal{H}) \le C(H,\theta,\sigma,\Delta)$.
\end{enumerate}
The remaining factors in each $G_l$ are bounded uniformly in $n$ by Lemma~\ref{lem:fOU_cov} (for covariance pairings involving $X$) and the affine-cell variation control (for covariance pairings involving $X^{(n)}$).
\end{lemma}

We prove the lemma at the level of the smooth kernels in~\citet[Def.~3.3]{CassFerrucci2024} and then pass to the limit. Fix a pairing pattern $P$ with $p$ paired indices and free set $\bar P$, so $j=|\bar P|=m-2p$. One summand of the smooth Cass--Ferrucci kernel is an integral over the ordered simplex of a finite product
\[
\prod_{a=1}^{r(P)} A_a,\qquad
A_a\in
\left\{
\left\langle \dot h^{\ell,i_r}_{u_r},\dot h^{\ell,i_s}_{u_s}\right\rangle_{\mathcal{H}}\,du_rdu_s\;(\{r,s\}\in P),\quad
\eta^{\ell,i_k}_{k}(u_k;v_k)\,du_k\;(k\in\bar P)
\right\}.
\]
Here $\eta^{\ell,i_k}_{k}$ is the free-variable Cameron--Martin factor written in~\citet[Def.~3.3]{CassFerrucci2024} as the $\varrho_\ell^{-1}\mathbb{1}_{[v^-_{k;\ell},v^+_{k;\ell})}^{i_k}$ factor; after the $u_k$-integration it is a finite linear combination of endpoint vectors $h_t^{\ell,i_k}$ with $t \in \{v^-_{k;\ell}, v^+_{k;\ell}\}$. Each endpoint vector is bounded in $\mathcal{H}$-norm by $2\sqrt{R(0)}$ uniformly in $\ell$ by stationarity and $|R| \le R(0)$, and the simplex bound of~\citet[Prop.~3.1]{CassFerrucci2024} (written there as $|P|_{st} \lesssim (t-s)^{(n-m)H}$ for $P \in \mathcal{P}^n_m$ in that reference's word-length-$n$, chaos-order-$m$ convention; in the present length-$m$, chaos-$j$ notation this reads $(t-s)^{(m-j)H}$) --- transferred to the smoothed kernels uniformly in $\ell$ by the dominating function of~\citet[Lem.~3.8, Prop.~3.9]{CassFerrucci2024} --- applies to any partial-contraction kernel containing this factor; hence the integrated free factor contributes $\mathcal{H}$-norm at most $C$ uniformly in $\ell$. The corresponding summand for $X^{(n)}$ has factors $A_a^{(n)}$ obtained by replacing $h^\ell$ by $h^{(n),\ell}$. For any fixed ordering of these factors, the exact finite telescoping identity is
\[
\prod_{a=1}^{r} A_a-\prod_{a=1}^{r} A_a^{(n)}
=\sum_{b=1}^{r}
\left(\prod_{a<b} A_a^{(n)}\right)
\left(A_b-A_b^{(n)}\right)
\left(\prod_{a>b} A_a\right).
\]
If $A_b$ is a free-variable factor, then after integrating its $u_b$-variable the difference is a finite linear combination of $h_t^e$ factors and has $\mathcal{H}$-norm at most $C_m e_n$. If $A_b$ is a paired factor, the limiting covariance increment contains
\[
\langle h_s,h_t\rangle_{\mathcal{H}}-\langle h_s^{(n)},h_t^{(n)}\rangle_{\mathcal{H}}
=\langle h_s^e,h_t\rangle_{\mathcal{H}}+\langle h_s^{(n)},h_t^e\rangle_{\mathcal{H}},
\]
hence is bounded by $C e_n$. This is the advertised error gate. Since every summand in the telescoping sum contains the distinguished index $b$, the induction on $m$ is only bookkeeping: adjoining one more letter either appends a free factor or forms one additional paired factor, and the same identity forces at least one gate in the product difference.

The domination used to remove the smoothing parameter follows from the simplex sup-norm bounds of CF24~\citep[Prop.~3.1, Lem.~3.8, Prop.~3.9]{CassFerrucci2024} --- applicable to the $X$-kernels since the CF24 standing hypotheses hold for $\mathcal{K}$ by Lemma~\ref{lem:CF24_hypotheses} (with $T = \Delta$); for $X^{(n)}$-factors and hybrid summands the direct envelope bound below substitutes --- applied at the level of the integrated gate, with the gate factor itself controlled as follows. Write $\mathcal{K}^e := \mathcal{K}^X - \mathcal{K}^{X^{(n)}}$ for the difference kernel.

\emph{Paired gates.} In a telescoping summand $G_l^\ell$ whose gate sits at a paired position $\{b, b'\}\in P$, the gate factor $\mathbb{E}[\dot X^\ell_{u_b}\dot X^\ell_{u_{b'}}] - \mathbb{E}[\dot X^{(n),\ell}_{u_b}\dot X^{(n),\ell}_{u_{b'}}]$ integrates over the two simplex variables $u_b, u_{b'}$, conditionally on the surrounding integration variables, over one of two regions: a product region $I_b \times I_{b'}$, when at least one further integration variable separates the pair in the simplex order (so that the two slot intervals are determined independently); or the triangle $\{a \le u_b \le u_{b'} \le c\}$, when the paired indices are adjacent in the simplex order. In the product case the integral is the rectangle increment $\mathcal{K}^e(I_b\times I_{b'})$. In the adjacent case, the mixed-derivative bimeasure of $\mathcal{K}^e$ is symmetric in its two arguments and absolutely continuous --- for $\mathcal{K}^X$ because $R \in C^2((0,\infty))$ off the diagonal and no atom sits on the diagonal, $R'(0^+) = 0$ being immediate from the decomposition $R(\tau) = R(0) - \kappa\tau^{2H} + \psi(\tau)$ with $\psi \in C^2$, $\psi'(0) = 0$, and $2H - 1 > 0$ (proof of Lemma~\ref{lem:fOU_cov}); for $\mathcal{K}^{X^{(n)}}$ because its rectangle increments integrate the piecewise-constant cell density exhibited below --- so the triangle integral equals exactly one half of the rectangle increment $\mathcal{K}^e([a,c]^2)$. In both cases, the bilinear identity
\begin{equation}
\label{eq:K_diff_bilinear}
\mathcal{K}^X(s,t) - \mathcal{K}^{X^{(n)}}(s,t) \;=\; \langle h_s^e, h_t\rangle_\mathcal{H} + \langle h_s^{(n)}, h_t^e\rangle_\mathcal{H}
\end{equation}
(taking the rectangle increment of both sides) combined with Cauchy--Schwarz, the gate-difference bound $\|h^e_t - h^e_s\|_\mathcal{H} \le 2e_n$, and the uniform $\mathcal{H}$-norm bounds $\sup_t\|h_t\|_\mathcal{H}, \sup_t\|h_t^{(n)}\|_\mathcal{H} \le 2\sqrt{R(0)}$ bounds the resulting increment by a constant multiple of $e_n$, uniformly in $\ell$, $n$, and the surrounding integration variables.

\emph{Free gates.} A gate at a free position $k\in\bar P$ is \emph{not} pointwise small: its kernel realization is supported in a single interpolation cell (up to the $\varrho_\ell$-mollification) but takes values of order one there, as in $h_t^e = (1-\lambda)\mathbb{1}_{(t_i,\,t]} - \lambda\mathbb{1}_{(t,\,t_{i+1}]}$ for $t = t_i + \lambda\delta$; only its Cameron--Martin norm is small. For $H > 1/2$ the increment bimeasure of $\mathcal{K}^X$ is absolutely continuous with density bounded by $D(s,t) := C(1 + |s-t|^{2H-2})$ (Lemma~\ref{lem:CF24_hypotheses}(i) together with $R'(0^+) = 0$ as above; the $L^\infty$-representation of the inner product used here is justified after~\eqref{eq:Gl_norm_isometry} below), so any measurable $\chi$ with $|\chi| \le 1$ supported in an interval of length at most $2\delta$ satisfies
\begin{equation}
\label{eq:cell_H_bound}
\|\chi\|_{\mathcal{H}}^2 \;\le\; \int\!\!\int |\chi(u)|\,|\chi(v)|\,D(u,v)\,du\,dv \;\le\; C\,\delta^{2H},
\end{equation}
a bound stable under multiplication by bounded functions and under restriction to subintervals. Free-gate summands are therefore handled below by contracting the gate slot in $\mathcal{H}$ \emph{before} the remaining slots, rather than through a pointwise bound.

The residual factor --- the integral over the remaining $m-2$ simplex variables, with the $|P|-1$ unaffected paired covariance densities and the $j$ free indicator factors --- is itself a CF24 partial-contraction integral on the index set $\{1,\ldots,m\}\setminus\{b,b'\}$ with $|P|-1$ pairs and $j$ free indices. The piecewise-linear covariance $\mathcal{K}^{X^{(n)}}$ has 2D $\rho$-variation on $[0,\Delta]^2$ uniformly bounded in $n$: on each grid rectangle $R_{ij}$, $\mathcal{K}^{X^{(n)}}$ is the bilinear interpolation of the grid values of $\mathcal{K}^X$, hence has constant mixed derivative $\Delta_{R_{ij}}\mathcal{K}^{X}/|R_{ij}|$ there; since 2D increments are additive under splitting at the interior grid lines and on each cell the bilinear rectangle increment is exactly the area fraction of the grid increment, every rectangle increment of $\mathcal{K}^{X^{(n)}}$ is the integral of this piecewise-constant density. Consequently, over \emph{every} product partition,
\[
V_1\bigl(\mathcal{K}^{X^{(n)}};[0,\Delta]^2\bigr) \;\le\; \sum_{i,j}\bigl|\Delta_{R_{ij}}\mathcal{K}^{X}\bigr| \;\le\; V_1\bigl(\mathcal{K}^{X};[0,\Delta]^2\bigr),
\]
which is finite in the present Young regime $H > 1/2$ (proof of Lemma~\ref{lem:fOU_cov}: the fractional part has locally integrable mixed partial $\partial^2_{s,t}|s-t|^{2H} = -2H(2H-1)|s-t|^{2H-2}$ and the remainder $\tilde\psi$ is $C^2$); since $V_\rho \le V_1$ for every $\rho \ge 1$, this yields the required uniform-in-$n$ 2D $\rho$-variation bound. Moreover, the piecewise-constant cell density of $\mathcal{K}^{X^{(n)}}$ obeys, uniformly in $n$, the same envelope $D(u,v) = C(1 + |u-v|^{2H-2})$ as the increment bimeasure of $\mathcal{K}^{X}$ (cf.\ the covariance-derivative hypotheses~\citealp[eqs.~(9)--(12)]{CassFerrucci2024} under which the Cass--Ferrucci estimates are proved): each cell density is the cell average of the bimeasure of $\mathcal{K}^{X}$, so on a cell within one grid step of the diagonal it equals an $O(\delta^{2H})$ grid increment (a second difference of $R$; Lemma~\ref{lem:fOU_cov}) divided by $\delta^2$, hence is $O(\delta^{2H-2}) \le C\,|u-v|^{2H-2}$ pointwise on the cell, where $|u-v| \le 2\delta$ and $2H-2 < 0$; on a cell separated from the diagonal it is bounded by $\sup_{\mathrm{cell}}|R''| \le C(1 + |u-v|^{2H-2})$ up to a fixed factor, the cell separation and $|u-v|$ being comparable there (these are the $X^{(n)}$-analogues of the covariance-derivative hypotheses verified for $\mathcal{K}$ in Lemma~\ref{lem:CF24_hypotheses}). Hence the residual integral admits a sup-norm envelope on the simplex, uniform in $\ell$ and $n$, by a direct computation requiring no CF24 input for the $X^{(n)}$- or hybrid summands: taking absolute values, enlarging the (partial) simplex domain to the box, and applying Fubini--Tonelli --- admissible since each surviving integration variable occurs in exactly one factor --- each paired factor integrates against its two variables to at most $C\iint_{[0,\Delta]^2} D(u,v)\,du\,dv < \infty$ (every non-gate paired factor is a pure-$X$ or pure-$X^{(n)}$ density obeying the envelope $D$, the telescoping never mixing the two processes within a single factor), while each free factor $\varrho_\ell^{-1}\mathbb{1}$ has unit mass; the product of these bounds is the envelope, uniformly over the free variables. For the pure-$X$ kernels this recovers what~\citet[Prop.~3.1, Lem.~3.8, Prop.~3.9]{CassFerrucci2024} give directly under Lemma~\ref{lem:CF24_hypotheses}. Two uniform bounds result. For a summand whose gate is \emph{paired}, combining the envelope with the gate scalar bound gives the pointwise estimate
\begin{equation}
\label{eq:Gl_uniform_simplex}
\bigl|G_l^\ell(v_1,\ldots,v_j)\bigr| \;\le\; C_{m,j}(H,\theta,\sigma,\Delta)\,e_n, \qquad (v_1,\ldots,v_j)\in\Delta^j[0,\Delta] \quad (\text{paired gate}),
\end{equation}
uniformly in $\ell$ and $n$. For a summand whose gate is \emph{free}, occupying the slot $q_0 \in \{1,\ldots,j\}$, the envelope yields only the order-one bound
\begin{equation}
\label{eq:Gl_O1_simplex}
\bigl|G_l^\ell(v_1,\ldots,v_j)\bigr| \;\le\; C_{m,j}(H,\theta,\sigma,\Delta), \qquad (v_1,\ldots,v_j)\in\Delta^j[0,\Delta] \quad (\text{free gate}),
\end{equation}
together with the following \emph{sectional} smallness. For fixed values of the remaining variables $v_{-q_0} := (v_q)_{q \ne q_0}$, the section $v_{q_0} \mapsto G_l^\ell(v)$ is, by construction, a mixture --- an integral over the surviving simplex variables, whose total absolute weight is the absolute-value version of the residual partial-contraction integral bounded above, hence at most $C$ uniformly in $\ell$, $n$, and $v_{-q_0}$ --- of gate elements, each a difference of at most two functions of the form covered by~\eqref{eq:cell_H_bound} --- bounded by one and supported in a single interpolation cell, up to the $\varrho_\ell$-mollification (take $\varrho_\ell \le \delta/2$, which suffices for the limit $\ell \to \infty$) --- one per endpoint of the gate variable's integration slot, as in the finite linear combination of $h_t^e$ factors in the proof of Lemma~\ref{lem:error_gate}; the resulting factor $2$ is absorbed into $C_{m,j}$. Minkowski's integral inequality in $\mathcal{H}$ therefore gives
\begin{equation}
\label{eq:free_gate_section}
\sup_{v_{-q_0}} \bigl\| G_l^\ell(\,\cdot\,, v_{-q_0}) \bigr\|_{\mathcal{H}} \;\le\; C_{m,j}(H,\theta,\sigma,\Delta)\,\delta^{H},
\end{equation}
uniformly in $\ell$ and $n$ (the ordering constraints of the simplex only restrict the gate elements to subintervals, which~\eqref{eq:cell_H_bound} tolerates). Lemma~\ref{lem:symm_ext_Linf} transfers the pointwise simplex bounds to the box: both the sorting-based extension and --- being an average of slot permutations, each bounded by the simplex essential supremum --- the symmetrization $\widetilde G_l^\ell$ (here and below $\widetilde{\,\cdot\,}$ denotes the symmetrization, i.e.\ the average over slot permutations, per the convention of Appendix~\ref{subsec:isonormal_setup}) satisfy $\|\cdot\|_{L^\infty([0,\Delta]^j)} \le C_{m,j}\,e_n$ for paired-gate summands and $\le C_{m,j}$ for free-gate summands, also uniform in $\ell, n$.

Convergence of the smooth-approximant kernels to the limiting Cass--Ferrucci kernels takes place in $\mathcal{H}^{\otimes j}$-norm for the two pure kernels $f_{w,j}^{X}$ and $f_{w,j}^{X^{(n)}}$ separately: for $X$ by~\citet[Thm.~2.3, Prop.~3.9]{CassFerrucci2024}, applicable since the CF24 standing hypotheses hold for $\mathcal{K}$ (Lemma~\ref{lem:CF24_hypotheses}); for the piecewise-linear $X^{(n)}$, whose truncated-signature coordinates are polynomials in the finitely many jointly Gaussian increments $(\Delta X_i)_{i \le n-2}$ and whose increment covariance carries the bounded piecewise-constant cell density exhibited above, the same expansion with kernels of the identical partial-contraction form, and the removal of the mollification, are classical Wiener--It\^o calculus (Isserlis/diagram formula,~\citealp[Thm.~1.28]{Janson1997}), no singular regularization being involved. Throughout, the $X^{(n)}$-expansion --- natively a chaos expansion over the isonormal space of $X^{(n)} - X_0$ --- is transported into $\mathcal{H}$ by the isometric embedding $\iota_n$ determined by $\mathbb{1}_{[0,t]} \mapsto h^{(n)}_t$ (isometric, since $\Cov(X^{(n)}_s - X_0,\, X^{(n)}_t - X_0) = \langle h^{(n)}_s, h^{(n)}_t\rangle_\mathcal{H}$ on the generating set), which intertwines multiple Wiener integrals, $I_j^{(n)}(f) = I_j^{\mathcal{H}}(\iota_n^{\otimes j} f)$~\citep[\S1.1]{Nualart2006}; under $\iota_n^{\otimes j}$ the kernels are identified with the $\mathcal{H}$-valued kernels compared here. The hybrid telescoping summands mix $X$- and $X^{(n)}$-factors and are not covered verbatim by~\citet[Prop.~3.9]{CassFerrucci2024}; their limits are constructed factorwise: each paired density and each free factor converges a.e.\ as $\ell \to \infty$ by the same smoothing-removal argument as in the pure case, and each summand's integrand is dominated, uniformly in $\ell$, by the integrable envelope underlying the residual bound above, so dominated convergence in the integrated simplex variables yields $G_l^\ell \to G_l$ a.e.\ in the kernel variables on $\Delta^j[0,\Delta]$, with the bounds~\eqref{eq:Gl_uniform_simplex}--\eqref{eq:free_gate_section} inherited by the limit (the sectional bound~\eqref{eq:free_gate_section} passes to the limit by dominated convergence in the bimeasure representation of the $\mathcal{H}$-norm below: the sections converge a.e.\ with the uniform bound~\eqref{eq:Gl_O1_simplex}, and the envelope $C^2 D$ is integrable). Symmetric extension preserves a.e.\ convergence. Telescoping after the limit yields the same finite decomposition for the kernel difference $g_{w,j}^{(n)} = f_{w,j}^X - f_{w,j}^{X^{(n)}}$ at the level of the limiting Cass--Ferrucci kernels.

By Lemma~\ref{lem:error_gate}, the coordinate error has the finite orthogonal chaos expansion
\begin{equation}
\label{eq:young_error_chaos}
\langle w,E_m\rangle
=\sum_{\substack{0\le j\le m\\ j\equiv m\,(2)}} I_j^\mathcal{H}(g_{w,j}^{(n)}),
\qquad
g_{w,j}^{(n)}:=f_{w,j}^{X}-f_{w,j}^{X^{(n)}} .
\end{equation}
For the deterministic zeroth chaos ($j=0$, only when $m$ is even; every gate is then a paired gate), the finite pairing expansion described above gives an absolute scalar bound $|g_{w,0}^{(n)}| \le C_m\,e_n$. For $j\ge 1$, the squared $\mathcal{H}^{\otimes j}$-norm of the symmetrization $\widetilde G_l$ of each telescoping summand expands as a $2j$-fold integral against the cross-copy bimeasures of $R$:
\begin{equation}
\label{eq:Gl_norm_isometry}
\bigl\|\widetilde G_l\bigr\|_{\mathcal{H}^{\otimes j}}^2 \;=\; \int_{[0,\Delta]^{2j}} \widetilde G_l(u)\,\widetilde G_l(v)\,\prod_{q=1}^{j} dR_{i_{k_q}}(u_q, v_q),
\end{equation}
where each $dR_{i_{k_q}}$ is the coordinate-$i_{k_q}$ stationary autocovariance bimeasure; the multiple-integral isometry~\citep[eqs.~(20)--(21)]{CassFerrucci2024} (eq.~\eqref{eq:wiener_isometry}) then reads $\mathbb{E}[I_j^\mathcal{H}(G_l)^2] = j!\,\|\widetilde G_l\|_{\mathcal{H}^{\otimes j}}^2$, the combination used below. For $H > 1/2$ the increment bimeasure is absolutely continuous on all of $[0,\Delta]^2$: off the diagonal $R \in C^2((0,\infty))$ with $|R''(\tau)| \le C(1 + |\tau|^{2H-2})$ on $(0,\Delta]$ (Lemma~\ref{lem:CF24_hypotheses}(i), with $T = \Delta$), and no singular part sits on the diagonal since $R'(0^+) = 0$ (as noted above; likewise $\mathcal{K}^{X^{(n)}}$ carries only its piecewise-constant cell density). Hence $|dR(u,v)| \le D(u,v)\,du\,dv$ with $D(s,t) := C(1 + |s-t|^{2H-2})$, and $\int_{[0,\Delta]^2} D(u,v)\,du\,dv < \infty$ since $2H - 2 > -1$. Consequently every bounded measurable function on $[0,\Delta]$ belongs to $\mathcal{H}$, and the representation of the Cameron--Martin inner product as integration against the bimeasure extends from step functions to $L^\infty([0,\Delta])$ --- and slotwise to bounded kernels in $\mathcal{H}^{\otimes j}$ --- by dominated convergence with envelope $D$: for $\phi, \psi \in L^\infty([0,\Delta])$ and uniformly bounded step functions $\phi_k \to \phi$, $\psi_k \to \psi$ a.e., the sequence $(\phi_k)$ is $\mathcal{H}$-Cauchy and $\langle \phi, \psi\rangle_\mathcal{H} = \lim_k \int\!\!\int \phi_k(u)\,\psi_k(v)\,dR(u,v) = \int\!\!\int \phi(u)\,\psi(v)\,dR(u,v)$. This justifies~\eqref{eq:Gl_norm_isometry} for the $L^\infty$ kernels at hand, the cell bound~\eqref{eq:cell_H_bound} above, and the slotwise contraction below.

\emph{Perfect-matching factorization: paired-gate summands.} The integrand of~\eqref{eq:Gl_norm_isometry} is a product of two factors $\widetilde G_l(u), \widetilde G_l(v)$ --- for a paired-gate summand each bounded uniformly by $C_{m,j}\,e_n$ via~\eqref{eq:Gl_uniform_simplex} together with Lemma~\ref{lem:symm_ext_Linf} --- multiplied by the $j$ cross-copy bimeasures $\prod_q D(u_q, v_q)\,du_q\,dv_q$. Each integration variable $u_q$ (respectively $v_q$) appears in \emph{exactly one} $D$-factor, namely the cross-copy contraction $D(u_q, v_q)$: the within-copy structure of $\widetilde G_l$ has been absorbed into the uniform sup-norm bound at the~\citet[Prop.~3.1]{CassFerrucci2024} step preceding~\eqref{eq:Gl_uniform_simplex}, so no within-copy $D$-factor remains explicit, and the only $D$-dependence on $u_q$ is through $D(u_q, v_q)$. The $2j$-fold integral therefore factorizes via Fubini--Tonelli into $j$ independent two-variable integrals:
\begin{equation}
\label{eq:Gl_norm_factorized}
\bigl\|\widetilde G_l\bigr\|_{\mathcal{H}^{\otimes j}}^2 \;\le\; (C_{m,j})^2\,e_n^2\,\biggl(\int_{[0,\Delta]^2} D(u,v)\,du\,dv\biggr)^j \;\le\; C'_{m,j}(H,\theta,\sigma,\Delta)\,e_n^2.
\end{equation}

\emph{Free-gate summands: contracting the gate slot first.} For a free-gate summand the available pointwise factor bound is only the order-one estimate~\eqref{eq:Gl_O1_simplex}, and~\eqref{eq:Gl_norm_factorized} is not available. Instead, observe that slot permutations act unitarily on $\mathcal{H}^{\otimes j}$ and the symmetrization operator is the average of $j!$ of them --- the orthogonal projection onto $\mathcal{H}^{\odot j}$ --- so $\|\widetilde G_l\|_{\mathcal{H}^{\otimes j}} \le \|G_l\|_{\mathcal{H}^{\otimes j}}$, and it suffices to bound the unsymmetrized summand, whose gate slot $q_0$ is well defined. By Fubini--Tonelli --- justified by the order-one bound~\eqref{eq:Gl_O1_simplex} and the integrable envelope $D$ --- the $\mathcal{H}^{\otimes j}$-norm may be evaluated by contracting the gate slot first:
\begin{equation}
\label{eq:gate_slot_first}
\bigl\|G_l\bigr\|^2_{\mathcal{H}^{\otimes j}} \;=\; \int_{[0,\Delta]^{2(j-1)}} \bigl\langle G_l(\,\cdot\,, u_{-q_0}),\, G_l(\,\cdot\,, v_{-q_0}) \bigr\rangle_{\mathcal{H}}\, \prod_{q \ne q_0} dR_{i_{k_q}}(u_q, v_q),
\end{equation}
the inner factor being the slot-$q_0$ Cameron--Martin pairing of the sections, well defined by the $L^\infty$ extension above. By the sectional bound~\eqref{eq:free_gate_section} and Cauchy--Schwarz in $\mathcal{H}$, this pairing is bounded by $(C_{m,j}\,\delta^H)^2$ uniformly in $(u_{-q_0}, v_{-q_0})$, whence
\begin{equation}
\label{eq:free_gate_norm}
\bigl\|\widetilde G_l\bigr\|_{\mathcal{H}^{\otimes j}}^2 \;\le\; \bigl\|G_l\bigr\|_{\mathcal{H}^{\otimes j}}^2 \;\le\; (C_{m,j})^2\,\delta^{2H}\,\biggl(\int_{[0,\Delta]^2} D(u,v)\,du\,dv\biggr)^{j-1} \;\le\; C'_{m,j}(H,\theta,\sigma,\Delta)\,\delta^{2H}.
\end{equation}

Summing over the $(m-j)/2 + j \le m$ telescoping summands per pairing $P \in \mathcal{P}^m_j$ (one summand per paired or free factor) and over the $|\mathcal{P}^m_j| = m!/(2^p\,p!\,j!) = O_m(1)$ pairings (with $p = (m-j)/2$ pairs), the triangle inequality in $\mathcal{H}^{\otimes j}$ (symmetrization is linear and norm-contracting) combines~\eqref{eq:Gl_norm_factorized} and~\eqref{eq:free_gate_norm}; since $e_n \le C\,\delta^H$ by~\eqref{eq:error_gate}, both envelopes are $O(\delta^{2H})$, and after absorbing the factor $j! \le m!$ into the constant,
\begin{equation}
\label{eq:young_kernel_gate_bound}
j!\,\bigl\|\widetilde g_{w,j}^{(n)}\bigr\|_{\mathcal{H}^{\otimes j}}^2 \;\le\; C_{m,j}(H,\theta,\sigma,\Delta)\,\delta^{2H}, \qquad 0\le j\le m,\ j\equiv m\,(2).
\end{equation}
Orthogonality of distinct chaos orders under the multiple-integral isometry~\citep[eqs.~(20)--(21)]{CassFerrucci2024} then gives $\mathbb{E}[|\langle w,E_m\rangle|^2] = \sum_j j!\|\widetilde g_{w,j}^{(n)}\|^2 \le C_m\,\delta^{2H}$ (the $j = 0$ term obeys $|g_{w,0}^{(n)}|^2 \le C_m^2 e_n^2 \le C\,\delta^{2H}$). Summing the squared $L^2$-norm over the $d^m$ coordinates of $E_m \in (\mathbb{R}^d)^{\otimes m}$ absorbs the tensor combinatorics into $C_m$.

This direct comparison covers every term type that would appear in a formal multilinear expansion: the contributions with no error-gate factor cancel between $f_{w,j}^X$ and $f_{w,j}^{X^{(n)}}$, while mixed and all-error terms (extreme cases of the telescoping where $b = 1$ or where the gate cascades through every position) are subsumed by the same one-gate-per-summand structure with $r$ summands of identical form. No integration-by-parts recursion is used, and no step attempts to convert a cluster of $de$ differentials into $dX$ differentials. Hence
\begin{equation}
\label{eq:young_level_m}
\mathbb{E}[\|E_m\|_{\mathrm{HS}}^2] \le C_m(H,\theta,\sigma,\Delta) \cdot \delta^{2H}, \qquad m = 3, \ldots, M.
\end{equation}
Together with the zero level-1 error and the level-2 area bound, and since $2H < \min(4H-1,\,2)$ for $H > 1/2$, the path interpolation rate dominates the L\'evy area rate at every level. Thus:
\begin{equation}
\label{eq:young_mse}
\mathbb{E}\bigl[\|S^{(M)}(\mathbf{X})_{0,\Delta} - S^{(M)}(\mathbf{X}^{(n)})_{0,\Delta}\|_{\mathrm{HS}}^2\bigr] = \sum_{m=1}^M \mathbb{E}[\|E_m\|_{\mathrm{HS}}^2] \le C(M,H,\theta,\sigma,\Delta) \cdot \delta^{2H} \quad \text{for } H > 1/2.
\end{equation}
The Young-regime exponent is therefore $\gamma = 2H$ exactly, with no $\varepsilon$-loss and no logarithmic factor.

\medskip
\noindent\textbf{Case 2: Rough regime ($H \in (1/4, 1/2)$).}
When $H \in (1/4, 1/2)$, the paths have $p$-variation only for $p > 1/H > 2$, so rough-path machinery is required. We establish the rate $O(\delta^{4H-1-\varepsilon})$ at truncation level $M$ (for any $\varepsilon > 0$) via level-by-level Gaussian $L^2$ bounds, with the level-2 rate $O(\delta^{4H-1})$ exact from Lemma~\ref{lem:fOU_area}; this bypasses the rough-path Lipschitz bound (see Remark~\ref{rem:inhomogeneous} for why that bound is suboptimal here). The boundary $H = 1/2$ is delegated to Appendix~\ref{app:gamma_semi}.

\paragraph{Level-by-level $L^2$ bounds}
Level~1 has zero error (endpoints match). At level~2, the error is the L\'evy area error: $\mathbb{E}[\|E_2\|_{\mathrm{HS}}^2] = O(\delta^{4H-1})$ by Lemma~\ref{lem:fOU_area}, where $E_m := S^{(m)}(\mathbf{X})_{0,\Delta} - S^{(m)}(\mathbf{X}^{(n)})_{0,\Delta}$ as in Case~1.

\paragraph{Level-$m \ge 3$ bound via Friz--Riedel Theorem~5}
For levels $m \ge 3$, we apply~\citet[Thm.~5]{FrizRiedel2014} to the Gaussian rough path $\mathbf{X}$ and its PL interpolation $\mathbf{X}^{(n)}$, using the 2D $\rho$-variation of $K_X$ established in Lemma~\ref{lem:fOU_cov}. Fix $\rho \in (\rho_*(H),\, 2)$ with $\rho_*(H) = \max(1,\, 1/(2H))$, and choose FR parameters $\gamma_{\mathrm{FR}} := \rho / \varepsilon'$ with $\varepsilon' > 0$ satisfying both $\varepsilon' < 1$ (equivalent to $\gamma_{\mathrm{FR}} > \rho$) and $\varepsilon' > \rho - 1$ (equivalent to the FR condition $1/\gamma_{\mathrm{FR}} + 1/\rho > 1$, since $\varepsilon'/\rho + 1/\rho > 1 \Leftrightarrow \varepsilon' > \rho - 1$). For $H \in (1/4, 1/2)$ the admissibility $\rho > 1/(2H)$ then forces $\varepsilon' > \rho - 1 > (1-2H)/(2H)$, which is bounded away from $0$. Applying~\citet[Thm.~5]{FrizRiedel2014} at truncation level $M$ with $r = 2$:
\begin{equation}
\label{eq:FR_fOU_rough}
\bigl\|\varrho_{q\text{-var};[0, \Delta]}(S_M(\mathbf{X}^{(n)}),\,S_M(\mathbf{X}))\bigr\|_{L^2} \le C_M\,\sup_{0 \le t \le \Delta}\|X^{(n)}_t - X_t\|_{L^2}^{1 - \varepsilon'}.
\end{equation}
The sup-$L^2$ PL error satisfies $\sup_t \|X^{(n)}_t - X_t\|_{L^2} \le C\,\delta^H$ by~\eqref{eq:modulus_bound}, which holds for all $H \in (1/4,1)$.

\paragraph{Conclusion}
As in Appendix~\ref{app:gamma_semi}, the inhomogeneous distance $\varrho_{q\text{-var}}$ dominates every level-$m$ endpoint difference ($m \le M$) via the trivial partition~\citep[p.~159]{FrizRiedel2014}, so
\[
\mathbb{E}[\|E_m\|_{\mathrm{HS}}^2] \;\le\; C_m\,\bigl\|\varrho_{q\text{-var};[0,\Delta]}(S_M(\mathbf{X}^{(n)}), S_M(\mathbf{X}))\bigr\|_{L^2}^2 \;\le\; C_m'\,\delta^{2H(1-\varepsilon')}.
\]
Fix a target $\varepsilon > 0$ and choose $\rho \in (1/(2H),\, 2)$ and $\varepsilon' \in (\rho-1,\, 1)$ sufficiently close to the FR admissibility boundary $(1/(2H),\, (1-2H)/(2H))$ that $2H(1 - \varepsilon') > 4H - 1 - \varepsilon$; such $(\rho, \varepsilon')$ exist because the boundary identity $2H\bigl(1 - (1-2H)/(2H)\bigr) = 2H - (1-2H) = 4H - 1$ realizes the limiting exponent and the strict inequality is preserved on a non-empty open neighborhood of the boundary. With $(\rho, \varepsilon')$ fixed, choose $q > 2\gamma_{\mathrm{FR}} = 2\rho/\varepsilon'$ once and for all; the FR constant $C(q,\rho,\gamma_{\mathrm{FR}}, K, M)$ in~\eqref{eq:FR_Thm5} is finite at this fixed tuple and may depend on $\varepsilon$, $H$, $M$. Therefore:
\begin{equation}
\label{eq:rough_level_m}
\mathbb{E}[\|E_m\|_{\mathrm{HS}}^2] \le C_m(H,\theta,\sigma,\Delta,\varepsilon)\cdot \delta^{4H - 1 - \varepsilon} \quad \text{for every } m \ge 3,\ H \in (1/4, 1/2),\ \varepsilon > 0.
\end{equation}
This bound \emph{matches} (rather than strictly dominates) the level-2 L\'evy area rate $O(\delta^{4H-1})$ of Lemma~\ref{lem:fOU_area}, modulo an arbitrarily small $\varepsilon$-loss at higher levels; the FR admissibility constraint $\varepsilon' > \rho - 1 > 0$ for $H \in (1/4, 1/2)$ is what makes the $\varepsilon$-loss unavoidable at levels $m \ge 3$ (with the present proof technique). The boundary $H = 1/2$ is delegated to Appendix~\ref{app:gamma_semi}, which gives $\gamma = 1-\varepsilon$ at truncation level $M$ (with $\gamma = 1$ exact at level~2). Assembling:
\begin{equation}
\label{eq:rough_mse}
\begin{aligned}
\mathbb{E}\bigl[\|S^{(M)}(\mathbf{X})_{0,\Delta} - S^{(M)}(\mathbf{X}^{(n)})_{0,\Delta}\|_{\mathrm{HS}}^2\bigr]
&= \sum_{m=1}^M \mathbb{E}[\|E_m\|_{\mathrm{HS}}^2] \\
&\le C(M,H,\theta,\sigma,\Delta,\varepsilon) \cdot \delta^{4H-1-\varepsilon}, \quad H \in (1/4, 1/2),\ \varepsilon > 0.
\end{aligned}
\end{equation}

\paragraph{Combined rate}
Combining~\eqref{eq:young_mse} and~\eqref{eq:rough_mse}:
\begin{equation}
\gamma =
\begin{cases}
4H - 1 - \varepsilon\ (\varepsilon > 0) & H \in (1/4, 1/2), \\
1 \text{ at level 2; } 1-\varepsilon \text{ at level } M & H = 1/2, \\
2H \text{ exactly} & H > 1/2.
\end{cases}
\end{equation}
The limiting exponents of the two regimes agree at the boundary $H = 1/2$ ($4H-1 = 2H = 1$) and match the direct analysis of Appendix~\ref{app:gamma_semi}. The $\varepsilon$-loss in the rough regime $H \in (1/4, 1/2)$ at truncation level $M$ is driven by the Friz--Riedel Theorem~5 admissibility constraint $\varepsilon' > \rho - 1$ in~\eqref{eq:FR_fOU_rough}, which forces $\varepsilon' > (1-2H)/(2H) > 0$ rather than $\varepsilon' \to 0^+$; the level-2 L\'evy area exponent $4H-1$ is attained cleanly by Lemma~\ref{lem:fOU_area}, and the Young regime $H > 1/2$ attains $\gamma = 2H$ exactly. The $2H$ exponent is consistent with prior signature-specific discretization results: for $H > 1/2$,~\citet[Thm.~2.2]{Passeggeri2020} proves a bias estimate at rate $2H$ for the expected signature of the piecewise-linear approximation of fractional Brownian motion. Here the same exponent governs the \emph{mean-squared} pathwise approximation error for truncated signatures of stationary fOU --- a second-moment rather than first-moment statement (by Jensen, our bound implies only the bias rate $\delta^{H}$), so the two results are complementary rather than nested.
\end{proof}

\begin{remark}[Gaussian $L^2$ approach vs.\ rough-path metric]
\label{rem:inhomogeneous}
The general Lyons-lift Lipschitz bound~\citep[Thm.~9.10, Cor.~9.11]{FrizVictoir2010} controls the signature difference by the $p$-variation rough-path distance $\varrho_{p\text{-var}}(\mathbf{X}, \mathbf{X}^{(n)})$, which involves the $p/2$-variation of the area difference over \emph{all} subintervals, not the global area value $|\mathbb{A}_{0,\Delta} - \mathbb{A}^{(n)}_{0,\Delta}|$. For piecewise-linear approximations, the global area error benefits from algebraic cancellations in the trapezoidal scheme, achieving $O(\delta^{(4H-1)/2})$ in $L^2$~\citep{NeuenkirchTindelUnterberger2010}. The $p/2$-variation norm $\|\mathbb{A} - \mathbb{A}^{(n)}\|_{p/2\text{-var}}$, however, sums uncancelled local errors and scales as $O(\delta^{2H-2/p})$, which degrades as $p \downarrow 1/H$. Applying~\citet[Thm.~9.10]{FrizVictoir2010} with this norm yields a rate strictly worse than $4H-1$ for levels $m \ge 3$. Our level-$m \ge 3$ bounds~\eqref{eq:young_level_m} (Young regime, via the Cass--Ferrucci finite-chaos comparison and the interpolation-error gate~\eqref{eq:error_gate}) and~\eqref{eq:rough_level_m} (rough regime, via~\citealp[Thm.~5]{FrizRiedel2014}) both circumvent this difficulty: the Young-regime bound works chaos-by-chaos and never differentiates the error path recursively, while the rough-regime bound controls the $q$-variation rough-path distance via the sup-$L^2$ PL error, sidestepping the $p/2$-variation area norm.
\end{remark}

\begin{remark}
\label{rem:fOU_consistency}
At the critical $H = 1/4$, we have $\gamma = 4H - 1 = 0$: the error fails to vanish, consistent with the failure of the second-level rough paths of dyadic approximations to converge (in $L^1$ with respect to the $p$-variation distance) to a geometric rough path~\citep[Thm.~2]{CoutinQian2002}.
\end{remark}

%%%%%%%%%%%%%%%%%%%%%%%%%%%%%%%%%%%%%%%%%%%%%%%%%%%%%%%%%%%%%%%%%%%%%%%%%%%%%%%
\subsection{Bounded Variation Paths ($\gamma = 2$)}
\label{app:gamma_bv}
%%%%%%%%%%%%%%%%%%%%%%%%%%%%%%%%%%%%%%%%%%%%%%%%%%%%%%%%%%%%%%%%%%%%%%%%%%%%%%%

For paths of bounded variation, iterated integrals reduce to Riemann--Stieltjes integrals. By~\citet[Prop.~1.28]{FrizVictoir2010}, the piecewise-linear interpolation converges uniformly to $X$ and satisfies $|X^{(n)}|_{1\text{-var};[s,t]} \le |X|_{1\text{-var};[s,t]}$. The dominant error is again the L\'evy area. By Chen's identity~\citep{Chen1954}, the level-2 error decomposes additively over subintervals: $S^{(2)}(\mathbf{X})_{0,\Delta} - S^{(2)}(\mathbf{X}^{(n)})_{0,\Delta} = \sum_{i=0}^{n-2} \mathbb{A}_i$, where $\mathbb{A}_i$ is the L\'evy area on $[t_i, t_{i+1}]$ because the symmetric part and the cross-terms from path increments are identical for $\mathbf{X}$ and $\mathbf{X}^{(n)}$. Under Assumption~\ref{ass:M} with $p = 1$ and $\beta = 1$, since $\mathbb{A}_i$ is the antisymmetric part of $\pi_2(\mathbf{X}_{t_i,t_{i+1}})$ and antisymmetrization is an orthogonal projection in Hilbert--Schmidt norm, Lemma~\ref{lem:level_moments} (whose proof applies verbatim on the subinterval $[t_i, t_{i+1}]$) gives $\mathbb{E}[\|\mathbb{A}_i\|_{\mathrm{HS}}^2] \le d^2 C_2'\,\delta^4$ with $C_2' := C^{\mathrm{sig}}_{2,2}\bigl(\Lambda_1\,\Gamma(3)\bigr)^{-2}$. By Cauchy--Schwarz, each cross-term satisfies $|\mathbb{E}[\langle \mathbb{A}_i, \mathbb{A}_j \rangle]| \le d^2 C_2'\,\delta^4$, so $\mathbb{E}[\|\sum_i \mathbb{A}_i\|_{\mathrm{HS}}^2] \le (n-1)^2 d^2 C_2'\,\delta^4 = O(\Delta^2 \delta^2)$, yielding $\gamma = 2$. For higher levels $m \ge 3$, we quantify the rate via Chen's identity and a telescoping expansion in the truncated tensor algebra. Writing $\mathbf{X}_{0,\Delta} = \bigotimes_{i=0}^{n-2}\mathbf{X}_{t_i,t_{i+1}}$ and $\mathbf{X}^{(n)}_{0,\Delta} = \bigotimes_{i=0}^{n-2}\mathbf{X}^{(n)}_{t_i,t_{i+1}}$ and then projecting onto level~$m$:
\begin{equation}
\pi_m(\mathbf{X}_{0,\Delta} - \mathbf{X}^{(n)}_{0,\Delta}) = \sum_{j=0}^{n-2}\sum_{\substack{(a_0,\ldots,a_{n-2}):\\ \sum a_i = m,\, a_j \ge 2}} \!\!\bigotimes_{i<j}\pi_{a_i}(\mathbf{X}_{t_i,t_{i+1}}) \otimes \pi_{a_j}\!\bigl(\mathbf{X}_{t_j,t_{j+1}} - \mathbf{X}^{(n)}_{t_j,t_{j+1}}\bigr) \otimes \bigotimes_{i>j}\pi_{a_i}(\mathbf{X}^{(n)}_{t_i,t_{i+1}}),
\end{equation}
where the constraint $a_j \ge 2$ holds because $\mathbf{X}$ and $\mathbf{X}^{(n)}$ agree at levels~0 and~1 on every subinterval (endpoints coincide). For a BV path $X$, the level-$a$ iterated integral $\pi_a(\mathbf{X}_{s,t}) = \int_{s < u_1 < \cdots < u_a < t}dX_{u_1}\otimes\cdots\otimes dX_{u_a}$ satisfies, via integration against the product variation measure over the simplex --- whose total mass is $\|X\|_{1\text{-var};[s,t]}^a/a!$ --- the factorial estimate $\|\pi_a(\mathbf{X}_{s,t})\|_{\mathrm{HS}} \le \|X\|_{1\text{-var};[s,t]}^a/a!$; the same bound holds for $\mathbf{X}^{(n)}$ (a BV path with $\|X^{(n)}\|_{1\text{-var};[s,t]} \le \|X\|_{1\text{-var};[s,t]}$ by~\citealp[Prop.~1.28]{FrizVictoir2010}). Applied on each subinterval of length $\delta$, the triangle inequality yields $\|\pi_{a_j}(\mathbf{X}_{t_j,t_{j+1}} - \mathbf{X}^{(n)}_{t_j,t_{j+1}})\|_{\mathrm{HS}} \le 2\|X\|_{1\text{-var};[t_j,t_{j+1}]}^{a_j}/a_j!$. Thus, writing $L_i := \|X\|_{1\text{-var};[t_i,t_{i+1}]}$, each term in the telescoping sum is bounded pointwise by
\[
\frac{2}{\prod_i a_i!}\prod_i L_i^{a_i} \;\le\; 2\prod_i L_i^{a_i}
\]
(using $a_i! \ge 1$; the intermediate factorial factor is not needed below and is dropped). Under Assumption~\ref{ass:M} with $p = 1$ and $\beta = 1$, $\mathbb{E}[L_i^{2m}] \le C_{2m}\delta^{2m}$; H\"older's inequality with exponents $p_i = m/a_i$ gives $\mathbb{E}\bigl[\prod_i L_i^{2a_i}\bigr] \le \prod_i(\mathbb{E}[L_i^{2m}])^{a_i/m} \le C_{2m}\delta^{2m}$. Each term therefore has $L^2$ norm at most $C_m\delta^m$. The number of terms is $(n-1)\cdot |\{(a_i): \sum a_i = m,\, a_j \ge 2\}| \le C_m n^{m-1}$. By Minkowski's inequality,
\begin{equation}
\bigl(\mathbb{E}[\|\pi_m(\mathbf{X}_{0,\Delta} - \mathbf{X}^{(n)}_{0,\Delta})\|^2_{\mathrm{HS}}]\bigr)^{1/2} \le C_m\, n^{m-1}\,\delta^m = C_m\,\Delta^{m-1}\,\delta,
\end{equation}
yielding $\mathbb{E}[\|\pi_m\|^2_{\mathrm{HS}}] \le C_m^2\Delta^{2(m-1)}\delta^2 = O(\delta^2)$ for each $m \ge 2$. Summing over $m = 2,\ldots,M$ (level~1 contributes zero error since endpoints coincide) gives $\mathbb{E}[\|S^{(M)}(\mathbf{X})_{0,\Delta} - S^{(M)}(\mathbf{X}^{(n)})_{0,\Delta}\|^2_{\mathrm{HS}}] = O(\delta^2)$, confirming $\gamma = 2$.

% Appendix C: fOU Expected Signatures and Block Covariance Decay
\section{Stationary fOU: Expected Signatures and Block Covariance Decay}
\label{app:expected_signature_fOU}

This appendix has two purposes: (i) it derives the level-2 and level-4 expected-signature formulas for stationary fractional Ornstein--Uhlenbeck (fOU) processes used as deterministic ground truth in Section~\ref{sec:numerical_results} --- in closed form for every $H \in (1/4,1)$, consistent with the Cass--Ferrucci 2D Young-integral construction --- and (ii) it proves Proposition~\ref{prop:fOU_covariance_decay} via four reusable lemmas (Sections~\ref{subsec:R_tails}--\ref{subsec:chaos_block}). \citet{CassFerrucci2024} recently established the general theory for expected signatures of Gaussian processes with Hurst parameter $H \in (1/4, 1)$; their formulas are proved via Wiener chaos expansion and a regularization technique that avoids singular covariance derivatives.

For $H < 1/2$, standard Wick formulas involve $R''(\tau) \sim \tau^{2H-2}$, which diverges at the origin. The framework of~\citet{CassFerrucci2024} defines the expected signature via the 2D Young integral of the covariance kernel $\mathcal{K}$, yielding well-defined values for $H > 1/4$. Since $R'(\tau) \sim \tau^{2H-1}$ is integrable for all $H > 0$, individual $R'$ factors pose no difficulty; the $H > 1/4$ threshold arises from products $R'(\tau)^2 \sim \tau^{4H-2}$, which require $4H - 2 > -1$.

\subsection{Spectral Representation and Notation}
\label{subsec:fOU_setup}

We work with the stationary fOU process of Section~\ref{sec:verification}. We suppress the coordinate index $i$, writing $R, \zeta, f, \theta, \sigma$ for $R^i, \zeta_i, f_i, \theta_i, \sigma_i$. When coordinate indices appear alongside derivative primes, we write $R_i$, $R'_i$, $R''_i$ for $R^i$, $(R^i)'$, $(R^i)''$. The stationary autocovariance admits the spectral representation \citep[Rem.~2.4, eq.~(2.2)]{CheriditoKawaguchiMaejima2003}:
\begin{equation}
\label{eq:fOU_spectral}
R(\tau) = 2\zeta \int_0^\infty \cos(\lambda\tau) \, f(\lambda) \, d\lambda, \qquad f(\lambda) = \frac{\lambda^{1-2H}}{\theta^2 + \lambda^2},
\end{equation}
where $f(\lambda)$ is the (unnormalized) spectral density and $\zeta$ is the normalization constant chosen so that $R(0) = \Var(X_0) = \sigma^2\Gamma(2H+1)/(2\theta^{2H})$:
\begin{equation}
\label{eq:spectral_normalization}
\zeta = \frac{\Var(X_0)}{2 \int_0^\infty f(\lambda) \, d\lambda} = \frac{\sigma^2 \Gamma(2H+1)}{2\theta^{2H} \cdot 2 \int_0^\infty f(\lambda) \, d\lambda}.
\end{equation}
The first derivative is obtained by differentiating under the integral:
\begin{equation}
\label{eq:fOU_R_prime_spectral}
R'(\tau) = -2\zeta \int_0^\infty \lambda \sin(\lambda\tau) \, f(\lambda) \, d\lambda.
\end{equation}

\begin{remark}[Validity of the differentiation]
For $H > 1/2$, differentiation in~\eqref{eq:fOU_R_prime_spectral} is justified by dominated convergence. For $H \le 1/2$, the integral is only conditionally convergent at infinity. In that regime, the derivative formula holds for $\tau > 0$ as an oscillatory-integral identity, justified by the $C^\infty((0,\infty))$ argument in Lemma~\ref{lem:R_tails}.
\end{remark}

\subsection{The Centered Increment Process}
\label{subsec:centered_increment}

The signature is translation-invariant, so we work with the centered increment process
\begin{equation}
\tilde{X}_t := X_t - X_0, \qquad t \ge 0.
\end{equation}
This process satisfies $\tilde{X}_0 = 0$, $\mathbb{E}[\tilde{X}_t] = 0$ (since $X$ is stationary with constant mean), and has covariance kernel
\begin{equation}
\label{eq:K_definition}
\mathcal{K}(s,t) := \Cov(\tilde{X}_s, \tilde{X}_t) = R(|t-s|) - R(s) - R(t) + R(0).
\end{equation}
Since $|R'(\tau)| \le C\,\tau^{2H-1}$ on $(0,\Delta]$ (eq.~\eqref{eq:R_prime_local}, Lemma~\ref{lem:fOU_cov}), $R$ is $\min(2H,1)$-H\"older continuous on $[0,\Delta]$, and $\mathcal{K}$ inherits this regularity in each variable. The additional terms $-R(s) - R(t) + R(0)$ depend on a single argument each, so they contribute zero mixed 2D variation; hence $\mathcal{K}$ inherits the 2D $\rho$-variation of $(s,t) \mapsto R(|t-s|)$, which is finite for every $\rho > \rho_*(H) := \max(1,\,1/(2H))$ by Proposition~\ref{prop:fOU_assumptions}. Since $\rho_*(H) < 2$ for all $H > 1/4$, we may choose $\rho < 2$, so the Gaussian rough-path criterion~\citep[Thm.~15.33]{FrizVictoir2010} applies. The quantitative estimates of~\citet{CassFerrucci2024} used in this paper --- Thm.~2.3, Prop.~3.1, Lem.~3.8, Prop.~3.9 --- are proved under the standing covariance-derivative hypotheses~\citep[eqs.~(9)--(12)]{CassFerrucci2024}, which are pointwise bounds and do \emph{not} follow from finite 2D $\rho$-variation; they are verified for $\mathcal{K}$ in Lemma~\ref{lem:CF24_hypotheses} below.

\begin{lemma}[CF24 standing hypotheses for the fOU increment kernel]
\label{lem:CF24_hypotheses}
Fix $T > 0$ and $H \in (1/4, 1)$. There exists $C = C(H, \theta, \sigma, T) < \infty$ such that the increment kernel $\mathcal{K}$ of~\eqref{eq:K_definition} satisfies, for all $0 < s < t \le T$ and $\tau \in (0,T]$,
\begin{align*}
\emph{(i)}&\quad \bigl|\partial_s \partial_t \mathcal{K}(s,t)\bigr| \;=\; |R''(t-s)| \;\le\; C\,(t-s)^{2H-2},\\
\emph{(ii)}&\quad \bigl|\tfrac{1}{2}\tfrac{d}{dt}\mathcal{K}(t,t) - \partial_t \mathcal{K}(s,t)\bigr| \;=\; |R'(t-s)| \;\le\; C\,(t-s)^{2H-1},\\
\emph{(iii)}&\quad \bigl|\tfrac{d}{d\tau}\mathcal{K}(\tau,\tau)\bigr| \;=\; 2\,|R'(\tau)| \;\le\; C\,\tau^{2H-1},\\
\emph{(iv)}&\quad \Var(\tilde{X}_t - \tilde{X}_s) \;=\; 2\bigl(R(0) - R(t-s)\bigr) \;\le\; C\,(t-s)^{2H}.
\end{align*}
These are, respectively, the standing hypotheses (9), (10), (11), (12) of~\citet{CassFerrucci2024} (stated there for the covariance of the process to which the framework is applied; here that covariance is $\mathcal{K}$). Together with the finite 2D $\rho$-variation ($\rho < 2$) established above, this places $\tilde{X}|_{[0,T]}$ within the standing framework of~\citet{CassFerrucci2024}, with constants depending only on $(H, \theta, \sigma, T)$ and the level.
\end{lemma}

\begin{proof}
The identities are elementary calculus on $\mathcal{K}(s,t) = R(|t-s|) - R(s) - R(t) + R(0)$: for $s < t$, $\partial_t \mathcal{K}(s,t) = R'(t-s) - R'(t)$ and $\partial_s\partial_t \mathcal{K}(s,t) = -R''(t-s)$, while $\mathcal{K}(\tau,\tau) = 2(R(0) - R(\tau))$ gives $\tfrac{d}{d\tau}\mathcal{K}(\tau,\tau) = -2R'(\tau)$, whence $\tfrac12 \tfrac{d}{dt}\mathcal{K}(t,t) - \partial_t\mathcal{K}(s,t) = -R'(t-s)$.

\emph{Case $H \ne 1/2$.} Each envelope follows from continuity plus a finite limit at the origin, the proofs of Lemma~\ref{lem:fOU_cov} applying on the window $[0,T]$ with window-dependent constants. (The inputs --- Lemma~\ref{lem:fOU_cov} and the identity~\eqref{eq:Rpp_identity} --- are established independently of the present lemma: the former in Appendix~\ref{app:gamma}, the latter in Lemma~\ref{lem:R_tails} below by a self-contained spectral argument; no circularity arises.) The function $\tau \mapsto R'(\tau)\,\tau^{1-2H}$ is continuous on $(0,T]$ and tends to $-2H\kappa$ as $\tau \downarrow 0$ (since $R' = -2H\kappa\,\tau^{2H-1} + \psi'$ with $\psi'(\tau) = O(\tau)$, eq.~\eqref{eq:R_prime_local}), hence is bounded, proving \emph{(ii)}--\emph{(iii)}. By the identity~\eqref{eq:Rpp_identity}, $R''(\tau)\,\tau^{2-2H} = \theta^2 R(\tau)\,\tau^{2-2H} - \sigma^2 H(2H-1)$ is continuous on $(0,T]$ with limit $-\sigma^2 H(2H-1)$ at the origin, hence bounded, proving \emph{(i)}. Likewise $\tau \mapsto (R(0)-R(\tau))\,\tau^{-2H}$ is continuous on $(0,T]$ with limit $\kappa$ (Lemma~\ref{lem:fOU_cov}), proving \emph{(iv)}.

\emph{Case $H = 1/2$.} Here $R(\tau) = \tfrac{\sigma^2}{2\theta}\,e^{-\theta\tau}$, so off the diagonal $|R''| = \theta^2 R \le \theta^2 R(0) \le C\,\tau^{-1}$ on $(0,T]$, $|R'| = \tfrac{\sigma^2}{2}\,e^{-\theta\tau} \le \tfrac{\sigma^2}{2} = C\,\tau^{0}$, and $R(0) - R(\tau) \le \tfrac{\sigma^2}{2}\,\tau$, giving \emph{(i)}--\emph{(iv)} with the exponents $2H-2 = -1$, $2H-1 = 0$, $2H = 1$. The atomic diagonal component of the mixed-derivative bimeasure (Remark~\ref{rem:H_half_atom}) is permitted: \emph{(i)}--\emph{(ii)} are off-diagonal conditions.
\end{proof}

In particular, the increment variance is
\begin{equation}
\label{eq:increment_variance}
\Var(\tilde{X}_\Delta) = \mathcal{K}(\Delta,\Delta) = 2(R(0) - R(\Delta)) = 2v,
\end{equation}
where we define the \emph{increment variance parameter} $v := R(0) - R(\Delta)$.

Since signatures are translation-invariant in the starting point, the deterministic shift $X \mapsto \tilde{X}$ gives $S^{(M)}(\mathbf{X})_{s,t} = S^{(M)}(\tilde{\mathbf{X}})_{s,t}$ pathwise, so working with $\tilde{X}$ is equivalent to working with $X$ for signature purposes.

\subsection{Isonormal-Gaussian Setup}
\label{subsec:isonormal_setup}

Let $\mathcal{H}$ denote the Cameron--Martin space of $\tilde{X}$, that is, the reproducing-kernel Hilbert space with kernel $\mathcal{K}$ of~\eqref{eq:K_definition}, constructed as the completion of step functions supported on $[0, \infty)$ (equivalently, on any finite time window containing the blocks $I_k$ used in the sequel) under the inner product $\langle \mathbb{1}_{[0,s]}, \mathbb{1}_{[0,t]}\rangle_\mathcal{H} = \mathcal{K}(s, t)$~\citep[\S1.1]{Nualart2006}. Let $I_j^\mathcal{H} : \mathcal{H}^{\odot j} \to L^2(\Omega)$ denote the $j$-th multiple Wiener integral on symmetric tensors, extended to $\mathcal{H}^{\otimes j}$ by $I_j^\mathcal{H}(f) := I_j^\mathcal{H}(\tilde{f})$, where $\tilde{f}$ is the symmetrization of $f$. The isometry~\citep[eqs.~(20)--(21)]{CassFerrucci2024} then reads
\begin{equation}
\label{eq:wiener_isometry}
\mathbb{E}[I_j^{\mathcal{H}}(f)\,I_k^{\mathcal{H}}(g)] = \delta_{j,k}\,j!\,\langle \tilde{f}, \tilde{g}\rangle_{\mathcal{H}^{\otimes j}}.
\end{equation}
Here $\tilde{f} := (j!)^{-1}\sum_{\pi \in S_j} f^\pi$ denotes the symmetrization, $f^\pi$ the slot permutation of $f$. Two elementary consequences of~\eqref{eq:wiener_isometry} are used repeatedly. First, expanding the symmetrizations and using the invariance of the pairing under simultaneous permutation of both arguments,
\begin{equation}
\label{eq:perm_expansion}
j!\,\langle \tilde{f}, \tilde{g}\rangle_{\mathcal{H}^{\otimes j}} \;=\; \sum_{\pi \in S_j} \langle f, g^{\pi}\rangle_{\mathcal{H}^{\otimes j}},
\end{equation}
with the \emph{plain} (unsymmetrized) tensor-product pairing on the right; in particular $|\mathbb{E}[I_j^{\mathcal{H}}(f)\,I_j^{\mathcal{H}}(g)]| \le j!\,\max_{\pi \in S_j}|\langle f, g^\pi\rangle_{\mathcal{H}^{\otimes j}}|$. Second, a kernel $f$ supported on the ordered simplex (extended by zero off it) is related to its sorting-based symmetric extension $\breve{f}$ of Lemma~\ref{lem:symm_ext_Linf} by $\breve{f} = j!\,\tilde{f}$ a.e.\ (off the null set of coordinate ties, exactly one permutation sorts each point into the simplex); the sorted extension is therefore used below only for sup-norm bookkeeping --- it shares the simplex essential supremum --- while chaos identities are stated with $\tilde{f}$ or with plain pairings. Only upper bounds are consumed downstream, and $(j!)^{\pm 1}$ discrepancies are absorbed into level-dependent constants. This is the natural ambient Gaussian space of~\citet[eqs.~(13)--(17)]{CassFerrucci2024} for $\tilde{X}$.

\subsection{Uniform Tail Bounds for $R$ and $R''$}
\label{subsec:R_tails}

\begin{lemma}[Restatement of Lemma~\ref{lem:gaussian_inputs}\emph{(a)} --- Uniform tail bounds]
\label{lem:R_tails}
For the stationary fOU autocovariance $R$ of~\eqref{eq:fOU_spectral} with $H \in (1/4, 1)$, $H \neq 1/2$, there exists a constant $C_R = C_R(H, \theta, \sigma, \Delta) < \infty$ such that
\begin{equation}
\label{eq:R_uniform_tail}
|R(\tau)| \;\le\; C_R\,\tau^{2H-2}, \qquad |R''(\tau)| \;\le\; C_R\,\tau^{2H-4} \quad \text{for all } \tau \ge \Delta.
\end{equation}
\end{lemma}

\begin{proof}
By~\citet[Thm.~2.3, applied with $N = 2$]{CheriditoKawaguchiMaejima2003}, the stationary fOU autocovariance admits the Poincar\'e asymptotic expansion
\begin{equation}
\label{eq:R_poincare}
R(\tau) \;=\; \kappa_{H,1}\,\tau^{2H-2} \;+\; \kappa_{H,2}\,\tau^{2H-4} \;+\; O(\tau^{2H-6}) \qquad (\tau \to \infty),
\end{equation}
with explicit constants $\kappa_{H,1} = \tfrac{1}{2}\sigma^2\theta^{-2}(2H)(2H-1)$ and $\kappa_{H,2} = \tfrac{1}{2}\sigma^2\theta^{-4}(2H)(2H-1)(2H-2)(2H-3)$. Hence there exists $T_0 = T_0(H,\theta,\sigma)$ such that $|R(\tau)| \le C\,\tau^{2H-2}$ for $\tau \ge T_0$. Continuity of $R$ on the compact interval $[\Delta, T_0]$ and positivity of $\tau \mapsto \tau^{2H-2}$ there absorb $\sup_{[\Delta, T_0]}|R(\tau)|/\tau^{2H-2}$ into the constant, yielding the $R$-part of~\eqref{eq:R_uniform_tail}.

For $R''$, write $f(\lambda) = \lambda^{1-2H}/(\theta^2+\lambda^2)$ in~\eqref{eq:fOU_spectral}. Since $f(\lambda)\sin(\lambda\tau) \to 0$ as $\lambda \to \infty$, while at $\lambda = 0$ one has $f(\lambda)\sin(\lambda\tau) \sim \theta^{-2}\tau\,\lambda^{2-2H} \to 0$ for every $H<1$, an integration by parts (IBP) in $\lambda$ gives
\[
R(\tau) = -\frac{2\zeta}{\tau}\int_0^\infty \sin(\lambda\tau)\,f'(\lambda)\,d\lambda, \qquad \tau>0.
\]
Here $f'(\lambda) \sim (1-2H)\theta^{-2}\lambda^{-2H}$ as $\lambda \downarrow 0$, whereas $f'(\lambda)=O(\lambda^{-2-2H})$ as $\lambda\to\infty$. At infinity the post-integration-by-parts integrand $f'(\lambda)\sin(\lambda\tau)$ is therefore absolutely integrable. At the origin, $f'(\lambda)$ alone fails to be absolutely integrable for $H > 1/2$, but the sine factor rescues integrability via $\sin(\lambda\tau)f'(\lambda) \sim (1-2H)\theta^{-2}\tau\,\lambda^{1-2H}$, which is integrable for every $H < 1$. Further iteration of this IBP argument requires care because the second-derivative boundary term $[f'(\lambda)\cos(\lambda\tau)]_0^\infty$ no longer vanishes at $\lambda = 0$ (where $f'(\lambda) \sim (1-2H)\theta^{-2}\lambda^{-2H}$ diverges). For $\tau > 0$, however, smoothness can be obtained directly: split $R(\tau) = R_\Lambda(\tau) + R^\Lambda(\tau)$ at any fixed $\Lambda > 0$, with $R_\Lambda(\tau) := 2\zeta\int_0^\Lambda \cos(\lambda\tau) f(\lambda)\,d\lambda$ and $R^\Lambda(\tau)$ the tail. Since $\lambda^k f(\lambda)$ is integrable on $[0, \Lambda]$ for every $k$ (as $f(\lambda) \sim \theta^{-2}\lambda^{1-2H}$ near 0 with $1-2H > -1$ for $H < 1$), differentiation under the integral gives $R_\Lambda \in C^\infty(\mathbb{R})$. Repeated nonstationary-phase IBP on $[\Lambda, \infty)$ (where each IBP step introduces a $\tau^{-1}$ factor and the relevant $f^{(k)}$ behaves as $O(\lambda^{-1-2H-k})$ at infinity, while the finite-endpoint boundary term at $\lambda = \Lambda$ is a smooth function of $\tau$ on $[\tau_0, \infty)$ --- as $f$ and its derivatives are finite at the fixed $\Lambda > 0$ --- hence harmless) shows $R^\Lambda \in C^\infty([\tau_0, \infty))$ for any $\tau_0 > 0$; hence $R \in C^\infty((0,\infty))$.

Since $R \in C^\infty((0,\infty))$ by the preceding paragraph, $R''$ exists classically on $(0,\infty)$. We claim the exact second-order identity
\begin{equation}
\label{eq:Rpp_identity}
R''(\tau) \;=\; \theta^2 R(\tau) \;-\; \sigma^2 H(2H-1)\,\tau^{2H-2}, \qquad \tau > 0.
\end{equation}
(At $H = 1/2$ the second term vanishes and~\eqref{eq:Rpp_identity} reduces to $R'' = \theta^2 R$, recovering the exponential OU covariance.) To prove~\eqref{eq:Rpp_identity}, fix $\phi \in C_c^\infty((0,\infty))$ and set $\widehat{\phi}_c(\lambda) := \int_0^\infty \phi(\tau)\cos(\lambda\tau)\,d\tau$, which decays faster than any polynomial. Two integrations by parts give $\int_0^\infty \phi''(\tau)\cos(\lambda\tau)\,d\tau = -\lambda^2\widehat{\phi}_c(\lambda)$, so Fubini's theorem (justified since $f \in L^1(0,\infty)$ and $\phi, \phi'' \in L^1$), the spectral representation~\eqref{eq:fOU_spectral}, and the algebraic identity $(\theta^2+\lambda^2)f(\lambda) = \lambda^{1-2H}$ give
\[
\int_0^\infty \bigl(\theta^2 R(\tau) - R''(\tau)\bigr)\phi(\tau)\,d\tau
= \int_0^\infty R(\tau)\bigl(\theta^2\phi(\tau) - \phi''(\tau)\bigr)\,d\tau
= 2\zeta \int_0^\infty \lambda^{1-2H}\,\widehat{\phi}_c(\lambda)\,d\lambda,
\]
where the first equality moves the derivatives onto $R$ ($\phi$ is compactly supported in $(0,\infty)$, where $R$ is smooth). The right-hand side is evaluated by Abelian regularization. For $\alpha > 0$ and $s := 2-2H \in (0, 3/2)$, the absolutely convergent Gamma integral $\int_0^\infty \lambda^{s-1} e^{-(\alpha - i\tau)\lambda}\,d\lambda = \Gamma(s)\,(\alpha - i\tau)^{-s}$ (principal branch) yields, upon taking real parts,
\[
\int_0^\infty \lambda^{1-2H} e^{-\alpha\lambda}\cos(\lambda\tau)\,d\lambda
= \Gamma(2-2H)\,(\alpha^2+\tau^2)^{-(1-H)}\cos\bigl((2-2H)\arctan(\tau/\alpha)\bigr).
\]
Inserting $e^{-\alpha\lambda}$ into the $\lambda$-integral (dominated convergence, using $\lambda^{1-2H}|\widehat{\phi}_c| \in L^1$), exchanging the order of integration at fixed $\alpha > 0$ (absolute convergence), and letting $\alpha \downarrow 0$ (dominated convergence on $\operatorname{supp}\phi \subset (0,\infty)$, where $\arctan(\tau/\alpha) \to \pi/2$) gives
\[
\int_0^\infty \lambda^{1-2H}\,\widehat{\phi}_c(\lambda)\,d\lambda
= \Gamma(2-2H)\cos\bigl(\pi(1-H)\bigr)\int_0^\infty \tau^{2H-2}\,\phi(\tau)\,d\tau.
\]
Since $\cos(\pi(1-H)) = -\cos(\pi H)$, $2\zeta = \sigma^2\Gamma(2H+1)\sin(\pi H)/\pi$, and the reflection formula gives $\Gamma(2-2H)\Gamma(2H-1) = \pi/\sin(\pi(2H-1)) = -\pi/\sin(2\pi H)$,
\begin{align*}
2\zeta\,\Gamma(2-2H)\cos\bigl(\pi(1-H)\bigr)
&= -\frac{\sigma^2\Gamma(2H+1)\sin(\pi H)\cos(\pi H)}{\pi}\cdot\Bigl(-\frac{\pi}{\sin(2\pi H)\,\Gamma(2H-1)}\Bigr) \\
&= \frac{\sigma^2\,\Gamma(2H+1)}{2\,\Gamma(2H-1)}
= \sigma^2 H(2H-1).
\end{align*}
Thus both sides of~\eqref{eq:Rpp_identity} have the same pairing with every $\phi \in C_c^\infty((0,\infty))$; both are continuous on $(0,\infty)$, so they coincide pointwise.

The tail bound is now immediate: $\kappa_{H,1} = \tfrac{1}{2}\sigma^2\theta^{-2}(2H)(2H-1) = \sigma^2 H(2H-1)/\theta^2$, so~\eqref{eq:Rpp_identity} and the expansion~\eqref{eq:R_poincare} give the exact leading-order cancellation
\begin{equation}
\label{eq:Rpp_poincare}
R''(\tau) \;=\; \theta^2\bigl(R(\tau) - \kappa_{H,1}\,\tau^{2H-2}\bigr) \;=\; \theta^2\kappa_{H,2}\,\tau^{2H-4} \;+\; O(\tau^{2H-6}) \qquad (\tau \to \infty).
\end{equation}
Hence there exists $T_1 = T_1(H,\theta,\sigma)$ with $|R''(\tau)| \le C\,\tau^{2H-4}$ for $\tau \ge T_1$; continuity of $R''$ on the compact interval $[\Delta,\max(T_1,T_0)]$ absorbs the remaining range, completing the proof.
\end{proof}

\subsection{Chaos Decomposition for Block Signatures}
\label{subsec:chaos_block}

Throughout this subsection, fix a truncation level $M \ge \lfloor p \rfloor$ (where $p > \max(2, 1/H)$ is the rough-path parameter of Proposition~\ref{prop:fOU_assumptions}) and a word $w = (w_1, \ldots, w_m)$ of length $|w| = m \le M$. We locally redefine $I_k := [k\Delta, (k+1)\Delta]$ with left endpoint $s_k := k\Delta$, so that $I_0 = [0, \Delta]$ and the separation between $I_0$ and $I_h$ is $(h-1)\Delta$ for $h \ge 1$. Let $Z_k^{(m)} \in (\mathbb{R}^d)^{\otimes m}$ denote the level-$m$ component of the block signature $S^{(M)}(\tilde{\mathbf{X}})_{s_k, s_k + \Delta}$, and let $c_{h,m}(w) := \Cov(\langle w, Z_0^{(m)}\rangle, \langle w, Z_h^{(m)}\rangle)$.

\begin{lemma}[Restatement of Lemma~\ref{lem:gaussian_inputs}\emph{(b)} --- Chaos expansion of block signature]
\label{lem:chaos_expansion_block}
Under the assumptions of Proposition~\ref{prop:fOU_assumptions}, $\tilde{\mathbf{X}}$ admits a Gaussian rough-path lift in $G^{\lfloor p \rfloor}(\mathbb{R}^d)$ with Fernique-type moments by~\citet[Thm.~15.33]{FrizVictoir2010}, and the Lyons extension theorem~\citep[Thm.~9.5]{FrizVictoir2010} determines $\langle w, Z_k^{(m)}\rangle \in L^2(\Omega)$. Applying~\citet[Thm.~2.3, Prop.~3.9]{CassFerrucci2024} to each signature coordinate yields the Wiener chaos expansion
\begin{equation}
\label{eq:chaos_expansion_block}
\langle w, Z_k^{(m)}\rangle \;=\; \sum_{\substack{0 \le j \le m\\ j \equiv m\,(2)}} I_j^{\mathcal{H}}(f_{k,w,j}^{(m)}),
\end{equation}
where $f_{k,w,j}^{(m)} \in \mathcal{H}^{\otimes j}$ is the partial-contraction kernel constructed in~\citet[Def.~3.3]{CassFerrucci2024}, supported on the ordered simplex of $I_k^j$ and extended by zero to $I_k^j$. The constants in all kernel bounds below are independent of the block index $k$. Centering $Z_k^{(m)}$ eliminates the $j=0$ (deterministic) contribution, so the isometry~\eqref{eq:wiener_isometry} gives
\begin{equation}
\label{eq:chm_chaos_sum}
c_{h,m}(w) \;=\; \sum_{\substack{1 \le j \le m\\ j \equiv m\,(2)}} j!\,\bigl\langle \tilde{f}_{0,w,j}^{(m)}, \tilde{f}_{h,w,j}^{(m)}\bigr\rangle_{\mathcal{H}^{\otimes j}}
\;=\; \sum_{\substack{1 \le j \le m\\ j \equiv m\,(2)}} \sum_{\pi \in S_j} \bigl\langle f_{0,w,j}^{(m)}, (f_{h,w,j}^{(m)})^{\pi}\bigr\rangle_{\mathcal{H}^{\otimes j}},
\end{equation}
the second form by the permutation expansion~\eqref{eq:perm_expansion}.
\end{lemma}

\begin{proof}
The application of~\citet{CassFerrucci2024} is blockwise, which is what makes the constants $k$-uniform. For each $k$, consider the re-based increment process $\tilde{X}^{(k)}_t := X_{s_k + t} - X_{s_k}$, $t \in [0, \Delta]$. By stationarity (Assumption~\ref{ass:S}), $\tilde{X}^{(k)} \stackrel{d}{=} \tilde{X}|_{[0,\Delta]}$, and its increment covariance is $\mathcal{K}|_{[0,\Delta]^2}$ for \emph{every} $k$; Lemma~\ref{lem:CF24_hypotheses} (with $T = \Delta$) verifies the CF24 standing hypotheses for it. Hence~\citet[Thm.~2.3, Prop.~3.9]{CassFerrucci2024} applies to $\tilde{X}^{(k)}$ on $[0,\Delta]$ and, by translation invariance of the signature, expands $\langle w, Z_k^{(m)}\rangle = \langle w, S^{(m)}(\tilde{\mathbf{X}}^{(k)})_{0,\Delta}\rangle$ with kernels on $[0,\Delta]^j$ and constants independent of $k$. The expansion is transported into the ambient space $\mathcal{H}$ by the shift embedding $\iota_k$ determined by $\mathbb{1}_{[0,t]} \mapsto \mathbb{1}_{(s_k,\, s_k + t]}$, which is isometric --- increment covariances are shift-invariant by stationarity, so $\langle \mathbb{1}_{(s_k, s_k+t]}, \mathbb{1}_{(s_k, s_k+t']}\rangle_{\mathcal{H}} = \mathcal{K}(t,t')$ --- and intertwines multiple Wiener integrals, $I_j^{(k)}(f) = I_j^{\mathcal{H}}(\iota_k^{\otimes j} f)$~\citep[\S1.1]{Nualart2006}. Under $\iota_k^{\otimes j}$ the block-$0$ kernel maps to its $s_k$-shift: this is the kernel $f_{k,w,j}^{(m)}$ supported on $I_k^j$, and in particular $\|f_{k,w,j}^{(m)}\|_{L^\infty}$ and the CF24 constants below carry no $k$-dependence. The displayed covariance identity is then the isometry~\eqref{eq:wiener_isometry} applied to~\eqref{eq:chaos_expansion_block} after centering.
\end{proof}

\begin{lemma}[Restatement of Lemma~\ref{lem:gaussian_inputs}\emph{(c)} --- Symmetric-extension $L^\infty$ bound]
\label{lem:symm_ext_Linf}
Let $\Delta^j[s,t] := \{(v_1, \ldots, v_j) \in [s,t]^j : v_1 \le \cdots \le v_j\}$ denote the ordered simplex, and let $\breve{f}_{k,w,j}^{(m)}$ denote the sorting-based symmetric extension of $f_{k,w,j}^{(m)}$ to the full box $I_k^j = [s_k, s_k + \Delta]^j$, defined by $\breve{f}_{k,w,j}^{(m)}(v_1, \ldots, v_j) := f_{k,w,j}^{(m)}(v_{\pi(1)}, \ldots, v_{\pi(j)})$, where $\pi$ is the unique permutation sorting $(v_1, \ldots, v_j)$ in increasing order. Then $\breve{f}_{k,w,j}^{(m)} = j!\,\tilde{f}_{k,w,j}^{(m)}$ a.e.\ (Appendix~\ref{subsec:isonormal_setup}), and
\begin{equation}
\label{eq:symm_ext_Linf}
\|\breve{f}_{k,w,j}^{(m)}\|_{L^\infty(I_k^j)} \;=\; \|f_{k,w,j}^{(m)}\|_{L^\infty(\Delta^j[s_k, s_k+\Delta])} \;\le\; C_{\mathrm{CF}}\,\Delta^{(m-j)H},
\end{equation}
where $C_{\mathrm{CF}} = C_{\mathrm{CF}}(H, \theta, \sigma, \Delta, m) < \infty$ is independent of the block index $k$ (Lemma~\ref{lem:chaos_expansion_block}). In particular $\|\tilde{f}_{k,w,j}^{(m)}\|_{L^\infty(I_k^j)} \le \|\breve{f}_{k,w,j}^{(m)}\|_{L^\infty(I_k^j)}$, so the same bound holds for the symmetrization.
\end{lemma}

\begin{proof}
By~\citet[Prop.~3.1]{CassFerrucci2024}, the uniform bound $|P|_{st} \lesssim (t-s)^{(n-m)H}$ holds over the free-variable simplex $\Delta^m[s,t]$ for every pairing-integral $P \in \mathcal{P}^n_m$; with $(n,m)_{\mathrm{CF24}} \leftrightarrow (m,j)_{\mathrm{paper}}$ and $(t-s) = \Delta$, this gives the simplex bound $|P|_{s_k, s_k + \Delta} \lesssim \Delta^{(m-j)H}$ on $\Delta^j[s_k, s_k+\Delta]$. By the `Moreover' part of~\citet[Prop.~3.9]{CassFerrucci2024}, $|P_{st}| \lesssim |P|_{st}$ on $\Delta^j[s,t]$, the right-hand side being the absolute-value integral bounded in~\citet[Prop.~3.1]{CassFerrucci2024}; hence the limiting kernel $f_{k,w,j}^{(m)}$ satisfies the simplex essential-supremum bound $\|f_{k,w,j}^{(m)}\|_{L^\infty(\Delta^j[s_k, s_k+\Delta])} \le C_{\mathrm{CF}}\,\Delta^{(m-j)H}$; the constants are $k$-independent by the blockwise application and shift embedding of Lemma~\ref{lem:chaos_expansion_block}. The sorting permutation $\pi$ is uniquely defined on $I_k^j \setminus E$, where $E := \{v \in I_k^j : v_i = v_{i'} \text{ for some } i \ne i'\}$ has Lebesgue measure zero and is therefore irrelevant for the essential supremum. Sorting-based extension preserves the essential supremum, because every point in $I_k^j \setminus E$ corresponds under some permutation to a point in the simplex: $\|\breve{f}_{k,w,j}^{(m)}\|_{L^\infty(I_k^j)} = \|f_{k,w,j}^{(m)}\|_{L^\infty(\Delta^j[s_k, s_k+\Delta])}$. The relation $\breve{f} = j!\,\tilde{f}$ a.e.\ holds because, for $f$ supported on the simplex and extended by zero, exactly one term of the sum $\sum_{\pi} f^\pi$ is non-zero at each point of $I_k^j \setminus E$, namely the sorting permutation; the bound for the symmetrization follows since $j!\,\tilde{f} = \breve{f}$.
\end{proof}

\begin{lemma}[Restatement of Lemma~\ref{lem:gaussian_inputs}\emph{(d)} --- Cameron--Martin integral representation on separated supports]
\label{lem:CM_integral_rep}
Assume the tail bound~\eqref{eq:R_uniform_tail} of Lemma~\ref{lem:R_tails} and the conventions of Lemma~\ref{lem:chaos_expansion_block}. For $h \ge 2$, the partial-contraction kernels of~\citet[Def.~3.3]{CassFerrucci2024} satisfy, in the \emph{plain} (unsymmetrized) tensor-product pairing and for every slot permutation $\pi \in S_j$,
\begin{equation}
\label{eq:CM_integral_rep}
\bigl\langle f_{0,w,j}^{(m)}, (f_{h,w,j}^{(m)})^{\pi}\bigr\rangle_{\mathcal{H}^{\otimes j}}
\;=\; (-1)^j \int_{I_0^j \times I_h^j} f_{0,w,j}^{(m)}(u)\, (f_{h,w,j}^{(m)})^{\pi}(v) \prod_{\ell=1}^{j} R''(v_\ell - u_\ell)\,du\,dv.
\end{equation}
The resulting envelope bounds below depend only on $L^\infty$-norms and the separated supports, both permutation-invariant, so they hold uniformly over $\pi$.
\end{lemma}

\begin{proof}
For $u \in I_0$ and $v \in I_h$ with $h \ge 2$, one has $v - u \ge (h-1)\Delta > 0$, so the cross term $R(v-u)$ of the covariance kernel $\mathcal{K}(u,v) = R(v-u) - R(u) - R(v) + R(0)$ is smooth on this cross-block region; the single-argument terms carry no joint dependence, so $\mathcal{K}$ has there the mixed partial density
\begin{equation}
\label{eq:bimeasure_density}
\partial_u \partial_v \mathcal{K}(u, v) \;=\; -R''(v-u), \qquad u \in I_0,\; v \in I_h.
\end{equation}
For subintervals $A = (a, b] \subseteq I_0$ and $B = (c, d] \subseteq I_h$ with $b < c$, the definition of the Cameron--Martin inner product via covariance increments gives
\[
\bigl\langle \mathbb{1}_A, \mathbb{1}_B\bigr\rangle_{\mathcal{H}} \;=\; \mathcal{K}(b,d) - \mathcal{K}(a,d) - \mathcal{K}(b,c) + \mathcal{K}(a,c);
\]
the single-argument terms of $\mathcal{K}$ cancel in this rectangle increment, which therefore equals the rectangle increment of $R(v-u)$; since the cross term is $C^2$ on $\{v - u \ge (h-1)\Delta\}$, the fundamental theorem of calculus and~\eqref{eq:bimeasure_density} yield
\[
\bigl\langle \mathbb{1}_A, \mathbb{1}_B\bigr\rangle_{\mathcal{H}} \;=\; -\int_a^b\!\!\int_c^d R''(v-u)\,dv\,du.
\]
Extending by linearity, the same identity holds for all step functions $\phi, \psi$ supported in $I_0, I_h$ respectively:
\[
\langle \phi, \psi\rangle_{\mathcal{H}} \;=\; -\int_{I_0 \times I_h} \phi(u)\,\psi(v)\,R''(v-u)\,du\,dv.
\]

Let $\Delta_h := (h-1)\Delta > 0$. Since $v-u \ge \Delta_h$ on $I_0 \times I_h$, Lemma~\ref{lem:R_tails} gives $\|R''\|_{L^\infty(I_h-I_0)} \le C_R\,\Delta_h^{2H-4}$. The bilinear form $B(\phi,\psi) := -\int_{I_0\times I_h}\phi(u)\psi(v)R''(v-u)\,du\,dv$ is therefore continuous on $L^2(I_0)\times L^2(I_h)$, with
\[
|B(\phi,\psi)| \;\le\; C_R\,\Delta_h^{2H-4}\,|I_0\times I_h|^{1/2}\,\|\phi\|_{L^2(I_0)}\|\psi\|_{L^2(I_h)}.
\]
The displayed identity extends from step functions to the CF24 approximant kernels by bilinearity and Fubini: by~\citet[Def.~3.3]{CassFerrucci2024}, each approximant kernel is an integrable mixture of indicator (step-function) tensors, and on the separated supports the bimeasure has the bounded density $-R''(v-u)$, so the step-function identity integrates over the mixture. (We do not invoke abstract $\mathcal{H}$-density: convergence in $\mathcal{H}$ alone would not interact with the $L^2$-continuity of $B$.) The passage to the limiting kernels then proceeds through dominated convergence on the finite-measure box $I_0^j \times I_h^j$ (Lemma~\ref{lem:symm_ext_Linf}), as follows. Applying this coordinatewise on simple tensors and passing to the limit in $\mathcal{H}^{\otimes j}$ for the CF24 kernels $f_{k,w,j}^{(m)}$ (supported on $I_k^j$ by~\citealp[Def.~3.3]{CassFerrucci2024}) reduces to an $L^2$-limit on separated supports. By~\citet[Prop.~3.9]{CassFerrucci2024}, the Def.~3.3 approximants $P^\ell_{st}$ converge a.e.\ on the simplex $\Delta^j[s,t]$ to $P_{st}$, and by~\citet[Lem.~3.8]{CassFerrucci2024} (dominating function) they are bounded on the simplex uniformly in $\ell$: $\sup_\ell|P^\ell_{st}| \le C'_{\mathrm{CF}}(H,\theta,\sigma,\Delta,m,j) < \infty$ (both results applicable by Lemma~\ref{lem:CF24_hypotheses}). Both statements extend from the simplex to the box $I_k^j$. The free-variable factors of~\citet[Def.~3.3]{CassFerrucci2024} are $\varrho_\ell^{-1}$-normalized indicator windows of width at most $\varrho_\ell$, so each approximant vanishes unless every free argument lies within $\varrho_\ell$ of the integration simplex: off the closed simplex, almost every point of $I_k^j$ eventually exits all supports, whence the approximants converge to zero there --- matching the extension of the limiting kernel by zero --- and a.e.\ convergence holds on the whole box. The dominating-function bound is likewise box-wide: an approximant depends on its free arguments only through the partition cells containing them, and every cell configuration on which it is non-zero is realized at a point of the ordered simplex (the free coordinates of any ordered tuple in the support of its integrand), so the supremum over $I_k^j$ coincides with the supremum over the simplex. The approximant kernels $f^{\,\ell}_{0,w,j}$ and $f^{\,\ell}_{h,w,j}$ therefore satisfy $\|f^{\,\ell}_{k,w,j}\|_{L^\infty(I_k^j)} \le C'_{\mathrm{CF}}$ uniformly in $\ell$, and the same bound holds for every slot permutation, whose application changes neither the $L^\infty$-norm nor the mixture-of-indicators structure. Combined with $|R''(v_r-u_r)| \le C_R((h-1)\Delta)^{2H-4}$ from Lemma~\ref{lem:R_tails} on the rectangle $I_0 \times I_h$, this yields the uniform dominator
\[
\bigl|f^{\,\ell}_{0,w,j}(u)\,(f^{\,\ell}_{h,w,j})^{\pi}(v)\,\prod_{r=1}^j R''(v_r-u_r)\bigr|
\;\le\;
(C'_{\mathrm{CF}})^2\,\bigl(C_R((h-1)\Delta)^{2H-4}\bigr)^j
\]
on the finite-measure box $I_0^j \times I_h^j$, uniformly in $\pi \in S_j$. Dominated convergence therefore passes the coordinatewise bilinear form to the limit and yields~\eqref{eq:CM_integral_rep}; for the limiting kernels, the left-hand side is thus understood through this $L^2$-limit on the separated supports, where the plain pairing is realized by the $L^2$-continuous form $B$ applied coordinatewise. The sequel in fact consumes only the aggregate of~\eqref{eq:CM_integral_rep} over $\pi \in S_j$, for which the identification is unambiguous: by~\eqref{eq:perm_expansion} applied at approximant level and the $L^2(\Omega)$-convergence of the chaos components~\citep[Thm.~2.3, Prop.~3.9]{CassFerrucci2024}, the sum $\sum_{\pi \in S_j}\bigl\langle f_{0,w,j}^{(m)}, (f_{h,w,j}^{(m)})^{\pi}\bigr\rangle_{\mathcal{H}^{\otimes j}}$ equals $j!\,\bigl\langle \tilde{f}_{0,w,j}^{(m)}, \tilde{f}_{h,w,j}^{(m)}\bigr\rangle_{\mathcal{H}^{\otimes j}}$, the quantity appearing in~\eqref{eq:chm_chaos_sum}.
\end{proof}

\subsection{Level 2: Exact Formula}
\label{subsec:level2_exact}

\begin{proposition}[Level-2 expected signature for fOU]
\label{prop:level2}
For the centered fOU increment process $\tilde{X}$ with independent coordinates, the level-2 expected signature is well-defined for all $H \in (1/4, 1)$:
\begin{equation}
\label{eq:level2_formula}
\mathbb{E}[S^{(2)}_{ij}] = \begin{cases}
v_i := R^i(0) - R^i(\Delta) & \text{if } i = j, \\
0 & \text{if } i \neq j,
\end{cases}
\end{equation}
where $v_i$ is the \emph{increment variance parameter} for coordinate $i$, computed via
\begin{equation}
\label{eq:vi_spectral}
v_i = R^i(0) - R^i(\Delta) = 2\zeta_i \int_0^\infty (1 - \cos(\lambda \Delta)) \, f_i(\lambda) \, d\lambda.
\end{equation}
\end{proposition}

\begin{proof}
Since $\tilde{\mathbf{X}}$ is a geometric rough path, $\operatorname{Sym}(S^{(2)}(\tilde{\mathbf{X}})_{0,\Delta}) = \tfrac{1}{2}(\tilde{X}_\Delta)^{\otimes 2}$ \citep[eq.~(2.6)]{FrizHairer2020}, so $S^{(2)}_{ii} = (\tilde{X}^i_\Delta)^2/2$ and $\mathbb{E}[S^{(2)}_{ii}] = \Var(\tilde{X}^i_\Delta)/2 = v_i$ by~\eqref{eq:increment_variance}. For $i \neq j$, the symmetric part $S^{(2)}_{ij} + S^{(2)}_{ji} = \tilde{X}^i_\Delta\, \tilde{X}^j_\Delta$ has vanishing expectation by independence of coordinates; the antisymmetric L\'evy-area part $\tfrac{1}{2}(S^{(2)}_{ij} - S^{(2)}_{ji})$ changes sign under $\tilde{X}^j \mapsto -\tilde{X}^j$, a law-preserving reflection on coordinate~$j$ (valid since $\tilde{X}^j$ is centered Gaussian and independent of $\tilde{X}^i$), so its expectation also vanishes. Hence $\mathbb{E}[S^{(2)}_{ij}] = 0$ for $i \neq j$.
\end{proof}

\subsection{Level 4: Wick's Theorem and Pairing Structure}
\label{subsec:level4_wick}

Since the centered increment process $\tilde{X}$ is Gaussian with mean zero (Appendix~\ref{subsec:centered_increment}), the distributional symmetry $\tilde{X} \stackrel{d}{=} -\tilde{X}$ implies $S^{(m)}(\tilde{\mathbf{X}}) \stackrel{d}{=} (-1)^m S^{(m)}(\tilde{\mathbf{X}})$, so odd signature levels vanish in expectation. Level 4 is computed via Wick's theorem.

\begin{theorem}[Wick's theorem for signature]
\label{thm:wick}
For the centered fOU increment process $\tilde{X}$ of Appendix~\ref{subsec:centered_increment} with independent coordinates and $H > 1/4$, let $\mathcal{P}_{ijkl}$ denote the set of valid pairings, i.e., partitions of $\{1,2,3,4\}$ into pairs where the indices match within each pair.
\begin{enumerate}
\item[\textup{(a)}] If $H > 1/2$, then
\begin{equation}
\label{eq:level4_wick}
\mathbb{E}[S^{(4)}_{ijkl}]
=
\int_{0 < t_1 < t_2 < t_3 < t_4 < \Delta}
\sum_{\pi \in \mathcal{P}_{ijkl}}
\prod_{(a,b)\in\pi}
\bigl(-\delta_{i_a,i_b} R''_{i_a}(t_b-t_a)\bigr)\,dt_1\,dt_2\,dt_3\,dt_4,
\end{equation}
as a classical Lebesgue integral.
\item[\textup{(b)}] If $H \in (1/4, 1/2]$, then
\[
\mathbb{E}[S^{(4)}_{ijkl}] = \sum_{\pi \in \mathcal{P}_{ijkl}} P_\pi^{(ijkl)},
\]
where each $P_\pi^{(ijkl)}$ is defined via the Cass--Ferrucci Wiener chaos expansion and the associated 2D Young integral of the covariance kernel $\mathcal{K}$, as in~\citep[Thm.~2.3, Def.~3.3]{CassFerrucci2024}.
\end{enumerate}
In case~\textup{(a)}, $\mathcal{K}_{i_a}(t_a,t_b) = R^{i_a}(|t_a-t_b|) - R^{i_a}(t_a) - R^{i_a}(t_b) + R^{i_a}(0)$ has mixed derivative $-{R^{i_a}}''(t_b-t_a)$ for $t_a < t_b$.
\end{theorem}

\begin{proof}
By~\citet[Thm.~2.3]{CassFerrucci2024} (applicable for every $H \in (1/4,1)$, including $H = 1/2$, by Lemma~\ref{lem:CF24_hypotheses}), $S^{(4)}_{ijkl}$ admits a finite Wiener chaos expansion of order $\le 4$; taking expectations retains only the zeroth chaos, whose pairing structure is given by~\citet[Def.~3.3, Prop.~3.9]{CassFerrucci2024}. In case~\textup{(a)}, $R''$ is locally integrable on the 4D simplex, so each pairing evaluates as a classical Lebesgue integral against the mixed derivative $-R''$ of $\mathcal{K}$. This is the Isserlis--Wick formula~\citep[Thm.~1.28]{Janson1997} applied to the Gaussian increments $dX$. In case~\textup{(b)}, the same pairings are realized via the regularized 2D Young integral of $\mathcal{K}$ from~\citet[Def.~3.3, Prop.~3.9]{CassFerrucci2024}.
\end{proof}

For a word $(i, j, k, l)$, the three possible pairings are:
\begin{enumerate}
\item[(P1)] \textbf{Adjacent pairing:} $(t_1, t_2)$ and $(t_3, t_4)$. Valid when $i = j$ and $k = l$.
\item[(P2)] \textbf{Alternating pairing:} $(t_1, t_3)$ and $(t_2, t_4)$. Valid when $i = k$ and $j = l$.
\item[(P3)] \textbf{Nested pairing:} $(t_1, t_4)$ and $(t_2, t_3)$. Valid when $i = l$ and $j = k$.
\end{enumerate}

\subsection{Regularized Formulas for Stationary fOU}
\label{subsec:regularization}

Each pairing contributes an integral involving $R''(\tau) \sim \tau^{2H-2}$, which is not locally integrable near $\tau = 0$ for $H < 1/2$. Accordingly, the 4D simplex integrals are ill-defined in the classical sense. The regularization framework of~\citet[Thm.~2.3, Def.~3.3, Prop.~3.9]{CassFerrucci2024} defines the expected signature via Wiener chaos expansion and the 2D Young integral of the covariance kernel $\mathcal{K}$, yielding well-defined values for $H > 1/4$.

The central result of this subsection is that the closed-form $R'$ formulas obtained for $H > 1/2$ by classical integration by parts remain valid \emph{verbatim} on the whole range $H \in (1/4, 1)$. The mechanism is a smooth stationary approximation: mollified versions of $X$ are stationary with smooth \emph{even} autocovariances $R_\varepsilon$, so $R_\varepsilon'(0) = 0$ and the integration by parts producing the $R'$ formulas has no boundary terms at any $\varepsilon > 0$; the formulas then pass to the limit by dominated convergence. Remark~\ref{rem:H_half_atom} explains why, by contrast, a classical computation against the pointwise density $-R''$ with boundary terms retained \emph{fails} at $H = 1/2$, even though $R''$ is integrable there.

\begin{lemma}[Smooth stationary approximation]
\label{lem:mollifier}
Let $\varphi$ be an even probability density in $C_c^\infty(\mathbb{R})$ with $\operatorname{supp}\varphi \subseteq [-1,1]$ (without loss of generality), $\varphi_\varepsilon := \varepsilon^{-1}\varphi(\cdot/\varepsilon)$, and let $X^\varepsilon_t := \int_{\mathbb{R}} X_{t-s}\,\varphi_\varepsilon(s)\,ds$ denote the coordinatewise mollification of the stationary fOU process (which is defined on the whole line). Then, for every $H \in (1/4,1)$ and $\varepsilon \in (0, 1/2]$:
\begin{enumerate}
\item[\emph{(i)}] $X^\varepsilon$ is a stationary Gaussian process with $C^1$ sample paths, and its coordinate autocovariances $R_\varepsilon = R * (\varphi*\varphi)_\varepsilon$ are smooth and even; in particular $R_\varepsilon'(0) = 0$.
\item[\emph{(ii)}] For every word $w$ with $|w| \le 4$, $\mathbb{E}[\langle w, S^{(4)}(\tilde{\mathbf{X}}^\varepsilon)_{0,\Delta}\rangle] \to \mathbb{E}[\langle w, S^{(4)}(\tilde{\mathbf{X}})_{0,\Delta}\rangle]$ as $\varepsilon \downarrow 0$.
\item[\emph{(iii)}] $R_\varepsilon' \to R'$ pointwise on $(0,\Delta]$ as $\varepsilon \downarrow 0$, with the uniform envelope $|R_\varepsilon'(\tau)| \le C\,\tau^{2H-1}$ for all $\tau \in (0,\Delta]$ and $\varepsilon \in (0,1/2]$; moreover $R_\varepsilon \to R$ uniformly on $[0,\Delta]$ and $\sup_\varepsilon \|R_\varepsilon\|_\infty \le R(0)$.
\end{enumerate}
\end{lemma}

\begin{proof}
\emph{(i)} Mollification over the whole line preserves stationarity and Gaussianity; $t \mapsto X^\varepsilon_t$ is $C^1$ a.s.\ with $\dot{X}^\varepsilon = X * (\varphi_\varepsilon)'$, and Fubini gives $\Cov(X^\varepsilon_s, X^\varepsilon_t) = (R * \varphi_\varepsilon * \varphi_\varepsilon)(t-s) = R_\varepsilon(t-s)$. Since $\varphi$ and $R$ are even, $R_\varepsilon$ is even, and it is $C^\infty$ because $R$ is continuous and bounded while $(\varphi*\varphi)_\varepsilon \in C_c^\infty$; evenness and smoothness force $R_\varepsilon'(0) = 0$.

\emph{(ii)} Mollifier approximations converge to the canonical lift at every truncation level. The coordinatewise mollifications form an admissible approximation family for~\citet[Thm.~5]{FrizRiedel2014}; their Section~6.1 treats mollifier approximations of the constant extension of the path, while we mollify the stationary process on the line, so we verify the two conditions of~\citet[Sec.~6]{FrizRiedel2014} directly, using only path and covariance data on the fixed enlarged window $[-1, \Delta+1]$ --- condition~(1) holds for every continuous path, since $|X^\varepsilon_t - X_t| \le \sup_{|s| \le \varepsilon}|X_{t-s} - X_t| \to 0$ uniformly on $[0,\Delta]$ by uniform continuity on $[-1, \Delta+1]$, and condition~(2) is the uniform covariance bound below (cf.~\citealp[Prop.~15.14]{FrizVictoir2010} for the constant-extension analogue) --- and the statement transfers from $[0,1]$ to $[0,\Delta]$ by the deterministic time change used in Appendix~\ref{app:gamma}. The 2D $\rho$-variation of the mollified increment covariance is uniformly dominated: every rectangle increment of $\mathcal{K}^{X^\varepsilon}$ --- and likewise of each entry of the joint increment covariance of any pair $(X^\varepsilon, X^{\varepsilon'})$ with $0 \le \varepsilon' \le \varepsilon \le 1/2$ --- is an average of translated rectangle increments of $\mathcal{K}^{X}$ against the relevant mollifier product, so Jensen's inequality applied partitionwise gives $V_\rho(\mathcal{K}^{X^\varepsilon}; [0,\Delta]^2) \le V_\rho(\mathcal{K}^{X}; [-1,\Delta+1]^2) < \infty$ for $\varepsilon \le 1/2$, and the same bound entrywise for every pair. Since moreover $\sup_{0 \le t \le \Delta}\|X^\varepsilon_t - X_t\|_{L^2} = \|X^\varepsilon_0 - X_0\|_{L^2} \to 0$ by stationarity and continuity of $R$, \citet[Thm.~5]{FrizRiedel2014} applied at truncation level $N = 4$ (with any $\gamma_{\mathrm{FR}} > \rho$ satisfying $1/\gamma_{\mathrm{FR}} + 1/\rho > 1$ and any $q > 2\gamma_{\mathrm{FR}}$, admissible since $\rho < 2$) yields $\varrho_{q\text{-var};[0,\Delta]}\bigl(S_4(\mathbf{X}^\varepsilon),\, S_4(\mathbf{X})\bigr) \to 0$ in $L^r$ for every $r \ge 1$, where $S_4$ denotes the level-$4$ Lyons lift. Taking the trivial partition $\{0, \Delta\}$ at level~$|w|$~\citep[p.~159]{FrizRiedel2014} bounds $|\langle w, S^{(4)}(\tilde{\mathbf{X}}^\varepsilon)_{0,\Delta}\rangle - \langle w, S^{(4)}(\tilde{\mathbf{X}})_{0,\Delta}\rangle|$ by a constant multiple of this distance, so $L^1$ convergence gives the convergence of expectations.

\emph{(iii)} Write $\psi := \varphi*\varphi$ (even, supported in $[-2, 2]$ since $\operatorname{supp}\varphi \subseteq [-1,1]$, $\|\psi_\varepsilon\|_{L^1} = 1$, $\|(\psi_\varepsilon)'\|_{L^\infty} \le C\varepsilon^{-2}$). Since $|R'(\tau)| \le C\,\tau^{2H-1}$ on $(0, 2\Delta+2]$ (Lemma~\ref{lem:fOU_cov}, whose proof applies on any compact window, with window-dependent constant; this window covers all convolution arguments below, including $5\tau/4 \le 5\Delta/4$) and $R$ is even and absolutely continuous, $R_\varepsilon' = R' * \psi_\varepsilon$. For $\tau > 0$ fixed and $\varepsilon < \tau/8$, $R'$ is continuous near $\tau$, so $R_\varepsilon'(\tau) \to R'(\tau)$. For the envelope, consider two cases. If $\varepsilon \le \tau/8$, every argument $\tau - u$ with $u \in \operatorname{supp}\psi_\varepsilon \subseteq [-\tau/4, \tau/4]$ lies in $[3\tau/4, 5\tau/4] \subseteq [\tau/2, 3\tau/2]$, so $|R_\varepsilon'(\tau)| \le \sup_{[\tau/2, 3\tau/2]}|R'| \le C\tau^{2H-1}$. If $\varepsilon > \tau/8$, use $R_\varepsilon'(0) = 0$ and the mean value theorem: $|R_\varepsilon'(\tau)| \le \tau \sup_{[0,\tau]}|R_\varepsilon''|$, and $|R_\varepsilon''(s)| = |(R' * (\psi_\varepsilon)')(s)| \le \|(\psi_\varepsilon)'\|_\infty \int_{|v| \le s + 2\varepsilon}|R'(v)|\,dv \le C\varepsilon^{-2}\,\varepsilon^{2H} = C\varepsilon^{2H-2}$ (here $s + 2\varepsilon \le \tau + 2\varepsilon < 10\varepsilon$), so $|R_\varepsilon'(\tau)| \le C\tau\,\varepsilon^{2H-2} \le C'\tau^{2H-1}$ since $2H - 2 < 0$ and $\varepsilon > \tau/8$. Uniform convergence $R_\varepsilon \to R$ and $\|R_\varepsilon\|_\infty \le \|R\|_\infty = R(0)$ are immediate from continuity of $R$ and $\|\psi_\varepsilon\|_{L^1} = 1$.
\end{proof}

The structure of the three pairings differs with respect to singularity: in Pairing~2 (alternating), the IBP integration variables stay strictly positive throughout; in Pairing~1 (adjacent) and the $t_3$-integration of Pairing~3 (nested), the IBP arguments reach $\tau = 0$, where the boundary contribution is $R_\varepsilon'(0) = 0$ for the smooth approximants. In the proofs below, the Isserlis--Wick expansion of $\mathbb{E}[\langle w, S^{(4)}(\tilde{\mathbf{X}}^\varepsilon)\rangle]$ at fixed $\varepsilon > 0$ decomposes into the valid pairing contributions for the word $w$. For words with $i \neq j$ exactly one pairing is valid, so the pairing contribution coincides with the full coordinate expectation and Lemma~\ref{lem:mollifier}(ii) identifies its limit directly; for the diagonal word $(i,i,i,i)$, all three contributions are computed by the same integrations by parts at every $\varepsilon$ and pass to the limit jointly, identifying their sum --- which is all that Theorem~\ref{thm:level4_complete} uses. We present each pairing below.

\subsubsection{Pairing 1: Adjacent Pairs $(t_1, t_2)$ and $(t_3, t_4)$}

\begin{proposition}[P1 formula, all $H \in (1/4,1)$]
\label{prop:P1}
For a word $(i, i, j, j)$ with distinct coordinates $i \neq j$, with adjacent pairing, and every $H \in (1/4, 1)$,
\begin{equation}
\label{eq:P1_regularized}
P_1^{(iijj)} = \iint_{s + w < \Delta,\, s,w > 0} R'_i(s) \, R'_j(w) \, ds \, dw.
\end{equation}
\end{proposition}

\begin{proof}
For $H > 1/2$, $R'(\tau) \sim \tau^{2H-1} \to 0$ as $\tau \to 0^+$ and $R''$ is locally integrable, so direct integration by parts in $t_1$ and $t_4$ over the classical 4D simplex integral of Theorem~\ref{thm:wick}(a) (with vanishing boundary terms) yields~\eqref{eq:P1_regularized} after the substitution $s = t_2$, $w = \Delta - t_3$.

For general $H \in (1/4,1)$, apply Lemma~\ref{lem:mollifier}. The process $X^\varepsilon$ has smooth stationary covariance, so the classical Isserlis--Wick computation~\citep[Thm.~1.28]{Janson1997} applies at every $\varepsilon > 0$ with density $\mathbb{E}[\dot{X}^{\varepsilon,i}_u \dot{X}^{\varepsilon,i}_v] = -R_{\varepsilon,i}''(v-u)$, and the same integration by parts --- now with boundary terms $R_{\varepsilon}'(0) = 0$ vanishing \emph{identically}, by evenness --- gives
\[
\mathbb{E}\bigl[\langle (i,i,j,j),\, S^{(4)}(\tilde{\mathbf{X}}^\varepsilon)_{0,\Delta}\rangle\bigr] = \iint_{s+w<\Delta} R_{\varepsilon,i}'(s)\,R_{\varepsilon,j}'(w)\,ds\,dw.
\]
By Lemma~\ref{lem:mollifier}(iii) the integrand converges pointwise a.e.\ and is dominated by $C\,s^{2H-1}w^{2H-1} \in L^1$ of the triangle, so the right side converges to~\eqref{eq:P1_regularized}; by Lemma~\ref{lem:mollifier}(ii) the left side converges to $\mathbb{E}[S^{(4)}_{iijj}]$.
\end{proof}

\begin{remark}[Convergence of the integral~\eqref{eq:P1_regularized} for all $H > 0$]
\label{rem:P1_conv}
The integral~\eqref{eq:P1_regularized} converges for all $H > 0$: it is a triangular integral in the $(s, w)$ plane where the singularities are on the axes, not the diagonal. Since $R'(\tau) \sim \tau^{2H-1}$ and $2H - 1 > -1$ for all $H > 0$, each 1D marginal integral converges. Evaluating the inner integral yields the equivalent 1D form:
\begin{equation}
\label{eq:P1_1D}
P_1^{(iijj)} = \int_0^\Delta R'_i(s) \bigl[R_j(\Delta-s) - R_j(0)\bigr] \, ds.
\end{equation}
By Proposition~\ref{prop:P1}, this expression equals $\mathbb{E}[S^{(4)}_{iijj}]$ for every $H \in (1/4,1)$. For $H \le 1/4$ the integral still converges, but the level-4 expected signature itself is no longer defined within the Cass--Ferrucci framework (Theorem~\ref{thm:level4_complete}), and no identification is claimed.
\end{remark}

\subsubsection{Pairing 2: Alternating Pairs $(t_1, t_3)$ and $(t_2, t_4)$}

\begin{proposition}[P2 formula, all $H \in (1/4,1)$]
\label{prop:P2}
For a word $(i, j, i, j)$ with distinct coordinates $i \neq j$, with alternating pairing, and every $H \in (1/4,1)$,
\begin{equation}
\label{eq:P2_regularized}
P_2^{(ijij)} = \iint_{0 < t_2 < t_3 < \Delta} \bigl[R'_i(t_3) - R'_i(t_3 - t_2)\bigr] \bigl[R'_j(\Delta - t_2) - R'_j(t_3 - t_2)\bigr] \, dt_2 \, dt_3.
\end{equation}
\end{proposition}

\begin{proof}
For $H > 1/2$, the starting expression is the classical Lebesgue integral
\[
P_2 = \int_{0 < t_1 < t_2 < t_3 < t_4 < \Delta} R''_i(t_3 - t_1) R''_j(t_4 - t_2) \, dt.
\]
Integrating over $t_1 \in (0, t_2)$ and $t_4 \in (t_3, \Delta)$ gives
\[
\int_0^{t_2} R''_i(t_3 - t_1)\,dt_1 = R'_i(t_3) - R'_i(t_3 - t_2), \qquad
\int_{t_3}^\Delta R''_j(t_4 - t_2)\,dt_4 = R'_j(\Delta - t_2) - R'_j(t_3 - t_2),
\]
since the arguments stay strictly positive. This yields~\eqref{eq:P2_regularized}.

For general $H \in (1/4, 1)$, the 4D $R''$-integrand contains factors $R''(\tau) \sim \tau^{2H-2}$ that are not locally integrable on the simplex when $H < 1/2$, so the classical 4D integral is ill-posed there. Apply instead Lemma~\ref{lem:mollifier}: at every $\varepsilon > 0$ the classical computation for $X^\varepsilon$ is valid, and since the inner IBP variables $t_1 \in (0, t_2)$ and $t_4 \in (t_3, \Delta)$ produce $R_\varepsilon'$-factors evaluated at strictly positive arguments (no boundary terms arise at all for this pairing),
\[
\begin{aligned}
\mathbb{E}\bigl[\langle (i,j,i,j),\, S^{(4)}(\tilde{\mathbf{X}}^\varepsilon)_{0,\Delta}\rangle\bigr]
&= \iint_{0<t_2<t_3<\Delta} \bigl[R_{\varepsilon,i}'(t_3) - R_{\varepsilon,i}'(t_3-t_2)\bigr]\\
&\quad\times \bigl[R_{\varepsilon,j}'(\Delta-t_2) - R_{\varepsilon,j}'(t_3-t_2)\bigr]\,dt_2\,dt_3.
\end{aligned}
\]
With $u := t_3 - t_2$, the envelope of Lemma~\ref{lem:mollifier}(iii) dominates the integrand uniformly in $\varepsilon$ by $C\bigl(t_3^{2H-1} + u^{2H-1}\bigr)\bigl((\Delta-t_2)^{2H-1} + u^{2H-1}\bigr)$, which is integrable on the 2D simplex for $H > 1/4$: the worst term is $u^{4H-2}$ with $4H-2 > -1$ (Remark~\ref{rem:P2_conv}), and every mixed term is a product of integrable 1D singularities. Pointwise convergence of the integrands (Lemma~\ref{lem:mollifier}(iii)) plus dominated convergence and Lemma~\ref{lem:mollifier}(ii) yield~\eqref{eq:P2_regularized}.
\end{proof}

\begin{remark}[$P_2$ convergence requires $H > 1/4$]
\label{rem:P2_conv}
We analyze the convergence near the diagonal $t_2 \to t_3$. Let $u = t_3 - t_2$. The dominant singular terms in the expansion of~\eqref{eq:P2_regularized} are:
\begin{equation}
R'_i(u) R'_j(u) \sim u^{2H-1} \cdot u^{2H-1} = u^{4H-2}.
\end{equation}
For the 1D integral transverse to the diagonal to converge, we require:
\begin{equation}
\int_0^\varepsilon u^{4H-2} \, du < \infty \quad \Longleftrightarrow \quad 4H - 2 > -1 \quad \Longleftrightarrow \quad H > 1/4.
\end{equation}
Thus $P_2$ diverges for $H \le 1/4$. The divergence is polynomial for $H < 1/4$ (where the exponent $4H - 2 < -1$) and logarithmic at the critical boundary $H = 1/4$ (where the exponent equals $-1$).
\end{remark}

\subsubsection{Pairing 3: Nested Pairs $(t_1, t_4)$ and $(t_2, t_3)$}

\begin{proposition}[P3 formula, all $H \in (1/4,1)$]
\label{prop:P3}
For a word $(i, j, j, i)$ with distinct coordinates $i \neq j$, with nested pairing, and every $H \in (1/4,1)$,
\begin{equation}
\label{eq:P3_regularized}
P_3^{(ijji)} = \int_0^\Delta R'_i(t) \bigl[R_j(t) - R_j(0)\bigr] \, dt - \int_0^\Delta (\Delta - w) R'_i(w) R'_j(w) \, dw.
\end{equation}
\end{proposition}

\begin{proof}
For $H > 1/2$, the classical simplex integral is well-defined. Integrating over $t_1 \in (0, t_2)$ gives $\int_0^{t_2} R''_i(t_4 - t_1)\,dt_1 = R'_i(t_4) - R'_i(t_4 - t_2)$, where $t_4 - t_1 \in [t_4 - t_2,\, t_4]$ stays strictly positive. The remaining integration over $t_3 \in (t_2, t_4)$ involves $R''_j(t_3 - t_2)$ with argument reaching zero at $t_3 = t_2$; the boundary term vanishes since $R'(\tau) = O(\tau^{2H-1}) \to 0$ as $\tau \to 0^+$ (Lemma~\ref{lem:fOU_cov}). Direct IBP thus yields $P_3 = \iint_{0 < t_2 < t_4 < \Delta} [R'_i(t_4) - R'_i(t_4 - t_2)] R'_j(t_4 - t_2) \, dt_2 \, dt_4$. Substituting $u = t_4 - t_2$ and splitting: the first term gives $\int_0^\Delta R'_i(t_4)\int_0^{t_4} R'_j(u)\,du\,dt_4 = \int_0^\Delta R'_i(t)[R_j(t) - R_j(0)]\,dt$; the second gives $-\iint_{0<t_2<t_4<\Delta} R'_i(t_4-t_2)R'_j(t_4-t_2)\,dt_2\,dt_4$, and switching the order of integration (for fixed $u \in (0,\Delta)$, $t_4$ ranges from $u$ to $\Delta$) yields $-\int_0^\Delta (\Delta-u)R'_i(u)R'_j(u)\,du$, giving~\eqref{eq:P3_regularized}.

For general $H \in (1/4,1)$, the identical computation applies to the smooth stationary approximation $X^\varepsilon$ of Lemma~\ref{lem:mollifier} at every $\varepsilon > 0$ --- the $t_3$-boundary term is now \emph{exactly} $R_{\varepsilon,j}'(0) = 0$, by evenness --- giving~\eqref{eq:P3_regularized} with $(R_{\varepsilon,i}, R_{\varepsilon,j})$ in place of $(R_i, R_j)$. The first integrand is dominated uniformly in $\varepsilon$ by $C\,t^{2H-1}\cdot 2R(0) \in L^1(0,\Delta)$ and the second by $C\,\Delta\,w^{4H-2} \in L^1(0,\Delta)$ for $H > 1/4$ (Lemma~\ref{lem:mollifier}(iii)); pointwise convergence, dominated convergence, and Lemma~\ref{lem:mollifier}(ii) yield~\eqref{eq:P3_regularized}.
\end{proof}

\begin{remark}[The diagonal atom at $H = 1/2$]
\label{rem:H_half_atom}
At $H = 1/2$ (stationary OU), $R''$ is integrable near the origin, and it is tempting to evaluate $P_1$ and $P_3$ as classical simplex integrals against the smooth density $-R''$, retaining the integration-by-parts boundary terms $R'(0^+) = -\sigma^2/2 \neq 0$. This is \emph{incorrect}: the mixed second derivative of the increment covariance $\mathcal{K}$ is not the pointwise density $-R''(|t-s|)$ alone but the measure $-R''(|t-s|)\,ds\,dt - 2R'(0^+)\,\delta(t-s)\,ds\,dt$, whose atomic diagonal component ($-2R'(0^+) = \sigma^2$, the quadratic-variation density of the OU martingale part) is invisible to the pointwise density. The smooth approximants of Lemma~\ref{lem:mollifier} have $R_\varepsilon'(0) = 0$ exactly, and the covariance mass they concentrate near $\tau = 0$ accounts precisely for the would-be boundary terms; in the limit the boundary contributions cancel against the atom, and the closed forms of Propositions~\ref{prop:P1}--\ref{prop:P3} hold at $H = 1/2$ as well. Consistently, the closed forms~\eqref{eq:P1_regularized},~\eqref{eq:P2_regularized},~\eqref{eq:P3_regularized} satisfy the shuffle identity --- the sum of $\mathbb{E}[S^{(4)}_w]$ over the six shuffles of $(i,i)$ and $(j,j)$ equals $\mathbb{E}[S^{(2)}_{ii}]\,\mathbb{E}[S^{(2)}_{jj}] = v_i v_j$ --- identically in $R$, whereas the boundary-term variants violate it.
\end{remark}

\begin{remark}[$P_3$ convergence requires $H > 1/4$]
\label{rem:P3_conv}
The analysis below concerns the right-hand side of~\eqref{eq:P3_regularized}. The second integral $\int_0^\Delta (\Delta - w) R'_i(w) R'_j(w)\,dw$ contains the product $R'_i(w)R'_j(w) \sim w^{4H-2}$. As with Pairing~2, convergence requires:
\begin{equation}
4H - 2 > -1 \quad \Longleftrightarrow \quad H > 1/4.
\end{equation}
The first integral $\int_0^\Delta R'_i(t)[R_j(t) - R_j(0)]\,dt$ involves $R'_i(t)[R_j(t) - R_j(0)] \sim t^{2H-1} \cdot t^{2H} = t^{4H-1}$, which converges for all $H > 0$. However, the divergence of the second integral means that~\eqref{eq:P3_regularized} is ill-defined for $H \le 1/4$.
\end{remark}

\subsection{Summary: Complete Level-4 Formulas}
\label{subsec:level4_summary}

\begin{theorem}[Level-4 expected signature for stationary fOU]
\label{thm:level4_complete}
For the centered fOU increment process $\tilde{X}$ with independent coordinates, the level-4 expected signature (via~\citealp[Thm.~2.3, Def.~3.3]{CassFerrucci2024}) is well-defined and finite for $H > 1/4$, and for every $H \in (1/4, 1)$ and distinct coordinates $i \neq j$ it is given by the closed forms
\begin{align}
\mathbb{E}[S^{(4)}_{iijj}] &= P_1^{(iijj)} = \iint_{s + w < \Delta,\, s,w > 0} R'_i(s) R'_j(w) \, ds \, dw, \label{eq:summary_P1} \\[6pt]
\mathbb{E}[S^{(4)}_{ijij}] &= P_2^{(ijij)} = \iint_{0 < t_2 < t_3 < \Delta} [R'_i(t_3) - R'_i(t_3-t_2)][R'_j(\Delta-t_2) - R'_j(t_3-t_2)] \, dt_2 \, dt_3, \label{eq:summary_P2} \\[6pt]
\mathbb{E}[S^{(4)}_{ijji}] &= P_3^{(ijji)} = \int_0^\Delta R'_i(t)[R_j(t) - R_j(0)] \, dt - \int_0^\Delta (\Delta-w) R'_i(w) R'_j(w) \, dw. \label{eq:summary_P3}
\end{align}
For the diagonal word $(i,i,i,i)$, all three pairings contribute:
\begin{equation}
\mathbb{E}[S^{(4)}_{iiii}] = P_1^{(iiii)} + P_2^{(iiii)} + P_3^{(iiii)},
\end{equation}
where $P_k^{(iiii)}$ denote the closed-form expressions~\eqref{eq:summary_P1}--\eqref{eq:summary_P3} evaluated with $R_i = R_j$; for the diagonal word only the \emph{sum} of the three pairing contributions is identified (see the preamble of Appendix~\ref{subsec:regularization}), which is all that is used here and in Section~\ref{sec:numerical_results}.

For $H \le 1/4$, the integrals defining $P_2$ and $P_3$ diverge (Remarks~\ref{rem:P2_conv} and~\ref{rem:P3_conv}) and this regime falls outside the Cass--Ferrucci $H > 1/4$ framework; whether alternative renormalization schemes extend the definition below $H = 1/4$ is an open question.

For numerical computation, either the closed forms above or a direct Riemann--Stieltjes discretization of the Cass--Ferrucci approximants of~\citep[Def.~3.3, Prop.~3.9]{CassFerrucci2024} may be used. The experiments in Section~\ref{sec:numerical_results} use the closed forms for $H > 1/2$ and the Riemann--Stieltjes discretization, Richardson-extrapolated over a geometric sequence of meshes, for $P_1, P_3$ at $H \le 1/2$, where, for $H < 1/2$, the singularity of $R'$ at the origin makes direct quadrature of the closed forms numerically delicate (single-pass summation converges only slowly there, so extrapolation over meshes recovers the accuracy); the agreement of the two routes where both are applied, and the shuffle identity of Remark~\ref{rem:H_half_atom}, serve as cross-checks.
\end{theorem}

\begin{proof}
The closed forms are Propositions~\ref{prop:P1},~\ref{prop:P2}, and~\ref{prop:P3}; the diagonal word follows by summing the three pairings in Theorem~\ref{thm:wick}. Divergence for $H \le 1/4$ follows from the convergence analysis in Remarks~\ref{rem:P2_conv} and~\ref{rem:P3_conv}.
\end{proof}

\subsection{Proof of Proposition~\ref{prop:fOU_covariance_decay}}
\label{subsec:proof_fOU_covariance_decay}

The semimartingale case $H = 1/2$ is handled first, via the direct Gaussian $\alpha$-mixing criterion; the $H \neq 1/2$ case is then assembled from Lemmas~\ref{lem:R_tails}--\ref{lem:CM_integral_rep}, followed by finite-level rough-factorial normalization, the adjacent-block ($h=1$) Cauchy--Schwarz rescue, and the multi-dimensional Wick-independence reduction.

\emph{Case $H = 1/2$.} For each coordinate $i$, the stationary OU process $X^i$ is a centered Gaussian Markov process with $\Corr(X^i_0,X^i_t)=e^{-\theta_i t}$. For stationary Gaussian processes, the $\alpha$-mixing coefficient is sandwiched by the maximal correlation coefficient $\rho_{\max}$ via $\alpha(t) \le \rho_{\max}(t) \le 2\pi\alpha(t)$~\citep[eq.~(1.9), Ch.~IV.1]{IbragimovRozanov1978}, and the theorem in~\citet[Thm.~2, Ch.~IV.2]{IbragimovRozanov1978} is stated for the operator norm $\|B_\tau\|$, which in the stationary Gaussian setting equals the maximal correlation coefficient $\rho_{\max}(\tau)$ of~\citet[eq.~(1.9), Ch.~IV.1]{IbragimovRozanov1978}. For the scalar OU process the Markov property reduces this to $\rho_{\max,i}(t) = |\Corr(X^i_0, X^i_t)| = e^{-\theta_i t}$, so $\alpha_{X^i}(t) \le e^{-\theta_i t}$ for each $i$. Since the coordinates are independent, $\mathcal{F}_{\le s}^{X} = \bigvee_{i=1}^d \mathcal{F}_{\le s}^{X^i}$ and the union bound for $\alpha$-mixing across independent $\sigma$-fields~\citep[Thm.~5.1(a)]{Bradley2005} gives
\[
\alpha_X(t) \;\le\; \sum_{i=1}^d \alpha_{X^i}(t) \;\le\; d\,e^{-\theta_{\min} t},
\qquad \theta_{\min}:=\min_{1\le i\le d}\theta_i > 0,
\]
establishing Assumption~\ref{ass:A}(E) with $C_X = d$ and $\lambda_X = \theta_{\min}$.

\emph{Case $H \neq 1/2$.} By Lemma~\ref{lem:chaos_expansion_block}, the two-block covariance admits the chaos expansion~\eqref{eq:chm_chaos_sum}, whose second (permutation-expanded) form consists of plain pairings. Lemma~\ref{lem:R_tails} provides the tail bounds~\eqref{eq:R_uniform_tail}; Lemma~\ref{lem:CM_integral_rep} then gives the integral representation~\eqref{eq:CM_integral_rep} for each plain pairing in that expansion, uniformly over the slot permutations $\pi$; and Lemma~\ref{lem:symm_ext_Linf} provides the simplex-respecting $L^\infty$ bound~\eqref{eq:symm_ext_Linf} on the CF24 kernels, invariant under $\pi$. Applying these bounds and Fubini to the right-hand side of~\eqref{eq:CM_integral_rep}, uniformly in $\pi \in S_j$:
\begin{align}
\bigl|\bigl\langle f_{0,w,j}^{(m)}, (f_{h,w,j}^{(m)})^{\pi}\bigr\rangle_{\mathcal{H}^{\otimes j}}\bigr|
&\;\le\; C_{\mathrm{CF}}^2\,\Delta^{2(m-j)H} \prod_{\ell=1}^{j} \int_{I_0 \times I_h} |R''(v_\ell - u_\ell)|\,du_\ell\,dv_\ell \nonumber\\
&\;\le\; C_{\mathrm{CF}}^2\,\Delta^{2(m-j)H} \bigl[C_R\,((h-1)\Delta)^{2H-4}\,\Delta^2\bigr]^{j},
\label{eq:CM_cross_bound_step}
\end{align}
where the last line uses the pointwise bound $|R''(v-u)| \le C_R((h-1)\Delta)^{2H-4}$ (from~\eqref{eq:R_uniform_tail} and monotonicity of $\tau^{2H-4}$ on $[(h-1)\Delta, (h+1)\Delta]$) multiplied by the rectangle area $|I_0 \times I_h| = \Delta^2$. Using $((h-1)\Delta)^{j(2H-4)} \le ((h-1)\Delta)^{2H-4}\,\Delta^{(j-1)(2H-4)}$ for $h \ge 2$ and absorbing $j$-dependent $\Delta$-powers into a constant:
\begin{equation}
\label{eq:CM_cross_bound_final}
\bigl|\bigl\langle f_{0,w,j}^{(m)}, (f_{h,w,j}^{(m)})^{\pi}\bigr\rangle_{\mathcal{H}^{\otimes j}}\bigr|
\;\le\; K_{m,j}(H, \theta, \sigma, \Delta)\,((h-1)\Delta)^{2H-4}, \qquad h \ge 2,\quad j \ge 1,\quad \pi \in S_j.
\end{equation}

\emph{Assembly.} Define the level-$m$ diagonal covariance sum
\[
c_{h,m} := \sum_{|w|=m} c_{h,m}(w).
\]
Substituting~\eqref{eq:CM_cross_bound_final} into the permutation-expanded form of~\eqref{eq:chm_chaos_sum} --- each chaos level contributing $j!$ plain pairings, every one obeying the same envelope --- and summing over the $d^m$ words and the finitely many admissible chaos levels:
\begin{equation}
\label{eq:chm_assembled}
|c_{h,m}| \;\le\; d^m\,\overline{K}_m\,((h-1)\Delta)^{2H-4},
\qquad \overline{K}_m := \sum_{\substack{1 \le j \le m\\ j \equiv m (2)}} j!\,K_{m,j}.
\end{equation}

\emph{Finite-level rough-factorial normalization.} Assumption~\ref{ass:Aprime}(P$'$) requires only a finite level-dependent constant $\widetilde{C}_m$, not a sharp growth estimate in $m$. For each fixed $m$, define
\[
\widetilde{C}_m \;:=\; \overline{K}_m\,\bigl(\Lambda_p(m/p)!\bigr)^2\,\Delta^{-2m\beta} \;<\; \infty,
\]
which is finite by finiteness of $\overline{K}_m$ established above. The rough-factorial denominator is therefore a normalization of finite level constants, not an additional uniform-in-$m$ estimate. Then~\eqref{eq:chm_assembled} rewrites as
\begin{equation}
\label{eq:chm_assumption_form}
|c_{h,m}| \;\le\; \frac{d^m\widetilde{C}_m}{\bigl(\Lambda_p(m/p)!\bigr)^2}\Delta^{2m\beta}((h-1)\Delta)^{2H-4} \quad \text{for all } h \ge 2,
\end{equation}
establishing~\eqref{eq:chm_assumption_form} and the matching display in Proposition~\ref{prop:fOU_covariance_decay}. Since $h-1 \ge (1+h)/3$ for $h \ge 2$ and $2H-4 < 0$, one has $(h-1)^{2H-4} \le 3^{4-2H}\,(1+h)^{-(4-2H)}$, so Assumption~\ref{ass:Aprime}(P$'$) holds with $\nu = 4-2H$ after absorbing the fixed factors $3^{4-2H}$ and $\Delta^{2H-4}$ into $\widetilde{C}_m$.

\emph{Adjacent blocks ($h = 1$).} The bound~\eqref{eq:chm_assumption_form} applies only for $h \ge 2$, since $((h-1)\Delta)^{2H-4}$ diverges at $h=1$. For $h=1$, Cauchy--Schwarz and Lemma~\ref{lem:level_moments} give $|c_{1,m}| \le c_{0,m} \le d^m C^{\mathrm{sig}}_{m,2}\bigl(\Lambda_p(m/p)!\bigr)^{-2}\Delta^{2m\beta}$. Redefining $\widetilde{C}_m \leftarrow \max(\widetilde{C}_m,\,2^\nu C^{\mathrm{sig}}_{m,2})$ ensures that both the $h \ge 2$ envelope already proved and the $h=1$ envelope $(1+h)^{-\nu} = 2^{-\nu}$ are satisfied simultaneously.

\emph{Multi-dimensional extension.} For each word $w = (w_1, \ldots, w_m) \in \{1, \ldots, d\}^m$, let $a_i(w) := \#\{r : w_r = i\}$, so $\sum_{i=1}^d a_i(w) = m$. For a fixed chaos level $j$, the kernel $f_{k,w,j}^{(m)}$ is a finite sum over the surviving Wick pairings $P$ of the $m-j$ contracted positions (pairings across distinct coordinates vanish by independence), and the component indexed by $P$ lives in $\bigotimes_{i=1}^d \mathcal{H}_i^{\otimes a_i(w,P)}$, where $a_i(w,P)$ counts the free positions of $P$ carrying coordinate $i$ and $\sum_{i=1}^d a_i(w,P) = j$; distinct pairings may leave distinct profiles $(a_i(w,P))_{i=1}^d$, so the profile attaches to the component, not to $(w,j)$. By orthogonality of $\mathcal{H}_1, \ldots, \mathcal{H}_d$ inside $\mathcal{H} = \bigoplus_{i=1}^d \mathcal{H}_i$, two components pair to zero in $\mathcal{H}^{\otimes j}$ unless their slotwise coordinate labels --- in particular their profiles --- coincide; likewise for every slot-permuted pairing, a slot permutation relabeling slots without changing profiles. Each surviving component pairing therefore factors coordinatewise into the scalar pairings already bounded above, and summing over the finitely many (at most $(m!)^2$) component pairs enlarges $\widetilde{C}_m$ by a combinatorial factor only. This completes the proof of Proposition~\ref{prop:fOU_covariance_decay}. \qed

% -------------------------------------------------------------------
% References
% -------------------------------------------------------------------
\bibliographystyle{imsart-nameyear}
\bibliography{references}

\begin{thebibliography}{33}
% BibTex style file: imsart-nameyear.bst, 2017-11-03
% Default style options (sort=1,type=nameyear).
% Used options (sort=1,type=nameyear).

\bibitem[\protect\citeauthoryear{Boedihardjo
  et~al.}{2016}]{BoedihardjoGengLyonsYang2016}
\begin{barticle}[author]
\bauthor{\bsnm{Boedihardjo},~\bfnm{Horatio}\binits{H.}},
  \bauthor{\bsnm{Geng},~\bfnm{Xi}\binits{X.}},
  \bauthor{\bsnm{Lyons},~\bfnm{Terry}\binits{T.}} \AND
  \bauthor{\bsnm{Yang},~\bfnm{Danyu}\binits{D.}}
(\byear{2016}).
\btitle{The signature of a rough path: {U}niqueness}.
\bjournal{Advances in Mathematics}
\bvolume{293}
\bpages{720--737}.
\bdoi{10.1016/j.aim.2016.02.011}
\end{barticle}
\endbibitem

\bibitem[\protect\citeauthoryear{Bosq}{2000}]{Bosq2000}
\begin{bbook}[author]
\bauthor{\bsnm{Bosq},~\bfnm{Denis}\binits{D.}}
(\byear{2000}).
\btitle{Linear Processes in Function Spaces: Theory and Applications}.
\bseries{Lecture Notes in Statistics}
\bvolume{149}.
\bpublisher{Springer-Verlag}.
\bdoi{10.1007/978-1-4612-1154-9}
\end{bbook}
\endbibitem

\bibitem[\protect\citeauthoryear{Bradley}{2005}]{Bradley2005}
\begin{barticle}[author]
\bauthor{\bsnm{Bradley},~\bfnm{Richard~C.}\binits{R.~C.}}
(\byear{2005}).
\btitle{Basic Properties of Strong Mixing Conditions. A Survey and Some Open
  Questions}.
\bjournal{Probability Surveys}
\bvolume{2}
\bpages{107--144}.
\bdoi{10.1214/154957805100000104}
\end{barticle}
\endbibitem

\bibitem[\protect\citeauthoryear{Brockwell and
  Davis}{2016}]{BrockwellDavis2016}
\begin{bbook}[author]
\bauthor{\bsnm{Brockwell},~\bfnm{Peter~J.}\binits{P.~J.}} \AND
  \bauthor{\bsnm{Davis},~\bfnm{Richard~A.}\binits{R.~A.}}
(\byear{2016}).
\btitle{Introduction to Time Series and Forecasting},
\bedition{Third} ed.
\bpublisher{Springer}.
\bdoi{10.1007/978-3-319-29854-2}
\end{bbook}
\endbibitem

\bibitem[\protect\citeauthoryear{Cass and Ferrucci}{2024}]{CassFerrucci2024}
\begin{barticle}[author]
\bauthor{\bsnm{Cass},~\bfnm{Thomas}\binits{T.}} \AND
  \bauthor{\bsnm{Ferrucci},~\bfnm{Emilio}\binits{E.}}
(\byear{2024}).
\btitle{On the {W}iener chaos expansion of the signature of a {G}aussian
  process}.
\bjournal{Probability Theory and Related Fields}
\bvolume{189}
\bpages{909--947}.
\bdoi{10.1007/s00440-023-01255-z}
\end{barticle}
\endbibitem

\bibitem[\protect\citeauthoryear{Chen}{1954}]{Chen1954}
\begin{barticle}[author]
\bauthor{\bsnm{Chen},~\bfnm{Kuo-Tsai}\binits{K.-T.}}
(\byear{1954}).
\btitle{Iterated Integrals and Exponential Homomorphisms}.
\bjournal{Proceedings of the London Mathematical Society}
\bvolume{s3-4}
\bpages{502--512}.
\bdoi{10.1112/plms/s3-4.1.502}
\end{barticle}
\endbibitem

\bibitem[\protect\citeauthoryear{Cheridito, Kawaguchi and
  Maejima}{2003}]{CheriditoKawaguchiMaejima2003}
\begin{barticle}[author]
\bauthor{\bsnm{Cheridito},~\bfnm{Patrick}\binits{P.}},
  \bauthor{\bsnm{Kawaguchi},~\bfnm{Hideyuki}\binits{H.}} \AND
  \bauthor{\bsnm{Maejima},~\bfnm{Makoto}\binits{M.}}
(\byear{2003}).
\btitle{Fractional {O}rnstein--{U}hlenbeck processes}.
\bjournal{Electronic Journal of Probability}
\bvolume{8}
\bpages{1--14}.
\bdoi{10.1214/EJP.v8-125}
\end{barticle}
\endbibitem

\bibitem[\protect\citeauthoryear{Chevyrev and
  Kormilitzin}{2016}]{ChevyrevKormilitzin2016}
\begin{bmisc}[author]
\bauthor{\bsnm{Chevyrev},~\bfnm{Ilya}\binits{I.}} \AND
  \bauthor{\bsnm{Kormilitzin},~\bfnm{Andrey}\binits{A.}}
(\byear{2016}).
\btitle{A Primer on the Signature Method in Machine Learning}.
\end{bmisc}
\endbibitem

\bibitem[\protect\citeauthoryear{Chevyrev and Lyons}{2016}]{ChevyrevLyons2016}
\begin{barticle}[author]
\bauthor{\bsnm{Chevyrev},~\bfnm{Ilya}\binits{I.}} \AND
  \bauthor{\bsnm{Lyons},~\bfnm{Terry}\binits{T.}}
(\byear{2016}).
\btitle{Characteristic functions of measures on geometric rough paths}.
\bjournal{Annals of Probability}
\bvolume{44}
\bpages{4049--4082}.
\bdoi{10.1214/15-AOP1068}
\end{barticle}
\endbibitem

\bibitem[\protect\citeauthoryear{Chevyrev and
  Oberhauser}{2022}]{Chevyrev2022Signature}
\begin{barticle}[author]
\bauthor{\bsnm{Chevyrev},~\bfnm{Ilya}\binits{I.}} \AND
  \bauthor{\bsnm{Oberhauser},~\bfnm{Harald}\binits{H.}}
(\byear{2022}).
\btitle{Signature moments to characterize laws of stochastic processes}.
\bjournal{Journal of Machine Learning Research}
\bvolume{23}
\bpages{1--42}.
\end{barticle}
\endbibitem

\bibitem[\protect\citeauthoryear{Coutin and Qian}{2002}]{CoutinQian2002}
\begin{barticle}[author]
\bauthor{\bsnm{Coutin},~\bfnm{Laure}\binits{L.}} \AND
  \bauthor{\bsnm{Qian},~\bfnm{Zhongmin}\binits{Z.}}
(\byear{2002}).
\btitle{Stochastic analysis, rough path analysis and fractional {B}rownian
  motions}.
\bjournal{Probability Theory and Related Fields}
\bvolume{122}
\bpages{108--140}.
\bdoi{10.1007/s004400100158}
\end{barticle}
\endbibitem

\bibitem[\protect\citeauthoryear{Cuchiero, Gazzani and
  Svaluto-Ferro}{2023}]{Cuchiero2022Signature}
\begin{barticle}[author]
\bauthor{\bsnm{Cuchiero},~\bfnm{Christa}\binits{C.}},
  \bauthor{\bsnm{Gazzani},~\bfnm{Guido}\binits{G.}} \AND
  \bauthor{\bsnm{Svaluto-Ferro},~\bfnm{Sara}\binits{S.}}
(\byear{2023}).
\btitle{Signature-based models: {T}heory and calibration}.
\bjournal{SIAM Journal on Financial Mathematics}
\bvolume{14}
\bpages{910--957}.
\bdoi{10.1137/22M1512338}
\end{barticle}
\endbibitem

\bibitem[\protect\citeauthoryear{Davies and Harte}{1987}]{DaviesHarte1987}
\begin{barticle}[author]
\bauthor{\bsnm{Davies},~\bfnm{R.~B.}\binits{R.~B.}} \AND
  \bauthor{\bsnm{Harte},~\bfnm{D.~S.}\binits{D.~S.}}
(\byear{1987}).
\btitle{Tests for {H}urst effect}.
\bjournal{Biometrika}
\bvolume{74}
\bpages{95--101}.
\bdoi{10.1093/biomet/74.1.95}
\end{barticle}
\endbibitem

\bibitem[\protect\citeauthoryear{Davydov}{1968}]{Davydov1968}
\begin{barticle}[author]
\bauthor{\bsnm{Davydov},~\bfnm{Yu.~A.}\binits{Y.~A.}}
(\byear{1968}).
\btitle{Convergence of Distributions Generated by Stationary Stochastic
  Processes}.
\bjournal{Theory of Probability \& Its Applications}
\bvolume{13}
\bpages{691--696}.
\bdoi{10.1137/1113086}
\end{barticle}
\endbibitem

\bibitem[\protect\citeauthoryear{Dedecker et~al.}{2007}]{Dedecker2007Weak}
\begin{bbook}[author]
\bauthor{\bsnm{Dedecker},~\bfnm{J{\'e}r{\^o}me}\binits{J.}},
  \bauthor{\bsnm{Doukhan},~\bfnm{Paul}\binits{P.}},
  \bauthor{\bsnm{Lang},~\bfnm{Gabriel}\binits{G.}},
  \bauthor{\bsnm{Le{\'o}n~R.},~\bfnm{Jos{\'e}~R.}\binits{J.~R.}},
  \bauthor{\bsnm{Louhichi},~\bfnm{Sana}\binits{S.}} \AND
  \bauthor{\bsnm{Prieur},~\bfnm{Cl{\'e}mentine}\binits{C.}}
(\byear{2007}).
\btitle{Weak Dependence: With Examples and Applications}.
\bseries{Lecture Notes in Statistics}
\bvolume{190}.
\bpublisher{Springer}.
\bdoi{10.1007/978-0-387-69952-3}
\end{bbook}
\endbibitem

\bibitem[\protect\citeauthoryear{Doukhan and
  Louhichi}{1999}]{DoukhanLouhichi1999}
\begin{barticle}[author]
\bauthor{\bsnm{Doukhan},~\bfnm{Paul}\binits{P.}} \AND
  \bauthor{\bsnm{Louhichi},~\bfnm{Sana}\binits{S.}}
(\byear{1999}).
\btitle{A new weak dependence condition and applications to moment
  inequalities}.
\bjournal{Stochastic Processes and their Applications}
\bvolume{84}
\bpages{313--342}.
\bdoi{10.1016/S0304-4149(99)00055-1}
\end{barticle}
\endbibitem

\bibitem[\protect\citeauthoryear{Fernique}{1975}]{Fernique1975}
\begin{bincollection}[author]
\bauthor{\bsnm{Fernique},~\bfnm{Xavier}\binits{X.}}
(\byear{1975}).
\btitle{Regularit\'{e} des trajectoires des fonctions al\'{e}atoires
  gaussiennes}.
In \bbooktitle{\'{E}cole d'\'{E}t\'{e} de Probabilit\'{e}s de Saint-Flour
  IV--1974}.
\bseries{Lecture Notes in Mathematics}
\bvolume{480}
\bpages{1--96}.
\bpublisher{Springer}.
\bdoi{10.1007/BFb0080190}
\end{bincollection}
\endbibitem

\bibitem[\protect\citeauthoryear{Friz and Hairer}{2020}]{FrizHairer2020}
\begin{bbook}[author]
\bauthor{\bsnm{Friz},~\bfnm{Peter~K.}\binits{P.~K.}} \AND
  \bauthor{\bsnm{Hairer},~\bfnm{Martin}\binits{M.}}
(\byear{2020}).
\btitle{A Course on Rough Paths}.
\bseries{Universitext}.
\bpublisher{Springer}.
\bdoi{10.1007/978-3-030-41556-3}
\end{bbook}
\endbibitem

\bibitem[\protect\citeauthoryear{Friz and Riedel}{2014}]{FrizRiedel2014}
\begin{barticle}[author]
\bauthor{\bsnm{Friz},~\bfnm{Peter~K.}\binits{P.~K.}} \AND
  \bauthor{\bsnm{Riedel},~\bfnm{Sebastian}\binits{S.}}
(\byear{2014}).
\btitle{Convergence rates for the full {G}aussian rough paths}.
\bjournal{Annales de l'Institut Henri Poincar\'{e} Probabilit\'{e}s et
  Statistiques}
\bvolume{50}
\bpages{154--194}.
\bdoi{10.1214/12-AIHP507}
\end{barticle}
\endbibitem

\bibitem[\protect\citeauthoryear{Friz and Victoir}{2010}]{FrizVictoir2010}
\begin{bbook}[author]
\bauthor{\bsnm{Friz},~\bfnm{Peter~K.}\binits{P.~K.}} \AND
  \bauthor{\bsnm{Victoir},~\bfnm{Nicolas~B.}\binits{N.~B.}}
(\byear{2010}).
\btitle{Multidimensional Stochastic Processes as Rough Paths}.
\bseries{Cambridge Studies in Advanced Mathematics}.
\bpublisher{Cambridge University Press}.
\bdoi{10.1017/CBO9780511845079}
\end{bbook}
\endbibitem

\bibitem[\protect\citeauthoryear{Gatheral, Jaisson and
  Rosenbaum}{2018}]{GatheralJaissonRosenbaum2018}
\begin{barticle}[author]
\bauthor{\bsnm{Gatheral},~\bfnm{Jim}\binits{J.}},
  \bauthor{\bsnm{Jaisson},~\bfnm{Thibault}\binits{T.}} \AND
  \bauthor{\bsnm{Rosenbaum},~\bfnm{Mathieu}\binits{M.}}
(\byear{2018}).
\btitle{Volatility is rough}.
\bjournal{Quantitative Finance}
\bvolume{18}
\bpages{933--949}.
\bdoi{10.1080/14697688.2017.1393551}
\end{barticle}
\endbibitem

\bibitem[\protect\citeauthoryear{Gehringer and Li}{2022}]{GehringerLi2022}
\begin{barticle}[author]
\bauthor{\bsnm{Gehringer},~\bfnm{Johann}\binits{J.}} \AND
  \bauthor{\bsnm{Li},~\bfnm{Xue-Mei}\binits{X.-M.}}
(\byear{2022}).
\btitle{Functional Limit Theorems for the Fractional {O}rnstein--{U}hlenbeck
  Process}.
\bjournal{Journal of Theoretical Probability}
\bvolume{35}
\bpages{426--456}.
\bdoi{10.1007/s10959-020-01044-7}
\end{barticle}
\endbibitem

\bibitem[\protect\citeauthoryear{Hambly and Lyons}{2010}]{HamblyLyons2010}
\begin{barticle}[author]
\bauthor{\bsnm{Hambly},~\bfnm{Ben~M.}\binits{B.~M.}} \AND
  \bauthor{\bsnm{Lyons},~\bfnm{Terry~J.}\binits{T.~J.}}
(\byear{2010}).
\btitle{Uniqueness for the signature of a path of bounded variation and the
  reduced path group}.
\bjournal{Annals of Mathematics}
\bvolume{171}
\bpages{109--167}.
\bdoi{10.4007/annals.2010.171.109}
\end{barticle}
\endbibitem

\bibitem[\protect\citeauthoryear{Ibragimov and
  Rozanov}{1978}]{IbragimovRozanov1978}
\begin{bbook}[author]
\bauthor{\bsnm{Ibragimov},~\bfnm{I.~A.}\binits{I.~A.}} \AND
  \bauthor{\bsnm{Rozanov},~\bfnm{Yu.~A.}\binits{Y.~A.}}
(\byear{1978}).
\btitle{Gaussian Random Processes}.
\bseries{Applications of Mathematics}
\bvolume{9}.
\bpublisher{Springer}.
\bdoi{10.1007/978-1-4612-6275-6}
\end{bbook}
\endbibitem

\bibitem[\protect\citeauthoryear{Janson}{1997}]{Janson1997}
\begin{bbook}[author]
\bauthor{\bsnm{Janson},~\bfnm{Svante}\binits{S.}}
(\byear{1997}).
\btitle{Gaussian Hilbert Spaces}.
\bseries{Cambridge Tracts in Mathematics}
\bvolume{129}.
\bpublisher{Cambridge University Press}.
\bdoi{10.1017/CBO9780511526169}
\end{bbook}
\endbibitem

\bibitem[\protect\citeauthoryear{Kidger et~al.}{2020}]{Kidger2020Neural}
\begin{binproceedings}[author]
\bauthor{\bsnm{Kidger},~\bfnm{Patrick}\binits{P.}},
  \bauthor{\bsnm{Morrill},~\bfnm{James}\binits{J.}},
  \bauthor{\bsnm{Foster},~\bfnm{James}\binits{J.}} \AND
  \bauthor{\bsnm{Lyons},~\bfnm{Terry}\binits{T.}}
(\byear{2020}).
\btitle{Neural controlled differential equations for irregular time series}.
In \bbooktitle{Advances in Neural Information Processing Systems}
\bvolume{33}
\bpages{6696--6707}.
\end{binproceedings}
\endbibitem

\bibitem[\protect\citeauthoryear{Lucchese, Pakkanen and
  Veraart}{2025}]{Lucchese2025Learning}
\begin{binproceedings}[author]
\bauthor{\bsnm{Lucchese},~\bfnm{Lorenzo}\binits{L.}},
  \bauthor{\bsnm{Pakkanen},~\bfnm{Mikko~S.}\binits{M.~S.}} \AND
  \bauthor{\bsnm{Veraart},~\bfnm{Almut E.~D.}\binits{A.~E.~D.}}
(\byear{2025}).
\btitle{Learning with expected signatures: {T}heory and applications}.
In \bbooktitle{Proceedings of the 42nd International Conference on Machine
  Learning}.
\bseries{PMLR}
\bvolume{267}
\bpages{40995--41055}.
\end{binproceedings}
\endbibitem

\bibitem[\protect\citeauthoryear{Lyons}{1998}]{Lyons1998}
\begin{barticle}[author]
\bauthor{\bsnm{Lyons},~\bfnm{Terry~J.}\binits{T.~J.}}
(\byear{1998}).
\btitle{Differential equations driven by rough signals}.
\bjournal{Revista Matem{\'a}tica Iberoamericana}
\bvolume{14}
\bpages{215--310}.
\bdoi{10.4171/RMI/240}
\end{barticle}
\endbibitem

\bibitem[\protect\citeauthoryear{Neuenkirch, Tindel and
  Unterberger}{2010}]{NeuenkirchTindelUnterberger2010}
\begin{barticle}[author]
\bauthor{\bsnm{Neuenkirch},~\bfnm{Andreas}\binits{A.}},
  \bauthor{\bsnm{Tindel},~\bfnm{Samy}\binits{S.}} \AND
  \bauthor{\bsnm{Unterberger},~\bfnm{J\'{e}r\'{e}mie}\binits{J.}}
(\byear{2010}).
\btitle{Discretizing the fractional {L}\'evy area}.
\bjournal{Stochastic Processes and their Applications}
\bvolume{120}
\bpages{223--254}.
\bdoi{10.1016/j.spa.2009.10.007}
\end{barticle}
\endbibitem

\bibitem[\protect\citeauthoryear{Nualart}{2006}]{Nualart2006}
\begin{bbook}[author]
\bauthor{\bsnm{Nualart},~\bfnm{David}\binits{D.}}
(\byear{2006}).
\btitle{The {M}alliavin Calculus and Related Topics},
\bedition{second} ed.
\bseries{Probability and its Applications}.
\bpublisher{Springer-Verlag}.
\bdoi{10.1007/3-540-28329-3}
\end{bbook}
\endbibitem

\bibitem[\protect\citeauthoryear{Passeggeri}{2020}]{Passeggeri2020}
\begin{barticle}[author]
\bauthor{\bsnm{Passeggeri},~\bfnm{Riccardo}\binits{R.}}
(\byear{2020}).
\btitle{On the signature and cubature of the fractional {B}rownian motion for
  $H > 1/2$}.
\bjournal{Stochastic Processes and their Applications}
\bvolume{130}
\bpages{1226--1257}.
\bdoi{10.1016/j.spa.2019.04.013}
\end{barticle}
\endbibitem

\bibitem[\protect\citeauthoryear{Rio}{1993}]{Rio1993}
\begin{barticle}[author]
\bauthor{\bsnm{Rio},~\bfnm{Emmanuel}\binits{E.}}
(\byear{1993}).
\btitle{Covariance inequalities for strongly mixing processes}.
\bjournal{Annales de l'Institut Henri Poincar\'{e} Probabilit\'{e}s et
  Statistiques}
\bvolume{29}
\bpages{587--597}.
\end{barticle}
\endbibitem

\bibitem[\protect\citeauthoryear{Wong and Zakai}{1965}]{WongZakai1965}
\begin{barticle}[author]
\bauthor{\bsnm{Wong},~\bfnm{Eugene}\binits{E.}} \AND
  \bauthor{\bsnm{Zakai},~\bfnm{Moshe}\binits{M.}}
(\byear{1965}).
\btitle{On the Relation between Ordinary and Stochastic Differential
  Equations}.
\bjournal{International Journal of Engineering Science}
\bvolume{3}
\bpages{213--229}.
\bdoi{10.1016/0020-7225(65)90045-5}
\end{barticle}
\endbibitem

\end{thebibliography}

\end{document}